\documentclass[11pt]{amsart}

\usepackage{cite}

\usepackage[left=1in,right=1in,top=1in, bottom=1in]{geometry}

\usepackage{amsmath, amsthm, amssymb, amsfonts}
\usepackage{bbm}
\usepackage{mathrsfs}
\usepackage{mathtools}

\usepackage{empheq}
\usepackage{graphicx}
\usepackage{enumitem}
\usepackage{setspace}
\usepackage{float}
\usepackage{hyperref}
\usepackage[utf8]{inputenc}
\usepackage[english]{babel}
\usepackage{framed}
\usepackage{ragged2e}
\usepackage{algorithm}
\usepackage[noend]{algpseudocode}
\usepackage[dvipsnames]{xcolor}
\usepackage{tcolorbox}
\usepackage{chapterbib}
\usepackage{soul} 
\usepackage{tikz}
\usetikzlibrary{calc}

\sethlcolor{red}

\colorlet{LightGray}{White!90!Periwinkle}
\colorlet{LightOrange}{Orange!15}
\colorlet{LightGreen}{Green!15}

\newcommand{\phiitem}[2]{%
  \phantomsection            
  \item[#2]%
  \def\@currentlabel{#2}
  \label{#1}%
}

\newcommand\numberthis{\addtocounter{equation}{1}\tag{\theequation}}

\newcommand{\2}[1]{\ind_{#1}}

\DeclareMathOperator{\E}{\mathbb E}
\renewcommand{\P}{\mathbb P}
\renewcommand{\tilde}{\widetilde}
\newcommand{\supp}{\text{supp}}
\renewcommand{\epsilon}{\varepsilon}

\renewcommand{\forall}{\text{for all }}

\newcommand{\rb}{{\mathrm b}}
\newcommand{\rr}{{\mathrm r}}

\newcommand{\ud}{{\mathrm d}}

\newcommand{\ind}{\mathbbm{1}}

\newcommand{\R}{{\mathbb R}}

\newcommand{\N}{{\mathbb N}}

\newcommand{\Q}{{\mathbb Q}}
\newcommand{\Sph}{{\mathbb{S}^{d-1}}}

\newcommand{\cA}{{\mathcal A}}
\newcommand{\cB}{{\mathcal B}}

\newcommand{\cE}{{\mathcal E}}

\newcommand{\cH}{{\mathcal H}}

\newcommand{\cK}{{\mathcal K}}
\newcommand{\cL}{{\mathcal L}}

\newcommand{\cN}{{\mathcal N}}
\newcommand{\cP}{{\mathcal P}}
\newcommand{\cQ}{{\mathcal Q}}

\newcommand{\cS}{{\mathcal S}}
\newcommand{\cT}{{\mathcal T}}
\newcommand{\cW}{{\mathcal W}}

\newcommand{\cV}{{\mathcal V}}

\newcommand{\Ltr}{{\mathcal L}_{\mathrm{tr}}}
\newcommand{\Lv}{{\mathcal L}_{\mathrm{v}}}

\newtheorem{theorem}{Theorem}[section]

\newtheorem{lemma}[theorem]{Lemma}
\newtheorem{corollary}[theorem]{Corollary}

\theoremstyle{definition}

\theoremstyle{remark}
\newtheorem{remark}[theorem]{Remark}

\theoremstyle{plain}

\title[Hypocoercivity of Tempered Bouncy Particle Samplers]{Hypocoercivity of Tempered Bouncy Particle Samplers for Heavy-Tailed Targets}
\author{Aleksandar Mijatovi\'c}
\address{Department of Statistics, University of Warwick, UK}
\email{a.mijatovic@warwick.ac.uk}
\author{Vinayak Niraj}
\address{Department of Statistics, University of Warwick, UK}
\email{vinayak.niraj@warwick.ac.uk}
\date{\today}



\begin{document}
\begin{abstract}
We introduce the Tempered Bouncy Particle Sampler~\eqref{eq:TBP}, a state-dependently tempered generalisation of the Bouncy Particle Sampler. Adapting the hypocoercivity framework of Dolbeault, Mouhot and Schmeiser~\cite{dolbeault_hypocoercivity_2010}, we obtain explicit, non-asymptotic exponential convergence bounds for a broad class of heavy-tailed targets, including polynomial tails, without requiring log-concavity or radial symmetry. A weighted Poincar\'e inequality for the target is the central assumption: its constant is the spectral gap of a tempered Langevin diffusion with the same invariant distribution and controls the convergence rate.  We give admissible tempering choices for polynomial and stretched-exponential tails and derive mixing-time bounds polynomial in the dimension.
\end{abstract}

\maketitle

\section{Introduction}\label{sec:introduction}
The Bouncy Particle Sampler (BPS) is a nonreversible piecewise deterministic Markov process designed to sample from a target distribution known only up to its normalizing constant. The process augments position with velocity and follows linear trajectories interrupted by random bounce and refreshment events, yielding an event-driven, rejection-free algorithm that can exploit gradients and factorizations of the target in high-dimensional Bayesian models~\cite{bouchard-cote_bouncy_2018}. Its convergence can nevertheless be intrinsically slow for heavy-tailed targets. If the velocity is of bounded magnitude,  finite propagation speed implies that the total variation distance at time $t$ is at least one half of the target mass outside the region reachable from the initial position by time $t$; polynomial tails thus prevent the total variation distance from decaying faster than polynomially~\cite{roberts_polynomial_2023}. This obstruction is reflected in existing theory: exponential ergodicity of the unmodified BPS is established under substantially lighter tail conditions, whereas heavy-tailed targets exhibit subgeometric and, in representative cases, sharp polynomial convergence~\cite{deligiannidis_exponential_2019,christophe_andrieu_subgeometric_2021,roberts_polynomial_2023}. This finite-speed limitation motivates the tempered dynamics, which we now introduce. Short YouTube presentations~\cite{YouTube_talk} discuss our \href{https://youtu.be/HhlTlsyhhzY?si=S9ZTErX9NkR8Ctwj}{main results} and their \href{https://youtu.be/tX93nqGz1WU?si=as1Pfu75ruAXv4A5}{proofs}, outlined in Subsections~\ref{sec:result} and~\ref{subsec:main_contributions}, respectively.

\subsection{Model and main results} \label{sec:result}
We now define a tempered variant of the BPS that targets a stationary distribution $\mu$ on $\R^d$ with a density proportional to $\exp(-U)$, i.e., $\mu(\ud x) = \exp(-U(x))\frac{\ud x}{Z_U}$ for $Z_U\in(0,\infty)$, the (typically unknown) normalisation constant.
Let $\sigma: \R^d \rightarrow [1, \infty)$ be a smooth, Lipschitz function. 
Define the following joint distribution on the position-velocity space $\R^d \times \R^d$ as follows: for $(x,v)\in\R^d \times \R^d$ set 
\begin{equation}
\label{eq:mu_sigma_def}
    \mu_\sigma(\ud x , \ud v) \coloneqq 
    \mu(\ud x) \kappa_x(\ud v),\quad\text{where}\quad \kappa_x (\ud v ) \coloneqq \frac{\exp\left(-|v|^2/(2\sigma^2(x))\right)}{(2\pi \sigma^2(x))^{d/2}}  \ud v
\end{equation} 
is the conditional, position-dependent Gaussian velocity distribution $\cN(0, \sigma(X_s) I_d)$. 
The Hamiltonian $H : \R^d \times \R^d \rightarrow \R$ of the measure $\mu_\sigma$ in~\eqref{eq:mu_sigma_def} equals 
 $H(x, v) \coloneqq   U (x) + |v|^2/(2 \sigma^2(x)) + d\log \sigma(x)$ (i.e., $\mu_\sigma (\ud x, \ud v ) = \exp(-H(x, v)) \, \ud x\ud v / (Z_U (2 \pi)^{d/2})$).
   Its 
    spatial gradient $\nabla_x H : \R^d \times \R^d \rightarrow \R^d$ equals
     \begin{equation}
         \nabla_x H(x, v)  =  \nabla_x U (x) + \frac{1}{{\sigma(x)} }\left[d - \frac{|v|^2}{\sigma^2(x)}  \right]   \nabla_x \sigma(x) . \label{eq:nablaH}
     \end{equation}
The process $(X, V) = (X_t, V_t)_{t \in \R_+}$ (where $\R_+ \coloneqq [0, \infty)$) satisfying the following stochastic differential equation is the \textit{Tempered Bouncy Particle Sampler}: for any $t \in \R_+$ and $(X_0,V_0)\in \R^d \times \R^d$ it holds
\begin{equation}
    \begin{aligned}
        X_t = X_0 &+ \int_0^t V_s \, \ud s,\\
        V_t =V_0  &+  \int_0^t (R_{\nabla_x H  (X_s, V_s-)} V_{s-} - V_{s-}) \, \ud N_s^\mathrm{b} + \int_0^t(\tilde V_s - V_{s-}) \,\ud N_s^\mathrm{r},
    \end{aligned}
     \label{eq:TBP} \tag{T-BPS}
\end{equation}
where $N^\mathrm{r} \coloneqq (N_t^\mathrm{r})_{t\in\R_+}$ is a Poisson process with constant intensity $\lambda_\rr > 0$. At each jump time $s>0$ of $N^\mathrm{r}$,  $\tilde V_s$ is sampled from the law $\kappa_{X_s}$ on $\R^d$, independently of the path $(X_t,V_t)_{t\in[0,s)}$ prior to the jump time $s$:
$\tilde V_s \sim \cN(0, \sigma(X_s) I_d)$.
For any  $u \in \R^d \setminus \{0\}$ and $b \in \R^d$, the map $R_u b \coloneqq b - 2|u|^{-2}\langle b , u\rangle {u}$ is the reflection of $b$ in the $u$-direction (denote by $\langle \cdot, \cdot \rangle$ and $|\cdot|$ the standard inner product and norm in $\R^d$, respectively, and set $R_0 b \coloneqq b$). 
The counting process 
$N^\mathrm{b}\coloneqq (N_t^\mathrm{b})_{t\in\R_+}$ 
in $\N\coloneqq \{0,1,\ldots\}$ has positive jumps of size one arriving at a stochastic intensity given by 
\begin{equation}\lambda_\rb{(X_t, V_t)} \coloneqq  \langle V_t, \nabla_x H (X_t, V_t)\rangle_+ 
\label{eq:deflambdab}
\end{equation}
(throughout the paper, for $a\in\R$,  we denote $a_+\coloneqq \max\{0,a\}$).
Note that the jumps at time $s > 0$ of the counting processes $N^\rb$ and $N^\rr$ are independent of each other and of the path $(X_t, V_t)_{t \in [0, s)}$. 

A basic but crucial property of the dynamics $(X,V)$ in~\eqref{eq:TBP} is that the flow between velocity jumps is linear,
$X_{t+s}=X_t+sV_t$
whenever
$N_{t+s}^{\mathrm b}=N_t^{\mathrm b}$ \&
$N_{t+s}^{\mathrm r}=N_t^{\mathrm r}$,
and hence globally defined. Since the jump rates $\lambda_{\rb}$ and $\lambda_{\rr}$ are locally bounded,~\cite[Construction~1]{durmus_piecewise_2021} implies that $(X, V)$ exists for all $t\in\R_+$ as a non-explosive piecewise deterministic Markov process (PDMP).

\subsubsection{The main result}
We start by stating assumptions on the potential $U$ (of the target density $\mu$) and the tempering function $\sigma:\R^d\to[1,\infty)$ featuring in the gradient $\nabla_x H$ in~\eqref{eq:nablaH} above.
\begin{enumerate}[start=0,label={(\bfseries A\arabic*)}]
    \item \label{ass:wpi}The following weighted Poincar\'e inequality holds for $\mu$ with weight $\sigma$ and constant $m > 0$:
    \begin{equation*}
    m\| g - \mu(g)\|_{L^2(\mu)}^2 \leq \int_{\R^d} \sigma^2|\nabla_x g|^2 \,\ud \mu \qquad \text{for all } g \in C^1_b(\R^d). 
    \end{equation*}

    \item \label{ass:sigmabounds}  There exist constants $ M_D, M_H, M_{U, \sigma}\geq 0$ such that
$$|\nabla_x \sigma| \leq M_D,\qquad \sigma|\nabla_x^2 \sigma|_F \leq M_H,\qquad  \sigma|\nabla_x U |  \leq M_{U, \sigma},$$
and assume $\sup_{x \in \R^d}\sigma^2(x)|\nabla^2_x U(x)|_F < \infty$.
    \item \label{ass:curvature} There exists $K \geq 0$ such that $\sigma^2 \nabla^2_x U \succeq - KI_d$, where $I_d$ is the identity matrix in $\R^{d\times d}$.
\end{enumerate}

For matrices $A, B \in \R^{d \times d}$,
the relation $A\succeq B$ means that $A- B $ is positive semi-definite, and the Frobenius matrix norm is defined as $|A|_F \coloneqq \sqrt{\mathrm{Trace}(A^TA)}$.
Note that bounded continuously differentiable functions with bounded derivatives in $C^1_b(\R^d)$ are in the Hilbert space $L^2(\mu)$ of square-integrable functions. 

The weighted Poincar\'e inequality in Assumption~\ref{ass:wpi} is satisfied by a large class of heavy-tailed distributions for choices of $\sigma$ that satisfy the growth conditions in~\ref{ass:sigmabounds} and the weighted negative curvature condition in~\ref{ass:curvature}~\cite{cattiaux_functional_2010, bobkov_weighted_2009}. In particular, Assumptions~\ref{ass:wpi}-\ref{ass:curvature} allow for multi-modal targets with tails that are neither necessarily radially symmetric nor log-concave (cf.~Remark~\ref{rem:tails} below). 
Moreover, we require no integrability of any power of $\sigma$ with respect to $\mu$, yielding non-asymptotic exponential convergence rates for polynomial-tailed targets of any degree (cf. Remark~\ref{rem:sigmaintB}).

Define the following constants based on Assumptions~\ref{ass:wpi}--\ref{ass:curvature}:
\begin{equation}
    \begin{aligned}
     K_0 &\coloneqq 4 d (d+ 2 )\left[1 + (K_1 + 4 (M_{U, \sigma}^2 + 2(d + 12)M_D^2) )/m\right],\\
        K_1 &\coloneqq 3M_{U, \sigma }M_D +(2 + \sqrt d )M_H + \left(2 + 2(\sqrt d + 2)^2\right)M_D^2 + K.
    \end{aligned}
    \label{eq:defconstants}
\end{equation}
Define the  Markov semigroup $(\cP_t)_{t \in \R_+}$ on $C_b(\R^d \times \R^d)$ of~\eqref{eq:TBP}  by $\cP_t f(x,v) \coloneqq \E_{(x, v) } [f(X_t, V_t)]$. Note that $(\cP_t)_{t \in \R_+}$ extends to a unique contraction semigroup on $L^2(\mu_\sigma)$ by Theorem~\ref{thm:Pcontraction} below.

\begin{theorem}~\label{thm:maincnvg}
      Let Assumptions~\ref{ass:wpi}, \ref{ass:sigmabounds}, and~\ref{ass:curvature} hold. Recall the constant $K_0$  in~\eqref{eq:defconstants} and the refresh rate $\lambda_\rr>0$ in~\eqref{eq:TBP}. Then for all $f \in L^2(\mu_\sigma)$, satisfying $\mu_\sigma(f) = 0$, 
      the  following inequality holds for all $t\geq0$,
    \begin{equation}
        \|\cP_t f\|_{L^2(\mu_\sigma)}^2 \leq 3e^{- \nu t} \|f\|^2_{L^2( \mu_\sigma)},  \quad\text{where $1/\nu\coloneqq3(4+(2\sqrt{K_0} + \lambda_\rr/(2\sqrt m) )^2)/\lambda_\rr$.}\label{eq:L2decayresult}
    \end{equation}
  \end{theorem}

  \begin{remark}\label{rem:tails} 
  Note that a single choice of $\sigma$ can be admissible (in the sense of satisfying~\ref{ass:wpi}, \ref{ass:sigmabounds} and~\ref{ass:curvature}) for a wide class of target distributions. In particular, the choice of $\sigma$ depends only on the tail behaviour of the target distribution $\mu$ as shown below, and is invariant under perturbations of the potential $U$ that have finite oscillation. Set $x\mapsto\langle x \rangle \coloneqq \sqrt{1 + |x|^2}$ for $ x \in \R^d$. 
      
      \noindent \textbf{(I) Polynomial distribution.}
      If $ \mu(\ud x) = V(x)^{-d - r}\ud x$ for a  positive, convex function $V$ and  $r\in(0,\infty)$, then the weighted Poincar\'e inequality~\ref{ass:wpi} is satisfied for $\sigma = \langle x \rangle$ and some $m > 0$~\cite[Prop.~3.2]{cattiaux_functional_2010}.
      If $V$ is a polynomial then~\ref{ass:sigmabounds} and~\ref{ass:curvature} are also satisfied for $\sigma = \langle x \rangle$.
      Specifically, for $U(x) = (d + r)\log\langle x \rangle$ with $r > d$,~\ref{ass:wpi} is satisfied with $\sigma = \langle x \rangle$ and $m^{-1} = d + r - 2 $~\cite[Thm.~3.1]{bobkov_weighted_2009}. 

      \noindent \textbf{(II) Sub-exponential distribution.} If $ \mu(\ud x) = e^{-V(x)^p} /Z_p$ with $p \in (0, 1)$ and $V$ positive and convex, then~\ref{ass:wpi} is satisfied by $\sigma(x) = \langle x \rangle^{1-p}$ for some $m > 0$~\cite[Prop.~3.6]{cattiaux_functional_2010}.
      Specifically, for $U(x) = \langle x \rangle^p $ for a fixed $ p \in (0, 1)$,~\ref{ass:wpi} is satisfied for $\sigma(x) = \langle x\rangle^{1- p}$ and $m^{-1}/e = 12 d/p^3 + (d + p)/p^4$ (using ~\cite[Prop.5.7]{cattiaux_functional_2010} and \textbf{(IV)} below); and in that case~\ref{ass:sigmabounds}-\ref{ass:curvature}  also hold.

      \noindent \textbf{(III) Exponential and thinner-tailed distributions (Standard BPS).} For $\sigma = 1$,~\eqref{eq:TBP} aligns with the BPS, and~\ref{ass:wpi} reduces to a (strong) Poincar\'e inequality, satisfied for exponential and thinner-tailed distributions; see~\cite{andrieu_hypocoercivity_2021,piecewise_hypocoercivity_wang_2022, lu2026windowedthinningquerycomplexity} for non-asymptotic exponential rates in this case.

      \noindent \textbf{(IV) Stability of $\sigma$ under finite perturbations of $\mu$.}
      An analogue of the Holley-Stroock Lemma (e.g. \cite[Prop 4.2.7]{bakry_analysis_2014}) for~\ref{ass:wpi} is as follows: if~\ref{ass:wpi} is satisfied by the weight $\sigma$ and probability distribution $\mu_1$ for some $m_1 > 0$, then it is also satisfied by the probability measure $\ud\mu_2  \coloneqq e^{W} \ud\mu_1$, where $W : \R^d \rightarrow \R$ is a smooth potential of finite oscillation $\mathrm{osc}_{\mu_1}(W) \coloneqq \mathrm{ess}\sup_{\mu_1}(W) - \mathrm{ess}\inf_{\mu_1}(W) < \infty$, the same weight $\sigma$ and the weighted Poincar\'e constant $m_2 \coloneqq e^{-\mathrm{osc}_{\mu_1}(W)}m_1$. 
      This demonstrates that the tempering function $\sigma$ depends primarily on the tail behaviour of the target distribution $\mu$. However, there exist functions $W$ with infinite oscillation $\mathrm{osc}_{\mu_1}(W)=\infty$, such that~\ref{ass:wpi} does not  hold for any $\sigma$ with a bounded gradient, see Lemma~\ref{lem:inf_oscillation}(ii) in Appendix~\ref{app:failure_wpi} below. The example in Lemma~\ref{lem:inf_oscillation}(iii) also shows that $\mathrm{osc}_{\mu_1}(W)=\infty$ does not guarantee the failure of~\ref{ass:wpi}.
  \end{remark}

\begin{remark}[The role of tempering in the dynamics of~\eqref{eq:TBP}] \label{rem:temperingheuristics} For the standard BPS, corresponding to~\eqref{eq:TBP} with $\sigma\equiv1$, the bounce rate is \[ \lambda_\rb(x,v)=\langle v,\nabla U(x)\rangle_+. \] For heavy-tailed targets, $|\nabla U(x)|\to0$ as $|x|\to\infty$, while the refresh distribution $\cN(0,I_d)$ produces velocities whose typical magnitude is independent of the position. Consequently, the bounce rate becomes small in the tails, leading to long excursions with few bounces and, over successive refresh cycles, random-walk-like behaviour. This mechanism is used in~\cite[Sec.~3.3]{deligiannidis_exponential_2019} to explain heuristically why the standard BPS need not be exponentially ergodic for heavy-tailed targets. Tempering counteracts this effect. At each refresh time $t_r$, the velocity is sampled from $\cN(0,\sigma^2(X_{t_r})I_d)$, so its typical magnitude is of order $\sigma(X_{t_r})$, up to a dimension-dependent factor. Thus, focusing on the tail order of the contribution from $\nabla U$, the typical scale of the bounce rate is heuristically \begin{equation} \lambda_\rb(X_t,V_t) \approx |V_t|\,|\nabla_xH(X_t,V_t)| \approx |V_t|\,|\nabla U(X_t)| \approx \sigma(X_t)|\nabla U(X_t)|. \label{eq:cons_heuristic} \end{equation} An appropriate choice of $\sigma$ can therefore prevent the bounce rate from vanishing when $X_t$ lies in the tails of $\mu$, removing the mechanism responsible for the slow tail behaviour of the standard BPS and enabling exponential ergodicity under the assumptions of our main result. In particular, the tempering functions used for the polynomial and sub-exponential targets in Corollary~\ref{cor:mixingtimeB} follow the heuristic \[ \sigma(x)\asymp 1+\frac{1}{|\nabla U(x)|}, \qquad |x|\to\infty, \] suggested by~\eqref{eq:cons_heuristic}. \end{remark}

  \subsubsection{Mixing Time Estimate}\label{sec:mixtime}
  Theorem~\ref{thm:maincnvg} can be translated into a mixing time estimate by using the $\chi^2$-divergence as a distance on the set of probability measures on $\R^d \times \R^d$. For the definition of the $\chi^2$-divergence, the notation $ O$, and the proof of the following corollary, see Appendix~\ref{sec:appendixchi}. 
  For $\epsilon > 0$ and an initial distribution $\mu_0$ on $\R^d \times \R^d$, define the mixing time $$t_\mathrm{mix}(\epsilon, \mu_0) \coloneqq \inf \{ t \in \R_+ : \chi^2(\mu_0\cP_s\|\mu_\sigma) \leq \epsilon   \ \forall s \in [t, \infty) \}.$$
  \begin{corollary}\label{cor:mixingtime} Let~\ref{ass:wpi}, \ref{ass:sigmabounds} and~\ref{ass:curvature} hold. Let $\nu$ be as in Theorem~\ref{thm:maincnvg}. Let $\mu_0$ be a probability measure on $\R^d \times \R^d$. Then for $\epsilon > 0$, the mixing time is bounded above  $$t_\mathrm{mix}(\epsilon, \mu_0) \leq  \nu^{-1}\left[\log(3\chi^2(\mu_0 \| \mu_\sigma)/\epsilon)\right]_+.$$ 
  \end{corollary}
  By an application of Pinker's inequality (e.g.~\cite[Lem. 1.5.1]{Chewi26Book}), the above lemma also provides us with mixing time bounds in the total variation distance.
   The following is a direct consequence of the previous corollary (and Remark~\ref{rem:tails}\textbf{(I)}-\textbf{(II)}). We use $\cB(\R^d)$ and $\cB(\R^d \times \R^d)$ to denote the Borel sigma algebras on $\R^d$ and $\R^d \times \R^d$, respectively.
    \begin{corollary}[Scaling of mixing times under a feasible start]\label{cor:mixingtimeB}
        Let the families $\{(\mu^{(d)},\cB(\R^d))\vert d \in \N\setminus \{0\}\}$ and 
        $\{(\bar \mu_\sigma^{(d)},\cB(\R^d\times\R^d))\vert d \in \N\setminus \{0\}\}$ be sequences of probability measures. Also, let $\{\sigma_d : \R^d \rightarrow [1, \infty)\vert d \in \N\setminus \{0\}\}$ be a sequence of smooth maps. For $d \in \N\setminus \{0\}$,
        define the joint target distribution 
        on $\R^d \times \R^d$ by
        $$\mu_\sigma^{(d)} (\ud x, \ud v) \coloneqq  {(2\pi \sigma_d^2(x))^{-d/2}} \exp(-|v|^2/(2\sigma_d^2(x))) \,\ud v\, \ud\mu^{(d)}( x ).$$
        (i) Assume that $\log \chi^2(\bar \mu_\sigma^{(d)} \| \mu_\sigma^{(d)}) =  O(d)$. \\
        (ii) Assume that $\mu^{(d)}$ and $\sigma_d$ satisfy~\ref{ass:wpi},~\ref{ass:sigmabounds} and~\ref{ass:curvature} with $m^{-1} = O(d^{l})$ (for some $l \in \R_+$), $M_D = O(1)$, $M_H = O(d^{1/2})$, $M_{U, \sigma} = O(d)$ and $K = O(d)$.\\
        Then we have that $t_\mathrm{mix}(1, \bar \mu_\sigma^{(d)}) =  O(d^{l+ 5})$.
        
        In particular, for $\mu^{(d)}(\ud x) \propto \langle x \rangle^{-d - r_d} \ud x$ for $r_d > d$ or $\mu^{(d)}(\ud x) \propto \exp( -\langle x \rangle^{p}) \,\ud x$ for $p \in (0, 1)$ with the tempering $\sigma_d = \langle x\rangle$ or $\langle x \rangle^{1 - p}$ respectively, assumption (ii) is satisfied for $l = 1$. Then, for initial distributions $\{\bar\mu_\sigma^{(d)}\}_{d \in \N\setminus \{0\}}$ satisfying assumption (i), we have that $t_\mathrm{mix}(1, \bar\mu_\sigma^{(d)}) =  O(d^{6})$. 
    \end{corollary}

    \begin{remark}
   The assumption
$\log \chi^2\bigl(\bar\mu_\sigma^{(d)} \,\Vert\, \mu_\sigma^{(d)}\bigr)=O(d)$
is commonly referred to as a \textit{feasible start} condition~\cite{andrieu_weak_2026}. It reflects initialisation from a distribution that is easy to sample from and whose density ratio relative to the target is suitably controlled, for example by a bound growing at most exponentially with \(d\); see~\cite{andrieu_weak_2026,dwivediLogConcavesampling}. This scaling arises naturally for product measures. Indeed, if \(\pi_0\) and \(\pi_1\) are probability measures on \(\mathbb{R}\) such that
$0<\chi^2(\pi_0\Vert\pi_1)<\infty$,
then
$\chi^2\bigl(\pi_0^{\otimes d}\Vert\pi_1^{\otimes d}\bigr)
=
\bigl(1+\chi^2(\pi_0\Vert\pi_1)\bigr)^d-1$.
Consequently,
$\log \chi^2\bigl(\pi_0^{\otimes d}\Vert\pi_1^{\otimes d}\bigr)
= d\log\bigl(1+\chi^2(\pi_0\Vert\pi_1)\bigr)+O(1)
= O(d)$. 
\end{remark}

\begin{remark}[Tails of feasible initial distributions]\label{rem:feasibletails}
Given a sequence of initial distributions $\bar \mu^{(d)}$ on $\R^d$, $d \in \N \setminus \{0\}$, that satisfies the assumption $\log \chi^2(\bar\mu^{(d)}\| \mu^{(d)}) = O(d)$ (corresponding to a feasible start with respect to the position distribution on $\R^d$), we can define a corresponding initial distribution in the position-velocity space by 
$\bar\mu_\sigma^{(d)}(\ud x, \ud v) \coloneqq  {(2\pi \sigma_d^2(x))^{-d/2}} \exp(-|v|^2/(2\sigma_d^2(x))) \,\ud v\, \ud\bar \mu^{(d)}( x )$ to get a sequence of joint initial distributions which satisfies the assumption $\log \chi^2\bigl(\bar\mu_\sigma^{(d)} \,\Vert\, \mu_\sigma^{(d)}\bigr)=\log \chi^2\bigl(\bar\mu^{(d)} \,\Vert\, \mu^{(d)}\bigr)=O(d)$ required in Corollary~\ref{cor:mixingtime}. (The first equality follows from the definition of $\chi^2$-divergence in Appendix~\ref{sec:appendixchi}.) 

Consider the following specific example.
 Fix $r \in (0, \infty)$. For the target $\mu^{(d)}(\ud x ) \propto \langle x \rangle^{- d - r}\ud x $ (recall $\langle x \rangle=\sqrt{1+|x|^2}$ from Remark~\ref{rem:tails}), an initial distribution $\bar \mu^{(d)}(\ud x )$ with density proportional to either $x\mapsto \exp(- \langle x \rangle^{q})\, \ud x $, for $q\in(0, \infty)$, or  $x\mapsto  \langle x\rangle^{-d -\rho}\, \ud x$,  for  $r/2<\rho<\infty$, satisfies the condition $\log \chi^2\left(\bar \mu^{(d)} \| \mu^{(d)}\right) = O(d)$.
Indeed, set $Z_{d, z} \coloneqq \int_{\R_+} ( 1 + s^2 )^{\frac{- d - z}{2}} s ^{d- 1} \, \ud s$ and note $Z_{d, z}=O(1)$ for any $z \in (0, \infty)$ (since $s^{d-1}\leq (1+s^2)^{\frac{d-1}{2}}$ for all $s\geq0$). Let $\Omega_{d - 1}$ be the surface area of the unit sphere $\Sph$ in $\R^d$.
For $\bar \mu^{(d)} (\ud x ) = (\Omega_{d-1}Z_{d, \rho})^{-1} \langle x\rangle^{-d -\rho}\, \ud x$,  we have 
$$\chi^2\left(\bar \mu^{(d)} \| \mu^{(d)}\right) + 1 = Z_{d, \rho}^{-2} Z_{d, r} Z_{d, 2\rho- r} = O(1).$$ Similarly, consider the family of initial distributions $\bar \mu^{(d)} (\ud x ) = (\Omega_{d-1}Y_{d,q})^{-1} e^{- \langle x \rangle^{q}}\, \ud x$, where $Y_{d, q} \coloneqq \int_{\R_+} e^{- (1+ s^2)^{q/2}} s^{d-1} \, \ud s $ for $q \in (0, \infty)$ and recall the Gamma function $\Gamma(z) \coloneqq \int_{\R^+} s^{z-1} e^{-s} \, \ud s $ for $z \in \R_+$. Also, set $\bar Y _{d, q, r} \coloneqq \int_{\R_+} e^{- 2(1 + s^2)^{q/2}} (1+ s^2)^{\frac{d + r}{2}} s^{d-1}\, \ud s$ for $ q, r \in (0, \infty)$. We have that $\chi^2(\bar \mu^{(d)} \| \mu^{(d)}) + 1 = Z_{d, r} Y_{d, q}^{-2} \bar Y_{d, q , r}$. Hence, using $\log Z_{d, r} = O(1)$,  $\log(Y_{d, q}) = \log\Gamma(d/q) + O(d)$, and $\log( \bar Y_{d, q , r} ) = \log\Gamma((2d+r)/q) + O(d)$. Since $\log \Gamma (z) = z \log z + O(z)$ as $z\to\infty$, we get that $\log \chi^2(\bar \mu^{(d)} \| \mu^{(d)}) = O(d)$ for $q\in(0,\infty)$.

Hence, under the feasible-start assumption, the initial distribution in Corollary~\ref{cor:mixingtimeB} may have lighter tails than the target. In particular, the standard Gaussian distribution on $\R^d$, as well as any super-Gaussian distribution with exponent $q>2$, is an admissible initial distribution for the mixing-time estimate. Recent work~\cite[Thm.~2]{monmarché2026entropicconvergencepiecewisedeterministic} shows that the standard BPS converges slowly when initialised from a heavy-tailed distribution. In contrast, Corollary~\ref{cor:mixingtimeB} shows that this obstruction does not arise for~\eqref{eq:TBP} when the initial distribution has polynomial tails whose order is at least half that of the target.
\end{remark}

\subsection{Main contributions of the paper and comparison with the literature}
\label{subsec:main_contributions}

\subsubsection{Main contributions}
The following summarise the main contributions of this paper,  including crucial results and observations that lead to the proof of Theorem~\ref{thm:maincnvg}.

\smallskip
\noindent\textbf{(I) Non-asymptotic convergence rates and mixing times for heavy-tailed targets.}
In Theorem~\ref{thm:maincnvg}, we use hypocoercivity to establish non-asymptotic exponential convergence for the class of PDMPs in~\eqref{eq:TBP} with heavy-tailed invariant distributions. In Section~\ref{sec:hypocoercivity}, we adapt the hypocoercivity framework of~\cite{dolbeault_hypocoercivity_2010}, together with its recently introduced optimal modification~\cite{fan_sharp_2026, lu2026windowedthinningquerycomplexity}, to the tempered bouncy particle sampler~\eqref{eq:TBP}. Our argument uses a modified \(L^2\)-Lyapunov functional involving the resolvent of the generator of the overdamped tempered Langevin dynamics~\eqref{eq:templangevin}, introduced in Section~\ref{sec:defgen}. The key observation is that the generator of~\eqref{eq:TBP} is a \textit{second-order lift}~\cite[Def.~1]{eberle_non-reversible_2026} of this Langevin generator. The analysis is primarily functional-analytic and relies on the weighted Poincar\'e inequality in Assumption~\ref{ass:wpi}, whose weight determines the appropriate choice of the tempering function \(\sigma\). The resulting exponential \(L^2\)-decay in Theorem~\ref{thm:maincnvg} yields mixing-time bounds that scale polynomially with the dimension \(d\) for sub-exponential and polynomially-tailed targets, including from non-warm starts
(recall that the mixing-time bounds for the Random Walk Metropolis are exponential in dimension $d$ for heavy-tailed targets with non-warm start~\cite{andrieu_weak_2026}).

\noindent\textbf{(II) Non-explosion and robustness.}
By construction, the process in~\eqref{eq:TBP} is non-explosive without additional assumptions; see Lemma~\ref{lem:explosiveness}. This contrasts with PDMPs that temper the velocity according to the current position, whose deterministic flows may explode in finite time~\cite{g_vasdekis_speed_2023}. Moreover, the weighted Poincar\'e inequality in Assumption~\ref{ass:wpi} requires the tempering function \(\sigma\) to depend only on the tail behaviour of the potential \(U\); see Remark~\ref{rem:tails}.

\noindent\textbf{(III) Technical contributions.}
Theorem~\ref{thm:stationaritycore}, which identifies a core for the generator of the Markov semigroup associated with~\eqref{eq:TBP}, provides the link between the stochastic process $(X,V)$ and the functional-analytic hypocoercivity result in Theorem~\ref{thm:coercivity}. Lemma~\ref{lem:TLD2} establishes analogous generator properties for the Markov semigroup of the overdamped tempered Langevin diffusion~\eqref{eq:templangevin}. These properties appear to be new and are essential to the proof of Theorem~\ref{thm:coercivity}.

\noindent\textbf{(IV) Weighted  Bochner inequality.}
Bochner's identity (e.g.,~\cite[Ch.~3]{bakry_analysis_2014}) is a fundamental tool in hypocoercivity~\cite{cao_explicit_2023,andrieu_hypocoercivity_2021,dolbeault_hypocoercivity_2010,eberle_convergence_2025,fan_sharp_2026}. We extend it to the weighted setting associated with the tempered Langevin dynamics~\eqref{eq:templangevin}: Lemma~\ref{lem:weightedbochner} establishes a weighted Bochner identity incorporating the tempering function \(\sigma\) and the tempered Langevin generator, from which we derive the weighted Bochner inequality in Lemma~\ref{lem:bochnerineq}. We believe that this framework lays the groundwork for obtaining non-asymptotic exponential convergence rates for a new class of hypoelliptic ergodic processes; see Section~\ref{sec:conclusions}(B).
\smallskip
 
 \subsubsection{Comparison with previous work} We now discuss related work, focusing mainly on processes and algorithms for sampling from heavy-tailed distributions.
 \smallskip

\noindent\textit{(i) Classical algorithms for heavy-tailed targets:}
For heavy-tailed targets, the standard BPS has been shown to converge only subgeometrically~\cite{christophe_andrieu_subgeometric_2021, roberts_polynomial_2023}, while the mixing times of classical algorithms such as Random Walk Metropolis (RWM) can scale as \(\exp(O(d))\) from non-warm starts~\cite{andrieu_weak_2026}. As noted in \textbf{(I)}, the convergence rates and mixing-time bounds for~\eqref{eq:TBP} therefore substantially improve upon those available for classical algorithms.
\smallskip

\noindent\textit{(ii) Tempered variants of classical algorithms:}
The Speed Up Zig-Zag sampler (SUZZ)~\cite{g_vasdekis_speed_2023} is a PDMP with exponential convergence to heavy-tailed targets, but relies on a speeding-up ODE flow. In contrast,~\eqref{eq:TBP} has piecewise-constant velocity and therefore avoids solving such an ODE, simplifying a potential implementation of the algorithm, while remaining non-explosive, even in regimes where SUZZ and~\eqref{eq:templangevin} are explosive; see \textbf{(II)}. This is achieved by stochastically accelerating the velocity at the jump times of the Poisson process \((N_t^\rr)_{t\in\R_+}\), when it is resampled from the position-dependent refresh distribution \(\kappa_x\). The use of \(\kappa_x\) is reminiscent of Random Walk Metropolis with position-dependent covariance~\cite{livingstone_geometric_2021}, although~\eqref{eq:TBP} is gradient-based.
\smallskip

\noindent\textit{(iii) Space transformations for thinning or eliminating tails:}
In~\cite{leif_t_johnson_variable_2012}, isotropic transformations map the target to a distribution with superexponential tails, for which RWM is geometrically ergodic. This approach is applied in~\cite{deligiannidis_exponential_2019} to enable BPS to sample from a transformation of a heavy-tailed distribution. Such methods crucially assume that the tails behave similarly in all directions, whereas our tempering function \(\sigma\) need not be radially symmetric. Moreover, unlike the approaches in~\cite{leif_t_johnson_variable_2012,deligiannidis_exponential_2019,g_vasdekis_speed_2023}, we establish non-asymptotic exponential convergence rates, made possible by the \textit{second-order lift} observation discussed in \textbf{(I)}. A better-behaved transformation and convergence guarantees for ULA applied to transformed isotropic heavy-tailed targets are developed in~\cite{he2022transformedULA}. More recently, an elimination-of-the-tails method~\cite{bresar2026diffeomorphicmarkovchainmonte} introduced a class of state-space transformations that map heavy-tailed distributions to distributions with compact, convex support, ensuring uniform ergodicity and typically providing polynomial in $d$ mixing-time estimates.
\smallskip

\subsection{Structure of the remainder of the paper}
Section~\ref{sec:hypocoercivity} develops our hypocoercivity framework based on a modified \(L^2\)-Lyapunov functional, outlines the proof strategy, and states the technical results needed for Theorem~\ref{thm:maincnvg}. Section~\ref{sec:invariance} shows that the semigroup \(\cP\) associated with~\eqref{eq:TBP} is a strongly continuous contraction on \(L^2(\mu_\sigma)\), defines its generator on this space, and identifies a core. Section~\ref{sec:proofs} proves the results stated in Section~\ref{sec:hypocoercivity}. Within it, Subsection~\ref{subsec:Sturcture_of_proof_of_main_thm} describes the structure of the argument and the dependencies among the technical results; see Figures~\ref{fig:lemmadeps} and~\ref{fig:auxdeps}. Subsection~\ref{sec:proofdecay} then combines the core from Section~\ref{sec:invariance} with the coercivity estimate proved in Subsection~\ref{sec:proofcoercivity} to complete the proof of Theorem~\ref{thm:maincnvg}. Finally, Section~\ref{sec:conclusions} presents open problems and directions for future research.

\section{Hypocoercivity for heavy tails via the modified \texorpdfstring{$L^2$}{L2}-framework} \label{sec:hypocoercivity}

We start by describing the strategy of the proof of Theorem~\ref{thm:maincnvg}, inspired by the hypocoercivity structure introduced in~\cite{dolbeault_hypocoercivity_2010}; see also Figure~\ref{fig:lemmadeps} in Section~\ref{sec:proofs} below.\\
\textbf{(a)} The hypocoercivity argument begins with the construction of the Lyapunov functional \(\cV_\delta\), defined in~\eqref{eq:defcV}, which is a perturbation of the squared \(L^2(\mu_\sigma)\)-norm. It involves the operator \(\cA\), defined in~\eqref{eq:defcA}, through the resolvent of the \(L^2(\mu)\)-generator \((\bar{\cL}_\sigma,\operatorname{Dom}(\bar{\cL}_\sigma))\) of the tempered Langevin dynamics~\eqref{eq:templangevin}. Lemma~\ref{lem:TLD2} establishes the properties of this generator and its resolvent required in the subsequent steps.\\
\textbf{(b)} We introduce the operator
$\cL=\Ltr+\lambda_\rr\Lv$,
$\Lv=\Pi-I$,
and show that it coincides with the \(L^2(\mu_\sigma)\)-generator $(\bar{\cL},\operatorname{Dom}(\bar{\cL}))$ of the semigroup  \(\cP\) of~\eqref{eq:TBP}, defined in Section~\ref{sec:invariance}, on sufficiently regular functions. Here, \(\Ltr\) describes the deterministic transport and bounce mechanisms, while \(\Pi\) averages over the state-dependent velocity distribution \(\kappa_x\). Lemma~\ref{lem:lifts} records the algebraic identities among \(\Ltr\), \(\Ltr^\star\), \(\Lv\), and \(\Pi\). In particular, \(\Ltr^\star\) is the \(L^2(\mu_\sigma)\)-adjoint of \(\Ltr\) on an appropriate subspace, and
$\Pi\Ltr^\star\Ltr\Pi=-\cL_\sigma\Pi$,
so that averaging the second-order transport--bounce operator over the velocity recovers the tempered Langevin generator $\cL_\sigma$ of diffusion~\eqref{eq:templangevin}.\\
\textbf{(c)} Using identities in Lemma~\ref{lem:lifts}, we establish the two operator bounds that control the perturbation terms in the Lyapunov functional. In Lemma~\ref{lem:Afbound} we show that $\cA$ and $\Ltr \cA$ extend to bounded operators on $L^2(\mu_\sigma)$, and in Lemma~\ref{lem:3.2} we bound $\|\Ltr^\star \Ltr g\|_{L^2(\mu_\sigma)}$ in terms of $\|\bar \cL_\sigma g\|_{L^2(\mu)}$, which is where the constant $K_0$ of Theorem~\ref{thm:maincnvg} enters. \\
\textbf{(d)} We combine these bounds into the dissipation functional $\cE_\delta$ associated with the Lyapunov functional $\cV_\delta$. In Theorem~\ref{thm:coercivity} we show that, for the specific choice $\delta = \delta_\ast$, the functional $\cE_{\delta_\ast}$ is \textit{coercive}: $\cE_{\delta_\ast}$ is bounded below by a multiple of $\cV_{\delta_\ast}$ and  $\cV_{\delta_\ast}$ is equivalent to the $L^2(\mu_\sigma)$-norm. A Gr\"onwall argument along the semigroup $(\cP_t)_{t \in \R_+}$ then converts this coercivity into the exponential decay of $\cV_{\delta_\ast}(\cP_t f )$, implying  Theorem~\ref{thm:maincnvg} with the rate $\nu$ determined by $\delta_\ast$. This last step also crucially uses the identification of a core in Theorem~\ref{thm:stationaritycore} of the $L^2(\mu_\sigma)$-generator of the semigroup \(\cP\) of~\eqref{eq:TBP}.

\subsection{Tempered Langevin Dynamics}\label{sec:defgen} 
Consider the stochastic differential equation (SDE)
\begin{equation}
    Y_t = Y_0 + \int_0^t \left( 2 \sigma(Y_s) \nabla_x \sigma(Y_s) - \sigma^2(Y_s) \nabla_x U (Y_s)\right) \ud s + \sqrt 2 \int_0^t \sigma(Y_s) \,\ud B_s , \label{eq:templangevin}\tag{TLD}
\end{equation}
where $(B_t)_{t \in \R_+}$ is a standard $d$-dimensional Brownian motion in $\R^d$.
\begin{lemma}\label{lem:TLD1}
    Let~\ref{ass:sigmabounds} hold. Then \eqref{eq:templangevin} has a unique (non-explosive) solution $Y = (Y_t)_{t \in \R_+}$ which is a strong Markov process. Also, $\mu$ is an invariant distribution for the semigroup of $Y$.
\end{lemma}
Well-posedness in the above lemma follows directly from~\cite[Ch.5 Thm~6]{protter_stochastic_2004} as the coefficients of the SDE~\eqref{eq:templangevin} are globally Lipschitz as a consequence of~\ref{ass:sigmabounds}. The invariance of $\mu$ follows by checking that the density $e^{-U}/Z_U$ solves the stationary Fokker-Planck equation corresponding to the diffusion~\eqref{eq:templangevin} (see, e.g.~\cite[Sec.~2.5]{pavliotis_stochastic_2014}).

For $t \in \R_+$, set $\cT_t g(y)\coloneqq \E_y[g(Y_t)]$ for $g \in C_b(\R^d)$ and $y \in \R^d$. This is uniquely extended to a strongly continuous contraction semigroup $\cT = (\cT_t)_{t \in \R_+}$ on the Banach space $L^2(\mu)$.
Let $(\bar\cL_\sigma, \mathrm{Dom}(\bar\cL_\sigma))$ denote the corresponding generator on $(L^2(\mu), \|\cdot\|_{L^2(\mu)})$. I.e.,
 set $\mathrm{Dom}(\bar \cL_\sigma) \coloneqq \{ g \in L^2(\mu) : \exists h \in L^2(\mu) \text{ such that } \lim_{t \rightarrow 0}\| (\cT_t g - g)/ t - h \|_{L^2(\mu)} = 0 \}$. For $g \in \mathrm{Dom}(\bar \cL_\sigma)$, $\bar\cL_\sigma g \coloneqq \lim_{t \rightarrow 0 }(\cT_t g - g)/t$, in the $L^2(\mu)$ topology.

 Recall the constant $m > 0$ for which Assumption~\ref{ass:wpi} holds.
By~\cite[Ch.1 Prop.~2.1]{ethier_markov_1986}, the resolvent $(m - \bar \cL_{\sigma})^{-1} : L^2(\mu) \rightarrow \mathrm{Dom}(\bar \cL_\sigma)$ is a bounded linear operator.

For $g \in C^2(\R^d)$, define the operator $\cL_\sigma$ by
$\cL_\sigma g \coloneqq - \nabla_x^\star \sigma^2 \nabla_x g $, where   
   $\nabla_x^\star F(x) \coloneqq - \mathrm{div}_x F(x) + \langle \nabla_x U(x) , F(x) \rangle$ and
 $\mathrm{div}_x F (x) \coloneqq \mathrm{Trace}((\nabla_x F)(x))$, for a differentiable vector field $F: \R^d \rightarrow \R^d$ ($\nabla_x F$ is the matrix of gradients of the components  of $F$). Let $C_c^n(\R^d)$ denote the subset of $C^n(\R^d)$ consisting of functions with compact support. We also let $L_0^2(\mu) \coloneqq \{ f \in L^2(\mu) : \mu(f) = 0\}$.
 \begin{remark}
 \label{rem:ibp}
 An integration-by-parts  implies
   that for a differentiable vector field $F : \R^d \rightarrow \R^d$ and a compactly supported function $g \in C_c^1(\R^d)$, we have  $\int_{\R^d}\langle F, \nabla_x g \rangle \ud \mu = \langle \nabla_x^\star F, g \rangle_{L^2(\mu)}$.
 \end{remark}
The proof of Lemma~\ref{lem:TLD2} is in Section~\ref{sec:proofTLD2}. (Core in Lemma~\ref{lem:TLD2} is as defined in~\cite[p.~17]{ethier_markov_1986}.)

\begin{lemma}\label{lem:TLD2} Let~\ref{ass:sigmabounds} hold. Then we have the following.\\
    (I) For $g \in C_b^2(\R^d) \cap \mathrm{Dom}(\bar \cL_\sigma)$, we have
    $\bar \cL_\sigma g = \cL_\sigma g$  $\mu$-a.e.  Moreover, $C_c^2(\R^d) \subseteq \mathrm{Dom}(\bar \cL_\sigma)$ and $C_c^\infty(\R^d)$ is a core of $(\bar \cL_\sigma, \mathrm{Dom}(\bar \cL_\sigma))$.\\
    (II) If $f,g \in C_b^2(\R^d) \cap \mathrm{Dom}(\bar \cL_\sigma)$ then  $$\langle f, (-\bar \cL_\sigma) g \rangle_{L^2(\mu)} = \langle \nabla_x f, \sigma^2 \nabla_x g \rangle_{L^2(\mu)} = \langle (-\bar \cL_\sigma  )f, g \rangle_{L^2(\mu)}.$$ 
    Moreover, if~\ref{ass:wpi} also holds then $\langle \nabla_x g, \sigma^2 \nabla_x g \rangle_{L^2(\mu)} \leq m^{-1}\|\bar \cL_\sigma g\|_{L^2(\mu)}^2 $.\\
    (III) For $g \in C_b^{1}(\R^d)$ we have $(m - \bar \cL_\sigma)^{-1} g \in C_b^{2} (\R^d) \cap \mathrm{Dom}(\bar \cL_\sigma)$. Moreover, if~\ref{ass:wpi} holds, then the inequality $\|(m - \bar\cL_\sigma)^{-1} g \|_{L^2(\mu)} \leq (2m)^{-1} \| g\|_{L^2(\mu)}$ holds for all $g \in L_0^2(\mu) \cap C^1_b(\R^d)$.
\end{lemma}

\begin{remark}
     By Assumption~\ref{ass:wpi}, the process $Y$ converges at an exponential rate $m$ to its stationary distribution $\mu$ in the sense of the decay of the $L^2(\mu)$-norm of its semigroup applied to $\mu$-mean-zero functions.
\end{remark}

\subsection{Construction of the modified \texorpdfstring{$L^2$}{L2} Lyapunov functional and coercivity}
We need the following operators and function spaces to represent the infinitesimal generator of the PDMP given by equation~\eqref{eq:TBP}.
For $f \in C^1(\R^d \times \R^d) $, define 
\begin{equation}
    \begin{aligned}
        \Ltr f(x, v)  &\coloneqq \langle  v,\nabla_x f(x, v) \rangle + \langle \nabla_x H (x, v), v \rangle_+\left[ f\left(x, R_{\nabla_x H(x, v)} v \right) - f(x, v)\right], \\
        \Ltr^\star f (x, v) &\coloneqq -\langle  v,\nabla_x f(x, v) \rangle + \langle -\nabla_x H (x, v), v \rangle_+\left[ f\left(x, R_{\nabla_x H(x, v)} v \right) - f(x, v)\right]. 
    \end{aligned}
    \label{eq:Ltradjoint}
\end{equation}
In Lemma~\ref{lem:lifts} below, we will see that $\Ltr^\star $ is the $L^2(\mu_\sigma)$-adjoint of $\Ltr$ when restricted to an appropriate function subspace.

Define the function spaces
\begin{equation}
    \begin{aligned}
        W_\mathrm{tr} &\coloneqq \{ f \in C^1(\R^d \times \R^d) :\|f\|_{L^2(\mu_\sigma)}+ \|\Ltr f\|_{L^2(\mu_\sigma)} < \infty\},\\
        W_{\mathrm{tr}^\star} &\coloneqq \{ f \in C^1(\R^d \times \R^d) : \|f\|_{L^2(\mu_\sigma)} + \|\Ltr^\star f\|_{L^2(\mu_\sigma)} < \infty\}.
    \end{aligned}
    \label{eq:defauxspaces}
\end{equation}
\begin{remark}\label{rem:sigmaintB}
    We have to work with the subspaces $W_\mathrm{tr}$ and $W_{\mathrm{tr}^\star}$ defined in~\eqref{eq:defauxspaces} as we do not assume that $\sigma \in L^2(\mu)$. Indeed, if $\sigma \in L^2(\mu)$, then the class of bounded functions with bounded continuous derivatives $C^1_b(\R^d \times \R^d )$ is a sufficiently rich class for our purposes, and is contained in $W_\mathrm{tr} \cap W_{\mathrm{tr}^\star}$.
\end{remark}
Define the following function spaces
\begin{equation}
    \begin{aligned}
        C_\Pi &\coloneqq \{ f \in C(\R^d \times \R^d) \, : \,  \int_{ \R^d} |f(x, w )| \, \ud \kappa_x(w) < \infty \text{ for all } x \in \R^d\},\\
        C_\Pi^n &\coloneqq \{ f \in C_\Pi : \Pi f \in C^n(\R^d)\} \quad \text{ for } n \in \N,
    \end{aligned} \label{eq:auxspacesB}
\end{equation}
where $\Pi f$ (in the definition of $C_\Pi^n$) is given in~\eqref{eq:Lv}.
For $f \in  C_\Pi$, set
\begin{equation}
\begin{aligned}
    \Lv f(x, v) \coloneqq \left[\Pi -I\right]f(x, v),\quad\text{where $\Pi f(x, v) \coloneqq \int_{\R^d} f(x, w) \ \kappa_x (\ud w)$.}
\end{aligned}
     \label{eq:Lv}
\end{equation}
(Note that $\Pi f $ does not depend on $v$ and $\Pi(L^2(\mu_\sigma))\subset  L^2(\mu)$.)
For any $f \in C_\Pi \cap C^1(\R^d \times \R^d)$, set
\begin{equation}
    \cL f \coloneqq \Ltr f + \lambda_\rr \Lv f. \label{eq:cL}
\end{equation}

The following result,  proved in Section~\ref{sec:prooflifts}, contains algebraic properties of the operators  in~\eqref{eq:Ltradjoint} and~\eqref{eq:Lv}, essential in the proof of our main result, Theorem~\ref{thm:maincnvg}. 

\begin{lemma}\label{lem:lifts} Recall $\cL_\sigma g = - \nabla_x^\star \sigma^2 \nabla_x g $ for any $g\in C^2(\R^d)$. The following equalities hold:
    \begin{align}
    \Pi\, \Ltr \Pi\, b  = 0 \quad \text{and}\quad \Pi\, \Ltr^\star \Pi\, b  = 0\quad&\text{for any  $b \in C^1_\Pi$;} \label{eq:L1} \tag{L1}\\
    \Pi\, \Ltr^\star \Ltr \Pi f  = -\cL_\sigma \Pi f, \quad &\text{for any } f \in C_\Pi^2;\label{eq:L2}\tag{L2}\\
    \langle \Ltr^\star h, g \rangle_{L^2(\mu_\sigma)} = \langle h, \Ltr g \rangle_{L^2(\mu_\sigma )} , \quad &\text{for any } h \in W_{\mathrm{tr}^\star} \text{ and } g \in W_\mathrm{tr}. \label{eq:L3} \tag{L3}
\end{align}
\end{lemma}

\paragraph*{\textbf{Notational convention:}} for ease of notation, functions of the $x$-variable are trivially identified with functions of $(x, v) \in \R^d \times \R^d$ that are constant in the $v$-variable.

The estimate in Lemma~\ref{lem:3.2}, proved in Section~\ref{sec:proof3.2} using the weighted Bochner inequality (see Lemma~\ref{lem:bochnerineq}), is crucial in the proofs of both Lemma~\ref{lem:Afbound} and Theorem~\ref{thm:coercivity}. Note that the role of Lemma~\ref{lem:3.2} in the proof of Lemma~\ref{lem:Afbound} is of an asymptotic nature (i.e., Lemma~\ref{lem:Afbound} only uses the finiteness of the constant $K_0$  in Lemma~\ref{lem:3.2}, but not its size). In contrast, via Theorem~\ref{thm:coercivity}, Lemma~\ref{lem:3.2} contributes the constant $K_0$ directly to the exponential convergence rate in Theorem~\ref{thm:maincnvg}.

\begin{lemma}\label{lem:3.2} Let~\ref{ass:wpi}, \ref{ass:sigmabounds} and \ref{ass:curvature} hold. Let $K_0$ be as in~\eqref{eq:defconstants}. Then 
    \begin{equation}\|\Ltr^\star \Ltr g \|^2_{L^2(\mu_\sigma)} \leq K_0 \|\bar \cL_\sigma g\|^2_{L^2(\mu)}\qquad\text{for $g \in C^{2}_b(\R^d) \cap \mathrm{Dom}(\bar \cL_\sigma)$.} \label{eq:lemma3.2}\end{equation}
\end{lemma}

Inspired by~\cite{fan_sharp_2026}, for 
$f \in C^2_c(\R^d \times \R^d)$
we define
 \begin{equation}
     \cA f := (m - \bar \cL_\sigma)^{-1} \Pi \, \Ltr^\star f. \label{eq:defcA}
 \end{equation}
 We will be using the operator $\cA$ to construct the modified $L^2$ Lyapunov functional. Hence, the key estimates, which follow from Lemma~\ref{lem:3.2} (see Section~\ref{sec:proofwtr} below), are needed.

\begin{lemma}\label{lem:Afbound} Let~\ref{ass:wpi}, \ref{ass:sigmabounds},~\ref{ass:curvature} hold.
    For $f \in C^2_c(\R^d \times \R^d)$, the following inequalities hold:
    \begin{align}
        \|\cA f\|_{L^2(\mu_\sigma)} &\leq \frac{1}{2\sqrt m } \|f\|_{L^2(\mu_\sigma)}, \label{eq:AfboundA}\\
         \|\Ltr\cA f\|_{L^2(\mu_\sigma)} &\leq  \|f\|_{L^2(\mu_\sigma)}. \label{eq:AfboundB}
    \end{align}
    Thus $\cA$ and $\Ltr \cA$ extend uniquely to bounded operators on $L^2(\mu_\sigma)$ with norms bounded by $1/(2\sqrt{m})$ and $1$, 
    respectively.
\end{lemma}

In the remainder of the paper, we will use $\cA$ to denote the unique bounded $L^2(\mu_\sigma)$-extended operator in Lemma~\ref{lem:Afbound}. 
In the spirit of~\cite{dolbeault_hypocoercivity_2010, fan_sharp_2026}, we define the Lyapunov functional 
\begin{equation}\ \cV_{\delta}(f) \coloneqq \frac{1}{2} \|f\|^2_{L^2( \mu_\sigma )} - \delta \langle \cA f, f \rangle_{L^2(\mu_\sigma)}, \qquad \text{for any } f \in L^2(\mu_\sigma) \text{ and some } \delta \geq 0.\label{eq:defcV}\end{equation}
The  inner product $\langle \cA f, f \rangle_{L^2(\mu_\sigma)}$ is well defined due to Lemma~\ref{lem:Afbound}.
Also, for $\delta \geq 0$ and  any function $f \in C^1_c(\R^d \times \R^d)$, set
\begin{equation}\begin{aligned}
    \cE_\delta(f) &\coloneqq \langle f, (- \cL ) f \rangle_{L^2(\mu_\sigma )} + \delta \langle \cA \, \cL  f, f \rangle_{L^2(\mu_\sigma )} + \delta \langle\cA f, \cL f \rangle_{L^2( \mu_\sigma)}.
\end{aligned}\label{eq:defcE}\end{equation} 
Note that for $f \in C_c^1(\R^d \times \R^d)$ and $\cL$ in~\eqref{eq:cL}, we have $\cL f =\Ltr f +\lambda_\rr(\Pi  f-f)\in L^2(\mu_
\sigma)$ (since $\Ltr f\in C_c(\R^d \times \R^d)$ and $\Pi f\in L^2(\mu_\sigma)$). Hence, using Lemma~\ref{lem:Afbound}, $\cE_\delta (f) < \infty$ for all $f \in C_c^1(\R^d \times \R^d)$.
 Using these definitions, we obtain the following abstract result,  proved in Section~\ref{sec:proofcoercivity}, which, together with Theorem~\ref{thm:stationaritycore}, will imply Theorem~\ref{thm:maincnvg}.

 \begin{theorem}[Coercivity]\label{thm:coercivity}
Let~\ref{ass:wpi}, \ref{ass:sigmabounds} and \ref{ass:curvature} hold.  Let $K_0$ be as in~\eqref{eq:defconstants}.
    Set $$\delta_\ast^{-1} \coloneqq{\lambda_\rr}^{-1}\left({4 + \left(2\sqrt{K_0} + \frac{\lambda_\rr}{2\sqrt{m}}\right)^2}\right),$$
    where $m$ is the Poincar\'e constant in~\ref{ass:wpi} and $\lambda_\rr>0$ the refresh intensity of~\eqref{eq:TBP}.
    Then for $f \in C_c^1(\R^d \times \R^d)$ with $ \mu_\sigma(f) = 0$,
    \begin{equation}\cE_{\delta_\ast} (f) \geq \frac{\delta_\ast}{4} \|f\|_{L^2(\mu_\sigma)}^2.\label{eq:coercivity}\end{equation}
    Also, for $g \in L^2(\mu_\sigma)$ 
    \begin{equation}\frac{1}{4}\|g\|_{L^2(\mu_\sigma)}^2 \leq \cV_{\delta_\ast}(g) \leq \frac{3}{4}\|g\|_{L^2(\mu_\sigma)}^2 . \label{eq:Vbounds}\end{equation}
\end{theorem}

\section{A core of the \texorpdfstring{$L^2(\mu_\sigma)$}{L2}-generator of the semigroup \texorpdfstring{$\cP$}{cP}}
\label{sec:invariance}
In this section, we adapt the methods in~\cite{durmus_piecewise_2021} to obtain fundamental properties of the PDMP $(X, V)$ in~\eqref{eq:TBP} and its generator, which are listed in Theorems~\ref{thm:Pcontraction} and~\ref{thm:stationaritycore} below. It should be noted that the results in this section do not rely on assumptions~\ref{ass:wpi}-\ref{ass:curvature}, only requiring the smoothness of $\sigma$ and $U$, which we assume throughout the paper.  

Recall that the Markov semigroup $\cP = (\cP_t)_{t \in \R_+}$ is defined by $$\cP_t f(x, v) \coloneqq  \E_{(x, v) }[f(X_t, V_t)]\quad\text{for $f \in C_b(\R^d \times \R^d)$ and $(x, v, t) \in \R^d \times \R^d \times \R_+$.}$$
We use the notions of generators of semigroups, cores, and operator closures as defined in~\cite[pp. 8, 17]{ethier_markov_1986}.
We need the following result before we can define the infinitesimal generator of the Markov semigroup $\cP$.
\begin{theorem}\label{thm:Pcontraction} $(\cP_t)_{t \in \R_+}$ uniquely extends to a strongly continuous contraction semigroup on the Hilbert space $L^2(\mu_\sigma)$. 
\end{theorem}
The above theorem follows from Lemma~\ref{lem:Pcontraction} below.
 Using Theorem~\ref{thm:Pcontraction} define the operator $(\bar \cL, \mathrm{Dom}(\bar \cL))$ as the $L^2(\mu_\sigma)$-generator of $(\cP_t)_{t \in \R_+}$. Recall the definition of $\cL$ in~\eqref{eq:cL}.
 
 \begin{theorem}\label{thm:stationaritycore} 
 (I) The set $C_c^1(\R^d \times \R^d) $ is contained in $\mathrm{Dom}(\bar \cL)$. For $f \in C_c^1(\R^d \times \R^d)$, $\bar \cL f = \cL f$. \\
 (II) $C_c^1(\R^d \times \R^d)$ is a core for the generator $(  \bar \cL , \mathrm{Dom}(  \bar \cL))$ of $(\cP_t)_{t \in \R_+}$.\\
 (III) $\mu_\sigma$ is an invariant distribution for $(\cP_t)_{t \in \R_+}$.
\end{theorem} 

 The remainder of this section is devoted to proving Theorems~\ref{thm:Pcontraction} and~\ref{thm:stationaritycore}.
The structure of the argument establishing these results is as follows:\\
\textbf{(I) } In Lemma~\ref{lem:explosiveness}, we show that the PDMP $(X, V)$ is non-explosive. For $\epsilon \in [0, 1]$, we construct PDMPs $(\hat X_t^\epsilon, \hat V_t^\epsilon)_{t \in \R_+}$ that have no \textit{refresh} component (i.e., $\lambda_\rr = 0$) and the \textit{bounce} rate $\lambda_\rb^\epsilon$ defined in~\eqref{eq:boundce_rate_epsilon} below; and we define the corresponding semigroup $\hat\cP^\epsilon = (\hat \cP^\epsilon_t)_{t\in \R_+}$. For $\epsilon \in (0, 1]$, $\lambda_\rb^\epsilon$ are differentiable approximators of $\lambda_\rb^0 \coloneqq \lambda_\rb$ (which is not differentiable).\\
\textbf{(II)} Using Lemmas~\ref{lem:compactsupport} and~\ref{lem:feller}, we show that for $\epsilon \in (0, 1]$, the semigroup $\hat \cP^\epsilon$  is Feller continuous and that  $C_c^1(\R^d \times \R^d)$ is a core of its generator.  By Lemma~\ref{lem:feller}, we also have that the semigroups $(\hat \cP^\epsilon)_{\epsilon \in (0, 1]}$ approximate $\hat\cP^0$. Using this, in Lemma~\ref{lem:core1}, we show that $C_c^1(\R^d \times \R^d)$ is a core of the strong generator of $\hat \cP^0$. On this core the strong generator equals $\Ltr$ in~\eqref{eq:Ltradjoint} and $\hat \cP^0$ is $\mu_\sigma$-invariant. \\
\textbf{(III)} Thus, we define the $L^2(\mu_\sigma)$-generator of $\hat \cP^0$. In Lemma~\ref{lem:core2}, we show that $C_c^1(\R^d \times \R^d)$ is a core of the  $L^2(\mu_\sigma)$-generator of $\hat \cP^0$, which offers the representation $\Ltr$ on this core. Using this, we show that the closure of $\Ltr + \lambda_\rr \Lv $ also generates a strongly continuous contraction semigroup $\cQ = (\cQ_t)_{t \in \R_+}$ on $L^2(\mu_\sigma)$. Step~\textbf{(III)} reintroduces the \textit{refresh} component (recall $\Lv$ from~\eqref{eq:Lv}.) \\
\textbf{(IV)} In Lemma~\ref{lem:Pcontraction}, we identify $\cQ$ with $\cP$, where $\cP=(\cP_t)_{t \in \R_+}$ is the Markov semigroup corresponding to the PDMP $(X, V)$ satisfying~\eqref{eq:TBP}. Hence, $\cP$ is a strongly continuous contraction on $L^2(\mu_\sigma)$, and we define the $L^2(\mu_\sigma)$-generator of $\cP$. Finally, in Theorem~\ref{thm:stationaritycore}, we show that the $L^2(\mu_\sigma)$-generator of $\cP$ aligns with $\cL = \Ltr + \lambda_\rr \Lv$ on a core $C_c^1(\R^d \times \R^d)$, and that $\cP$ is $\mu_\sigma$-invariant. 

\smallskip

We have the following result on non-explosiveness.
It should be noted that the following lemma does not require any of the assumptions~\ref{ass:wpi}, \ref{ass:sigmabounds}, and~\ref{ass:curvature} to hold.
\begin{lemma}[Non-explosiveness]\label{lem:explosiveness}
    Let $\tau \coloneqq \inf\{t \in \R_+ : |X_t| + |V_t| +N^{\rb}_t= \infty\}$. Then $\P(\tau < \infty) = 0$.
\end{lemma}
\begin{proof}
    Consider the case when $\lambda_\rr = 0$. I.e., the PDMP~\eqref{eq:TBP} has no \textit{refresh} jumps. In this case, the process is non-explosive because it moves at constant speed (i.e., the magnitude of velocity does not change), as \textit{bounce} jumps conserve speed.
    Also, using the fact that the bounce rate $\lambda_\rb$ is locally bounded, we have that $\lambda_\rb (X_t, V_t) \leq K_t$, where $K_t = \sup\{\lambda_\rb(x, v) : |x - X_0| \leq |V_0| t \text{ and } |v| = |V_0|\} < \infty$. Therefore, $N^\rb_t < \infty$ a.s., for all $t \in \R_+$.
    Using this, for the case when $\lambda_\rr > 0$, the PDMP~\eqref{eq:TBP} is also non-explosive using \cite[Prop.~9]{durmus_piecewise_2021} and the fact that the refresh rate $\lambda_\rr$ is a finite constant.
\end{proof}

 We modify the notion of \textit{smoothly and compactly approximable} PDMPs from~\cite[Def.~20]{durmus_piecewise_2021} to fit our purpose. Slight modifications are required, as~\cite[Def.~20]{durmus_piecewise_2021} is primarily applicable only to PDMPs with compact velocity spaces. On the other hand, due to the nature of our tempering approach, the velocity state space is $\R^d$. Specifically, the approximating PDMPs $(\hat X^\epsilon, \hat V^\epsilon)_{\epsilon \in (0, 1]}$ as defined below, do not satisfy assumption A2 of~\cite{durmus_piecewise_2021}, hence Lemma~\ref{lem:feller} below is needed.

Recall that $\lambda_\rb : \R^d \times \R^d \rightarrow \R_+$ denotes the non-differentiable jump rate for~\eqref{eq:TBP} given by $\lambda_\rb(x, v) = \langle v, \nabla_x H (x, v) \rangle_+$.
 For $\epsilon \in (0, 1]$, define the function 
\begin{equation}
\label{eq:boundce_rate_epsilon}
    \lambda^\epsilon_\rb(x, v) \coloneqq \frac{(\langle\nabla_x H(x, v), v \rangle - \epsilon )_+^2}{\epsilon + (\langle\nabla_x H(x, v), v \rangle - \epsilon )_+} \quad \text{for } (x, v) \in \R^d \times \R^d.
\end{equation}
For $\epsilon = 0$, set $\lambda_\rb^\epsilon \coloneqq \lambda_\rb$.
For $\epsilon \in [0, 1]$ let $(\hat X^\epsilon, \hat V^\epsilon) \coloneqq (\hat X^\epsilon_t,\hat V^\epsilon_t)_{t \in \R_+} $ be the solution to the following SDE in $\R^d \times \R^d$ defined very similarly to~\eqref{eq:TBP}.
\begin{equation}
    \begin{aligned}
        \hat X_t^\epsilon = \hat X_0^\epsilon &+ \int_0^t \hat V_s^\epsilon \, \ud s,\qquad
        \hat V_t^\epsilon =\hat V_0^\epsilon  +  \int_0^t (R_{\nabla_x H  (\hat X_s^\epsilon,\hat V_s^\epsilon)} \hat V_{s-}^\epsilon -\hat V_{s-}^\epsilon) \, \ud N_s^{\mathrm{b},\epsilon}, 
    \end{aligned}
     \label{eq:TBPapprox}
\end{equation}
for $t \in \R_+$ and $(\hat X_0^\epsilon,\hat V_0^\epsilon)\in \R^d \times \R^d$.  
The counting process 
$ (N_t^{\mathrm{b}, \epsilon})_{t\in\R_+}$ 
in $\N\coloneqq \{0,1,\ldots\}$ has positive jumps of size one arriving at a stochastic intensity given by
$ \lambda_\rb^\epsilon(\hat X_t^\epsilon,\hat V_t^\epsilon)$.
Note that the PDMP $(\hat X^\epsilon, \hat V^\epsilon)$ has no \textit{refresh} behaviour.
Again, as the jump intensity  $\lambda_\rb^\epsilon$ is locally bounded, hence, the PDMP $(\hat X^\epsilon, \hat V^\epsilon)$ is given by~\cite[Construction 1]{durmus_piecewise_2021}. The next lemma has the same proof as Lemma~\ref{lem:explosiveness}.

\begin{lemma} Fix $\epsilon \in [0, 1]$.
    Let $\tau^\epsilon \coloneqq \inf\{t \in \R_+ : |\hat X_t^\epsilon| + | \hat V_t^\epsilon| = \infty\}$. Then $\P(\tau^\epsilon < \infty) = 0$.
\end{lemma}

Define the Markov semigroup $ \hat\cP^\epsilon \coloneqq ( \hat\cP^\epsilon_t)_{t \in \R_+}$ by $\hat\cP_t^\epsilon f (x, v) \coloneqq \E_{(x, v)} [f(\hat X_t^\epsilon, \hat V_t^\epsilon)]$ for $f \in C_b(\R^d \times \R^d)$ and $(x, v) \in \R^d \times \R^d$. We define the norm  $\| f\|_\infty \coloneqq \sup_{x \in H } |f(x)|$ for $f : H \rightarrow \R$ on some Polish space~$H$. 
 Let $C_0(\R^d \times \R^d)$ denote the closure of $C_c(\R^d \times \R^d)$ in the $\| \cdot\|_\infty$ norm. Also, let $\supp(f)$ denote the support of $f : \R^d \times \R^d \rightarrow \R$.
 We prove the following property that helps us show that the semigroup $\hat \cP^\epsilon$ is Feller continuous (defined, e.g., in~\cite [Sec. 8]{durmus_piecewise_2021}) for $\epsilon \in (0, 1]$. The proof of the first part of the following is centered around the idea that for $t \geq 0$, and $\cK_0 \subset \R^d \times \R^d$ compact, there exists compact set $\cK_t \subset \R^d \times \R^d$ such that $(\hat X_t^\epsilon, \hat V_t^\epsilon) \in \cK_0$ implies that $(\hat X_0^\epsilon, \hat V_0^\epsilon) \in \cK_t$.
 \begin{lemma}\label{lem:compactsupport} Fix $\epsilon \in [0, 1]$. Then the following hold.\\
 (I) For $f \in C_c(\R^d \times \R^d)$, we have that $\supp (\hat \cP^\epsilon_tf)$ is compact for $t \in \R_+$. \\
 (II) Let $f \in C_0(\R^d \times \R^d)$. Then $\lim_{t \rightarrow 0}\| \hat \cP_t^\epsilon f - f \|_\infty = 0 $.
 \end{lemma}
 \begin{proof}
 Let $\epsilon \in [0, 1]$.
     Initially, assume $f \in C_c({\R^d \times \R^d})$. Let $R>0$ satisfy  $\supp (f) \subseteq \{(x, v ) \in \R^d \times \R^d:\max(|x|, |v|) \leq R\}$.
     Note that the particle $(\hat X^\epsilon, \hat V^\epsilon)$ moves with constant speed. Fix $t > 0$. Set $\cK_t \coloneqq \{(x, v) \in \R^d \times \R^d : |x| \leq (3/2 + t)R  \text{ and } |v| \leq R\}$.  Consider the following event.
     \begin{align*}\{(\hat X_0^\epsilon, \hat V_0^\epsilon) \notin \cK_t\} &\cap \{(\hat X_t^\epsilon, \hat V_t^\epsilon) \in \supp f \} \subseteq  \{(\hat X_0^\epsilon, \hat V_0^\epsilon) \notin \cK_t\} \cap \{\max(|\hat X_t^\epsilon| , |\hat V_t^\epsilon|) \leq R \}\\
     &= \{(\hat X_0^\epsilon, \hat V_0^\epsilon) \notin \cK_t\} \cap \{\max(|\hat X_t^\epsilon| , |\hat V_0^\epsilon|) \leq R \},
     \end{align*}
     where the last equality holds because $|\hat V_t^\epsilon| = |\hat V_0^\epsilon|$. From the definition of $\cK_t$, and 
     the fact   $\{|\hat V_0^\epsilon| > R\} \cap \{\max(|\hat X_t^\epsilon| , |\hat V_0^\epsilon|) \leq R \} = \emptyset$, we have the following, where we let the event $E_{3/2} \coloneqq \{|\hat X_0^\epsilon| > (3/2+t)R\}$.
     \begin{align*}
     \{(\hat X_0^\epsilon, \hat V_0^\epsilon) \notin \cK_t\}\cap \{(\hat X_t^\epsilon, \hat V_t^\epsilon) \in \supp f \}&=    E_{3/2} \cap \{\max(|\hat X_t^\epsilon| , |\hat V_0^\epsilon|) \leq R \}\\
     &\subseteq E_{3/2}  \cap \{\max(|\hat X_0^\epsilon| - |\hat V_0^\epsilon| t  , |\hat V_0^\epsilon|) \leq R \}\\
     &\subseteq E_{3/2}  \cap \{|\hat X_0^\epsilon| \leq R(1 + t) \}=\emptyset.
     \end{align*}
     Therefore, $\supp(\hat \cP_t^\epsilon f) \subseteq \cK_t$ . This proves part (I). 
     
     Note that for $t \in [0, 1]$, $\supp(\hat \cP_t^\epsilon f) \subseteq \cK_1 $.
     For $(x, v) \in \R^d \times \R^d$, consider the quantity
     \begin{equation}
     \begin{aligned}
         \hat \cP_t^\epsilon f(x, v) - f(x, v)& = \E_{(x, v) }[\2{N_t^{\rb, \epsilon} > 0}(f(\hat X_t^\epsilon, \hat V_t^\epsilon) - f(x, v)) ]\\
         &+ \E_{(x, v) }[\2{N_t^{\rb, \epsilon} = 0}(f(\hat X_t^\epsilon, \hat V_t^\epsilon) - f(x, v)) ].  
     \end{aligned}\label{eq:inter1801}
     \end{equation}
     Note that both of the terms above have their supports contained in $\cK_1$ individually for the same reason as mentioned earlier.
     Let $\tilde \cK_1 \coloneqq \{(x, v ): |x|\leq 7R/2,\> |v|\leq R\}$.
     Note that  $(\hat X_s^\epsilon, \hat V_s^\epsilon)_{s \in [0, t]} \subseteq \tilde \cK_1 $ for $t \in [0, 1]$ and  all initial states $(\hat X_0^\epsilon, \hat V_0^\epsilon) \in \cK_1\subset \tilde \cK_1$. Let $ \lambda_{\tilde \cK_1} \coloneqq \sup_{(x, v) \in \tilde \cK_1} \lambda_\rb^\epsilon(x, v)$ which is finite as $\lambda_\rb^\epsilon$ is locally bounded. Therefore, 
     \begin{align*}
         \sup_{(x, v) \in \R^d \times \R^d }\bigl | \E_{(x, v) }[\2{N_t^{\rb, \epsilon} > 0}&(f(\hat X_t^\epsilon, \hat V_t^\epsilon) - f(x, v)) ] \bigr |\\ &=  \sup_{(x, v) \in \cK_1 }\bigl | \E_{(x, v) }[\2{N_t^{\rb, \epsilon} > 0}(f(\hat X_t^\epsilon, \hat V_t^\epsilon) - f(x, v)) ] \bigr |\\
         &\leq  2\|f\|_\infty (1 - \exp(-\lambda_{\tilde \cK_1}t)) .\numberthis \label{eq:inter1901}
     \end{align*}
     The second term in~\eqref{eq:inter1801} is uniformly bounded in the following way.
     \begin{align}
         \sup_{(x, v) \in \R^d \times \R^d }\bigl | \E_{(x, v) }[\2{N_t^{\rb, \epsilon} = 0}(f(\hat X_t^\epsilon, \hat V_t^\epsilon) - f(x, v)) ] \bigr | \leq \sup_{(x, v) \in \R^d \times \R^d } | f(x + vt, v ) - f(x, v) | \label{eq:inter2001}
     \end{align}
     Note that the right-hand side in the above tends to zero as $t \rightarrow 0$ due to uniform continuity of $f \in C_c(\R^d \times \R^d)$.
     Putting~\eqref{eq:inter1801}, ~\eqref{eq:inter1901} and~\eqref{eq:inter2001} together, we get that $\lim_{t \rightarrow 0}\|\hat\cP_t^\epsilon f -f\|_\infty = 0$.

      As the space $(\{f \in C_b(\R^d \times \R^d) : \lim_{t \rightarrow 0}\|\hat\cP_t^\epsilon f - f \|_\infty = 0\}, \|\cdot \|_\infty)$ is Banach~\cite[p.~28]{davis_markov_2018}, and $C_0(\R^d \times \R^d)$ is the closure of $C_c(\R^d \times \R^d)$ in the uniform norm $\|\cdot \|_\infty$, we have the result for $f \in C_0(\R^d \times \R^d)$.
 \end{proof}
 Next, we have the following lemma on regularity of the semigroup.
 \begin{lemma}\label{lem:feller}
     (I)  Let $\epsilon \in (0, 1]$ and $i \in \{0, 1\}$. Let $f \in C^i(\R^d \times \R^d)$. Then $\hat \cP^\epsilon_t f \in C^i(\R^d \times \R^d)$ for all $t \geq 0$. If $f \in C_c^i(\R^d \times \R^d)$, then $\hat \cP_t^\epsilon f \in C_c^i(\R^d \times \R^d)$.\\
     (II) Let $\epsilon \in [0, 1]$. Let $f \in C_0(\R^d \times \R^d)$. Then $\|\hat\cP_t^0 f - \hat \cP_t^\epsilon f \|_\infty \leq 4 t\epsilon \|f \|_\infty$ for all $t \geq 0$.\\
     (III) Let $\epsilon \in [0, 1]$. Let $f \in C_0(\R^d \times \R^d)$. Then $\hat \cP_t^\epsilon f \in C_0(\R^d \times \R^d)$ for all $t \geq 0$.
 \end{lemma}
 \begin{proof} \underline{Proof of part (I):}
 For $\epsilon \in (0, 1]$, it can be seen that $\lambda_\rb^\epsilon \in C^1(\R^d \times \R^d)$ as $U$ and $\sigma$ are smooth. Moreover, for $\epsilon \in (0, 1]$, 
      the PDMP $(\hat X^\epsilon, \hat V^\epsilon)$ satisfies A3 of~\cite{durmus_piecewise_2021}, in a way very similar to~\cite[Prop. 23]{durmus_piecewise_2021}. Thus, using~\cite[Thm 17]{durmus_piecewise_2021}, we have the first statement of part (I) of this lemma. The second statement of part (I) follows from Lemma~\ref{lem:compactsupport}(I).

      \noindent \underline{Proof of part (II):} Since  $\sup_{(x, y) \in \R^d \times \R^d} |\lambda_\rb (x, y ) - \lambda_\rb^\epsilon(x, y)| \leq 2 \epsilon$, Part (II) follows by~\cite[Thm~11]{durmus_piecewise_2021}.
      
      \noindent \underline{Proof of part (III):} Assume first $\epsilon \in (0, 1]$. Recall that $C_0(\R^d \times \R^d)$ is the closure of $C_c(\R^d \times \R^d)$ in the norm $\| \cdot \|_\infty$.
      For $\delta > 0$ and $f \in C_0(\R^d \times \R^d)$, choose $f_\delta \in C_c(\R^d \times \R^d)$ such that $\|f - f_\delta\|_\infty < \delta$.
       We have that $\|\hat\cP_t^\epsilon f - \hat \cP_t^\epsilon f_\delta \|_\infty = \|\hat\cP_t^\epsilon( f -  f_\delta) \|_\infty \leq \| f -  f_\delta \|_\infty < \delta.$
       Note that $\hat\cP_t^\epsilon f_\delta \in C_c(\R^d \times \R^d)$ because of part (I) of this lemma (which may be applied as $\epsilon > 0$). Thus, as $\delta > 0$ can be chosen arbitrarily small, we have that $\hat \cP_t^\epsilon f$ is in the closure of $C_c(\R^d \times \R^d)$ in the uniform norm $\|\cdot \|_\infty$. In other words, $\hat\cP_t^\epsilon f \in C_0(\R^d \times \R^d)$. By using part (II) of this lemma, using a similar approximation technique, utilising the fact that $C_0(\R^d \times \R^d)$ is closed, we get the result for the case  $\epsilon = 0$.
 \end{proof}
 
Using Lemma~\ref{lem:compactsupport}(II) and Lemma~\ref{lem:feller}(III) and, $\hat \cP^\epsilon$ is a strongly continuous contraction semigroup on the Banach space $(C_0(\R^d \times \R^d), \|\cdot \|_\infty)$ for $\epsilon \in [0, 1]$. Define the strong generator $(\hat\cL^\epsilon,\mathrm{Dom}(\hat \cL^\epsilon))$ by 
$\hat \cL^\epsilon f \coloneqq \lim_{t \rightarrow 0 }(\hat \cP_t^\epsilon f - f  )/ t $ for $f \in \mathrm{Dom}(\hat \cL^\epsilon)$, where 
$$\mathrm{Dom}(\hat \cL^\epsilon) \coloneqq \{ f \in C_0 : \exists\, h \in C_0 \text{ such that }\lim_{t \rightarrow 0 }\|(\hat\cP_t^\epsilon f - f)/ t - h \|_\infty = 0 \}.$$ 
In the above, we use $C_0 \coloneqq C_0(\R^d \times \R^d)$ for brevity.
\begin{lemma}\label{lem:core1}
(I) $C_c^1(\R^d \times \R^d) \subseteq \mathrm{Dom}(\hat \cL^\epsilon)$ for $\epsilon \in [0, 1]$.\\
(II) $C_c^1(\R^d \times \R^d)$ is a core of $(\hat \cL^0,\mathrm{Dom}(\hat \cL^0))$.\\
(III) For $f \in C_c^1(\R^d \times \R^d)$, we have $\hat \cL^0 f = \Ltr f$.\\
(IV) $\mu_\sigma$ is an invariant distribution of $(\hat \cP_t^0)_{t \in \R_+}$.
\end{lemma}
\begin{proof}
Parts (I) and (III) follow by an application of the It\^o formula (e.g.,~\cite[Ch.2 Thm~32]{protter_stochastic_2004}), and using the fact that $f \in C_c^1(\R^d \times \R^d)$ is uniformly continuous and has bounded first derivatives that are also uniformly continuous.

     Using Lemma~\ref{lem:feller}(I), (II) and~\cite[Prop.27]{durmus_piecewise_2021}, we have that $C_c^1(\R^d \times \R^d)$ is a core of $(\hat \cL^0,\mathrm{Dom}(\hat \cL^0))$, hence we get part (II). Part (IV) follows from the fact that for $f \in C_c^1(\R^d \times \R^d)$ (which forms a core), we have that $\int_{\R^d \times \R^d}  \hat\cL^0 f \,\ud\mu_\sigma= \langle 1, \Ltr f \rangle_{L^2(\mu_\sigma)} = \langle \Ltr^\star 1, f \rangle_{L^2(\mu_\sigma)} = 0$ using~\eqref{eq:L3}.
\end{proof}
\begin{remark}The reader should note that the operator $\cL_0$ defined in section~\ref{sec:proof3.2} is not related to $(\hat \cL^0,\mathrm{Dom}(\hat \cL^0))$ as defined here.\end{remark} 
Due to Lemma~\ref{lem:core1}(IV), the semigroup $\hat \cP^0$, given by $\hat \cP^0_t f(x, v) = \E_{(x, v)} [f(\hat X_t^0, \hat V_t^0)]$ for $f \in C_0(\R^d \times \R^d)$ can be uniquely extended to a strongly continuous contraction semigroup on $L^2(\mu_\sigma)$.
Define the $L^2(\mu_\sigma)$-generator corresponding to $\hat \cP^0$ by $(\bar \cL^0, \mathrm{Dom}( \bar \cL^0))$. For operator $A$ defined on space $\mathrm{Dom}(A)$, let $\mathrm{Ran}(A|_D)$ denote the range of $A$ restricted to $D \subset \mathrm{Dom}(A)$.

\begin{lemma}\label{lem:core2} (I) $\bar \cL^0 f = \Ltr f$ for $f \in C_c^1(\R^d \times \R^d)$.\\
(II) $C_c^1(\R^d \times \R^d)$ is a core of  $(\bar \cL^0, \mathrm{Dom}( \bar \cL^0))$.\\
(III) The operator $\cL= \lambda_\rr \Lv + \Ltr$ in~\eqref{eq:cL} with the domain $C_c^1(\R^d \times \R^d)$ in $L^2(\mu_\sigma)$ generates a strongly continuous contraction semigroup on $L^2(\mu_\sigma)$.
\end{lemma}
\begin{proof}
Part (I) follows from Lemma~\ref{lem:core1}(I) using the fact that $\|\cdot \|_{L^2(\mu_\sigma)} \leq \|\cdot \|_\infty$.

\noindent \underline{Proof of part (II):}
Note that $\|f \|_{L^2(\mu_\sigma)} \leq \|f\|_\infty$ for all $f \in C(\R^d \times \R^d)$. 
    It follows from Lemma~\ref{lem:core1}(II), $C_c^1(\R^d \times \R^d)$ is a core of $(\hat \cL^0,\mathrm{Dom}(\hat \cL^0))$. According to \cite[Ch.1 Prop.3.1]{ethier_markov_1986},  $\mathrm{Ran}(k - \hat \cL^0\big|_{C_c^1(\R^d \times \R^d)})$ is dense in $C_0(\R^d \times \R^d )$ in the norm $\|\cdot \|_\infty$ for some $k > 0$. Thus, due to Lemma~\ref{lem:core1}(III),  $\mathrm{Ran}(k - \Ltr|_{C_c^1(\R^d \times \R^d)}) $ is dense in $C_0(\R^d \times \R^d )$ in the $\|\cdot \|_\infty$ norm.
    Moreover, $C_0(\R^d \times \R^d)$ is dense in $L^2(\mu_\sigma)$ (in the $\|\cdot \|_{L^2(\mu_\sigma)}$ norm).
    This implies that $\mathrm{Ran}(k - \Ltr|_{C_c^1(\R^d \times \R^d)}) $ is dense in $L^2(\mu_\sigma)$.
    Using part (I) of this lemma,  $\mathrm{Ran}(k - \bar \cL^0|_{C_c^1(\R^d \times \R^d)}) $ is dense in $L^2(\mu_\sigma)$ and hence $C_c^1(\R^d \times \R^d)$ is a core of  $(\bar \cL^0, \mathrm{Dom}( \bar \cL^0))$, using~\cite[Ch.1 Prop.3.1]{ethier_markov_1986} and the fact that $C_c^1(\R^d \times \R^d)$ is dense in $L^2(\mu_\sigma)$.

\noindent \underline{Proof of part (III):} Part (III) follows from part (II) and~\cite[Ch.~1,~Thm~7.1]{ethier_markov_1986}, using the fact that due to parts (I) and (II), the $L^2(\mu_\sigma)$-closure of $(\Ltr, C_c^1(\R^d \times \R^d))$ is $\bar \cL^0 $ which generates the strongly continuous contraction semigroup $\hat \cP^0$ and that the operator $\Lv f = \Pi f - f$ is dissipative in the sense of the definition  in~\cite[Ch.1 Sec.2]{ethier_markov_1986}.
\end{proof}

Set $\cQ \coloneqq (\cQ_t)_{t \in \R_+}$ as the strongly continuous contraction semigroup generated by the $L^2(\mu_\sigma)$-closure of $(\cL, C_c^1(\R^d \times \R^d))$ as given by Lemma~\ref{lem:core2}(III). Recall the Markov semigroup $\cP = (\cP_t)_{t \in \R_+}$ as defined in Section~\ref{sec:result} given by $\cP_t f (x, v) = \E_{(x, v)} [f(X_t, V_t)]$ for $f \in C_b(\R^d \times \R^d)$ and $(x, v) \in \R^d \times \R^d$, where $(X, V) = (X_t, V_t)_{t \in \R_+}$ is the piecewise deterministic Markov process given by~\eqref{eq:TBP}.
\begin{lemma}\label{lem:Pcontraction}
    For all $f \in C^1_b(\R^d \times \R^d)$ and $t \in \R_+$, we have $\cQ_t f =  \cP_t f $ $\mu_\sigma$-a.s. Also, $(\cP_t)_{t \in \R_+} $ uniquely extends to the strongly continuous contraction semigroup $(\cQ_t)_{t \in \R_+}$ on $L^2(\mu_\sigma)$.
\end{lemma}
\begin{proof}
    Define the operator $A   \coloneqq \Ltr   - \lambda_\rr I$ on $C_c^1(\R^d \times \R^d)$. Note that the $L^2(\mu_\sigma)$-closure of $(\Ltr, C_c^1(\R^d \times \R^d))$ generates the strongly continuous contraction semigroup $\hat \cP^0$.  Hence, by~\cite[Ch.1 Thm 7.1]{ethier_markov_1986}, we have that the $L^2(\mu_\sigma)$-closure of $(A, C_c^1(\R^d \times \R^d))$ generates a strongly continuous contraction semigroup on $L^2(\mu_\sigma)$.
    Call this semigroup $\cS \coloneqq (\cS_t)_{t \in \R_+} $.
    Using~\cite[Ch.~III,~Thm~1.10]{engel_one-parameter_2000}, we have that $\cS_t =\sum_{n \in \N} \cS_t^{(n)}$, where $\cS_t^{(0)}f \coloneqq \hat \cP_t^{0}f$ and $\cS_t^{(n)} \coloneqq -\lambda_\rr\int_0^t \hat \cP_{t-s}^0 \cS_s^{(n - 1)} f \, \ud s $ for $n \in \N \setminus \{0\}$. Inductively, it can be seen that $\cS_t^{(n)} f = (- \lambda_\rr t)^n\hat \cP_t^0 f /(n!)$. Thus, $\cS_tf = \exp(- \lambda_\rr t) \hat \cP_t^0 f$ for $f \in L^2(\mu_\sigma)$. Hence, the closure of $(A, C_c^1(\R^d \times \R^d))$ generates the semigroup $(\exp(-\lambda_\rr t ) \hat \cP_t^0)_{t \in \R_+}$.
    
    Note that the operator $ \lambda_\rr\Pi $ is bounded on $L^2(\mu_\sigma)$ . We have that $A + \lambda_\rr \Pi = \cL$. Also, recall that $(\cQ_t)_{t \in \R_+}$ is the strongly continuous contraction semigroup generated by the $L^2(\mu_\sigma)$-closure of $(\cL, C_c^1(\R^d \times \R^d))$.
    Therefore, using~\cite[Ch.III Thm 1.10]{engel_one-parameter_2000}, we have that $\cQ_t = \sum_{n \in \N} \cQ^{(n)}_t $, where $\cQ_t^{(0)} f  \coloneqq \cS_t f $ and for $n \in \N\setminus\{0\}$, $\cQ_t^{(n)} f  \coloneqq \int_0^t \cS_{t-s} (\lambda_\rr \Pi) \cQ_s^{(n - 1)} f\, \ud s $ for $f \in L^2(\mu_\sigma)$.
    
    Moreover, for $f \in C_b^1(\R^d \times \R^d)$, the operation $\cP_t f(x, v) = \E_{(x, v)} [f(X_t, V_t)]$ can be decomposed as $\cP_t f(x, v) = \sum_{n \in \N} \cP_t^{(n)} f$ where
    $\cP_t^{(n)} f(x, v) \coloneqq \E_{(x, v)} [ \2{N^\rr_t = n }f(X_t, V_t)]$ for $n \in \N$. Therefore 
    \begin{align*}\cP_t^{(0)} f (x, v) &= \E_{(x, v)} [ \2{N^\rr_t = 0 }f(X_t, V_t)]\\
    &=  \E_{(x, v)} [ f(X_t, V_t) | N_t^\rr = 0] \P_{(x, v) }(N_t^\rr = 0) \\
    &= \exp(-\lambda_\rr t)\hat \cP^0_t f = \cS_t f = \cQ_t^{(0)} f\quad \text{for all } t \in \R_+ \text{ and } f \in C^1_b(\R^d \times \R^d).
    \end{align*}
    We proceed inductively. Assume that for some $n \in \N$, $\cP_t^{(n)} f = \cQ_t^{(n)} f$ $\mu_\sigma$-a.e. for all $t \in \R_+$. Let $\tau_{ 1}\coloneqq \inf\{t \in \R_+:N_t^\rr =  1\}$. Then we have that
    \begin{align*}
        \cP_t^{(n+1)} f(x, v) &= \E_{(x, v)} [ \2{N^\rr_t = n + 1 }f(X_t, V_t)] = \E_{(x, v)}[\E [ \2{N^\rr_t = n + 1 }f(X_t, V_t) | (\tau_1 , X_{\tau_1}, V_{\tau_1})]]\\
        &= \E_{(x, v)}[\cP_{t- \tau_1}^{(n)} f (X_{\tau_1}, V_{\tau_1})]= \E_{(x, v)}[\E[\cP_{t- \tau_1}^{(n)} f (X_{\tau_1}, V_{\tau_1})|( X_{\tau_1-}, V_{\tau_1-})]]\\
        &= \int_0^t\E_{(x, v)}\biggl[\E\bigl[\cP_{t- s}^{(n)} f (X_{s}, V_{s})\bigm|(X_{s-}, V_{s-})\bigr] \biggm| \tau_1 = s\biggr]\,\P_{(x, v)}(\tau_1 \in \ud s).
    \end{align*}
    We have that $\E\bigl[\cP_{t- s}^{(n)} f (X_{s}, V_{s})\bigm|(X_{s-}, V_{s-})\bigr] = \Pi \cP_{t- s}^{(n)} f (X_{s-}, V_{s-})$. Hence, we get that 
    \begin{align*}
        \cP_t^{(n+1)} f(x, v)&= \int_0^{t} \E_{(x, v)}[\Pi \cP_{t- s}^{(n)} f (X_{s-}, V_{s-})| \tau_1 = s ] \,\P_{(x, v)}(\tau_1 \in \ud s)\\
        &= \int_0^{t} \hat \cP_s^0\Pi \cP_{t- s}^{(n)} f (x, v) \,\P_{(x, v)}(\tau_1 \in \ud s) \\
        &= \int_0^{t}\lambda_\rr e^{- \lambda_\rr s} \hat \cP_s^0\Pi \cP_{t- s}^{(n)} f (x, v)\, \ud s = \cQ_t^{(n+1) } f(x, v).
    \end{align*}
    where, we use the induction hypothesis, $e^{- \lambda_r s} \hat \cP^0_s = \cS_s$ and the change of variable $s \mapsto t - s$ in the last step.
     Thus, we have that $\cP_t f = \cQ_t f$ $\mu_\sigma$-a.s.

    The second statement in the lemma follows from the bound $\|\cP_t f \|_{L^2(\mu_\sigma)} = \|\cQ_t f \|_{L^2(\mu_\sigma)} \leq \|f\|_{L^2(\mu_\sigma)}$ for $f \in C_b^1(\R^d \times \R^d)$ to uniquely extend $\cP_t$ to $L^2(\mu_\sigma)$. Strong continuity follows from this extension, using the fact that $(\cQ_t)_{t \in \R_+}$ is strongly continuous.
\end{proof}

\begin{proof}[Proof of Theorem~\ref{thm:stationaritycore}]
 Parts (I) and (II) follow from the fact $(\cP_t)_{t \in \R_+}=(\cQ_t)_{t \in \R_+}$ on $L^2(\mu_\sigma)$ (by Lemma~\ref{lem:Pcontraction}) and $(\cQ_t)_{t \in \R_+}$ is generated by the $L^2(\mu_\sigma)$-closure of $(\cL, C_c^1(\R^d \times \R^d))$ (by Lemma~\ref{lem:core2}). 
 
   For part (III),  let $f \in C_c^1(\R^d \times \R^d)$.  Thus, by Part (I), we have $\bar \cL f = \cL f = \lambda_\rr\Lv f + \Ltr f$ and
    \begin{align}
        \int_{\R^d \times \R^d}& \bar \cL f \,\ud \mu_\sigma  =  \langle 1, \Ltr f \rangle_{L^2(\mu_\sigma)}+ \lambda_r \langle 1, \Lv f \rangle_{L^2(\mu_\sigma)}.
        \label{eq:L_bar}
    \end{align}
    By~\eqref{eq:L3}, we have 
    $\langle 1, \Ltr f \rangle_{L^2(\mu_\sigma)} = \langle \Ltr^\star 1 , f \rangle_{L^2(\mu_\sigma)} = 0$. 
  Since $\Pi$ is a self-adjoint operator on $L^2(\mu_\sigma)$ and $\Lv=\Pi-I$, the operator $\Lv$ is self-adjoint, implying $\langle\Lv f, 1\rangle_{L^2(\mu_\sigma)} = \langle f, \Lv 1\rangle_{L^2(\mu_\sigma)} = 0$.
    Thus, by~\eqref{eq:L_bar}, we have $\int_{\R^d \times \R^d} \bar \cL f \,\ud \mu_\sigma = 0$.
\end{proof}

\section{Proofs}\label{sec:proofs}
\subsection{Structure of the proof of Theorem~\ref{thm:maincnvg}}
\label{subsec:Sturcture_of_proof_of_main_thm}
This section proves Theorem~\ref{thm:maincnvg}, together with the auxiliary theorems and lemmas from Section~\ref{sec:hypocoercivity} on which its proof depends; see Figure~\ref{fig:lemmadeps}. It also establishes Lemmas~\ref{lem:DMSboundA} and~\ref{lem:DMSboundB}, which bound the individual terms in the coercivity inequality~\eqref{eq:coercivity}, thereby yielding Theorem~\ref{thm:coercivity}. The proof of our main result, Theorem~\ref{thm:maincnvg}, crucially relies on the identification of the core of~\eqref{eq:TBP} in Theorem~\ref{thm:stationaritycore} of Section~\ref{sec:invariance}.

%
%

\begin{figure}[H]
\centering
\begin{tikzpicture}[
    >={stealth},
    lem/.style  = {draw, rounded corners=2pt, align=center, inner sep=3.5pt,
                   font=\small},
    thm/.style  = {draw, line width=1pt, rounded corners=2pt, align=center,
                   inner sep=3.5pt, font=\small},
    aux/.style  = {draw, densely dashed, rounded corners=2pt, align=center,
                   inner sep=3.5pt, font=\small},
    note/.style = {align=left, font=\footnotesize\itshape, text width=3.4cm},
    imp/.style  = {->, thick},
    side/.style = {->, thick, densely dotted},
]




\node[lem] (wb)    at (0.7,-4.5){Lemmas~\ref{lem:weightedbochner} \&~\ref{lem:bochnerineq} \\ Weighted Bochner formula and inequality};



\node[lem] (l32)   at (6.6,-4.5){Lemma~\ref{lem:3.2}\\ $\|\Ltr^\star\Ltr g\|^2 \le K_0\|\bar \cL_\sigma g\|^2$};


\node[lem] (dmsA)  at (0,-6.2)   {Lemma~\ref{lem:DMSboundA}
\\$|\langle \cA f, \cL f \rangle_{L^2(\mu_\sigma )}| \leq \|f- \Pi f\|_{L^2(\mu_\sigma)}^2$
};
\node[lem] (dmsB)  at (6.6,-6.2) {Lemma~\ref{lem:DMSboundB}
\\$\langle \cA\,\cL f, f \rangle_{L^2( \mu_\sigma )} \geq \frac{1}{2} \|\Pi f \|_{L^2(\mu_\sigma)}^2 + $ other terms
};

\node[thm] (coer)  at (6.6,-7.8){Theorem~\ref{thm:coercivity}\\ Coercivity in $\cV_\delta$};
\node[lem] (stat)  at (1.8,-7.8){Theorem~\ref{thm:stationaritycore}\\ Core and $\mu_\sigma$-invariance of $\bar \cL$};

\node[thm] (main)  at (6.6,-9.4){Theorem~\ref{thm:maincnvg}\\ $L^2(\mu_\sigma)$-decay of semigroup};


\draw[imp] (wb)    -- (l32);
\draw[imp] (l32)   -- (dmsB);
\draw[imp] (dmsA)  |- ($(coer.north)+(0,0.35)$) -- (coer.north);
\draw[imp] (dmsB)  |- ($(coer.north)+(0,0.35)$) -- (coer.north);
\draw[imp] (coer)  |- ($(main.north)+(0,0.35)$) -- (main.north);
\draw[imp] (stat)  |- ($(main.north)+(0,0.35)$) -- (main.north);




\end{tikzpicture}
\caption{Structure of the hypocoercivity framework in Theorem~\ref{thm:maincnvg}. Auxiliary results, used in the proofs of the theorems and lemmas in the diagram, are depicted in Figure~\ref{fig:auxdeps}. 
}

\label{fig:lemmadeps}
\end{figure}
Lemmas~\ref{lem:TLD2},~\ref{lem:lifts},  \ref{lem:Afbound} and~\ref{lem:wtr} are auxiliary results used in the proofs of Lemmas~\ref{lem:DMSboundA} and~\ref{lem:DMSboundB}, and Theorems~\ref{thm:coercivity} and~\ref{thm:maincnvg}. Their proofs are based on Lemmas~\ref{lem:dirichletfinite}, \ref{lem:compactapprox}, \ref{lem:trdensity} and~\ref{lem:3.2}, which are proved ``from first principles'' (see Figure~\ref{fig:auxdeps}). 
%
%

\begin{figure}[H]
\centering
\begin{tikzpicture}[
    >={stealth},
    lem/.style  = {draw, rounded corners=2pt, align=center, inner sep=3.5pt,
                   font=\small},
    thm/.style  = {draw, line width=1pt, rounded corners=2pt, align=center,
                   inner sep=3.5pt, font=\small},
    aux/.style  = {draw, densely dashed, rounded corners=2pt, align=center,
                   inner sep=3.5pt, font=\small},
    note/.style = {align=left, font=\footnotesize\itshape, text width=3.4cm},
    imp/.style  = {->, thick},
    side/.style = {->, thick, densely dotted},
]


\node[lem] (dir) at (0, 0) {\scriptsize Lemma~\ref{lem:dirichletfinite}\\
\scriptsize $\int_{\R^d} \sigma^2 |\nabla_x g|^2 \ud \mu < \infty$
};

\node[lem] (tld2)  at (0,-1.8)   {\scriptsize Lemma~\ref{lem:TLD2}\\ \scriptsize Properties of $\bar\cL_\sigma$};

\node[lem] (cmp) at (3.3, 0) {\scriptsize Lemma~\ref{lem:compactapprox}:
$C_c^n$ dense\\ \scriptsize in $\{\cL_\sigma g \in L^2(\mu)\}$
};

\node[lem] (wtr)   at (6.8,-1.8) {\scriptsize Lemma~\ref{lem:wtr}\\  \scriptsize $C^n_\Pi$, $W_\mathrm{tr}$ and $W_{\mathrm{tr}^\star}$};
\node[lem] (afb)   at (3.4,-1.8) {\scriptsize Lemma~\ref{lem:Afbound}\\\scriptsize  $\cA$, $\Ltr\cA$ are bounded
};

\node[lem] (dens) at (10.2,0)    {\scriptsize Lemma~\ref{lem:trdensity}\\  \scriptsize $C_c^1$ is dense in $W_\mathrm{tr}$};
\node[lem] (lifts) at (10.2,-1.8) {\scriptsize Lemma~\ref{lem:lifts}\\ \scriptsize \textit{Lift} properties: \eqref{eq:L1}--\eqref{eq:L3}};

\node[lem] (l32)   at (6.8,0){\scriptsize Lemma~\ref{lem:3.2}\\ \scriptsize $\|\Ltr^\star\Ltr g\|^2 \le K_0 \|\bar \cL_\sigma g\|^2$};



\draw[imp] (dens)  -- (lifts);

\draw[imp] (l32) -- (wtr);
\draw[imp] (dir) -- (cmp);
\draw[imp] (dir) |- ($(dir.south)+(0, -0.33)$) -| ($(tld2.north)+(0, 0.37)$) -- (tld2.north);
\draw[imp] (cmp) |- ($(cmp.south)+(0, -0.35)$) -| ($(tld2.north)+(0, 0.35)$) -- (tld2.north);
\draw[imp] (wtr) -- (afb);



\end{tikzpicture}
\caption{ Auxiliary lemmas (the term \textit{lift} in Lemma~\ref{lem:lifts} is is used in the sense of~\cite[Def.1]{eberle_non-reversible_2026}). Recall the definition $\cA \coloneqq (m - \bar \cL_\sigma)^{-1} \Pi \Ltr^\star$ in Lemma~\ref{lem:Afbound}.
}

\label{fig:auxdeps}
\end{figure}

\subsection{Proof of Lemma~\ref{lem:TLD2}}\label{sec:proofTLD2}
Throughout this section, we let $\phi : \R_+ \rightarrow [0, 1]$ be a smooth non-increasing cutoff function such that $\phi(s) = 1$ for $s \in [0, 1]$ and $\phi(s) = 0$ for $s \in [2, \infty)$.
For $n \in \N \setminus \{0\}$, set $\phi_n(x) \coloneqq  \phi(|x| / n)$ for $x \in \R^d$. By ~\ref{ass:sigmabounds}, we have that
$\2{|x| \leq 2n} \sigma(x) \leq \sigma(0) + 2n M_D$, 
using the mean value theorem for $\sigma$.
Hence, under~\ref{ass:sigmabounds}, we have the bound
    \begin{align*}
        \sigma(x)|\nabla_x \phi_n(x)| &= \2{|x| \leq 2n} \sigma(x) |\nabla_x \phi_n(x)| \leq \| \phi'\|_\infty\frac{1}{n} \2{|x| \leq 2n} \sigma(x) \\
        &\leq \| \phi'\|_\infty\frac{1}{n} \left(\sigma(0) + 2nM_D \right) \leq \| \phi'\|_\infty \left(\sigma(0) + 2M_D \right) =: M_{\phi, 1}  ,\numberthis \label{eq:inter4202}
    \end{align*}
     Similarly, the following bound holds.
    \begin{align*}
        \sigma^2(x) |\Delta_x \phi_n(x)|&\leq  
        (\sigma(0) + 2 M_D)^2 \left(\|\phi''\|_\infty +({d-1}) \|\phi'\|_\infty\right) =: M_{\phi, 2}. \numberthis \label{eq:inter4302}
    \end{align*}

\begin{proof}[Proof of Lemma~\ref{lem:TLD2}]
\underline{Proof of part (I):}
Recall the definition of $(\bar\cL_\sigma, \mathrm{Dom}(\bar \cL_\sigma))$ from Section~\ref{sec:defgen}. For $g \in \mathrm{Dom}(\bar \cL_\sigma)\cap C_b^2(\R^d)$, define the difference quotient $Q_t(x) = (\cT_t g (x) - g(x) )/ t $ for $t > 0$, and choose a sequence $(t_n)_{n \in \N} \subseteq (0, \infty)$ such that $t_n \rightarrow 0$ as $n \rightarrow \infty$. As $g \in \mathrm{Dom}(\bar \cL_\sigma)$, $Q_{t_n} \rightarrow \bar \cL_\sigma g $ in $L^2(\mu)$. Thus, using~\cite[Cor.~2.32]{folland_real_1999}, there exists a subsequence $(t_{n_j})_{j \in \N}$ such that $Q_{t_{n_j}} \rightarrow \bar \cL_\sigma g$ $\mu$-almost everywhere.
Moreover, as $g \in C_b^2(\R^d)$, by an application of the It\^o formula (e.g.,~\cite[Ch.2 Thm~32]{protter_stochastic_2004}), 
 $Q_{t_n} (x) \rightarrow -(\nabla_x^\star \sigma^2 \nabla_x g)(x) = \cL_\sigma g(x)$ for all $x \in \R^d$. In particular, this limit holds along the subsequence $(t_{n_j})_{j \in \N}$. 
 Hence, we have that $\bar \cL_\sigma g = \cL_\sigma g$ $\mu$-a.e.

The fact that $C_c^2(\R^d) \subseteq \mathrm{Dom}(\bar \cL_\sigma)$  follows by an application of the It\^o formula (e.g.,~\cite[Ch.2 Thm~32]{protter_stochastic_2004}) and the dominated convergence theorem.

We have that $C_c^\infty(\R^d)$ is dense in $L^2(\mu)$, hence $C_b^\infty(\R^d) \cap \mathrm{Dom}(\bar \cL_\sigma)$ is dense in $L^2(\mu)$.
Using~\cite[Ch.1 Prop.1.5]{ethier_markov_1986} and~\cite[Thm~1.5.2]{cerrai_second_2001}, for $g \in C_b^\infty(\R^d)$, we have $\cT_t g \in  C^\infty_b(\R^d) \cap \mathrm{Dom}(\bar \cL_\sigma)$.
Using~\cite[Ch.1 Prop.3.3]{ethier_markov_1986}, $C^\infty_b(\R^d) \cap \mathrm{Dom}(\bar \cL_\sigma)$ is a core of $(\bar\cL_\sigma, \mathrm{Dom}(\bar \cL_\sigma))$. 

Hence, to show that $C_c^\infty(\R^d)$ is a core of $(\bar\cL_\sigma, \mathrm{Dom}(\bar \cL_\sigma))$, it remains to show that for all $g \in C^\infty_b(\R^d) \cap \mathrm{Dom}(\bar \cL_\sigma)$, there exists a sequence $(g_n)_{n \in \N} \subseteq C_c^\infty(\R^d)$ such that the folloiwng limit holds $\| g_n - g \|_{L^2(\mu)} + \|\bar \cL_\sigma(g_n - g)\|_{L^2(\mu)} \rightarrow 0$. Using Lemma~\ref{lem:compactapprox} and part (I), we get such a sequence.

\smallskip

\noindent \underline{Proof of part (II):}
Let $f, g \in C^2_b(\R^d) \cap \mathrm{Dom}(\bar \cL_\sigma)$. Recall that $\phi_n$ is the cutoff function defined at the beginning of this section for $n \in \N\setminus \{0\}$. Recall that $\cL_\sigma f = - \nabla_x^\star \sigma^2 \nabla_x f$. Using integration-by-parts, we have
\begin{align*}
\langle \phi_n g, (- \cL_\sigma) f \rangle_{L^2(\mu)} &= \langle \nabla_x (\phi_n g ), \sigma^2 \nabla_x f \rangle_{L^2(\mu)} \\&= \langle g\nabla_x \phi_n  , \sigma^2 \nabla_x f \rangle_{L^2(\mu)} + \langle \phi_n  \nabla_x g  , \sigma^2 \nabla_x f \rangle_{L^2(\mu)}. \numberthis \label{eq:inter5101}
\end{align*}
Using dominated convergence theorem, we have that $\|\phi_n g - g\|_{L^2(\mu)} \rightarrow 0 $. Therefore,
\begin{align}
\lim_{n \rightarrow \infty}|\langle \phi_n g, (- \cL_\sigma) f \rangle_{L^2(\mu)} - \langle g, (- \cL_\sigma) f \rangle_{L^2(\mu)}| \leq \lim_{n \rightarrow \infty} \|\phi_n g - g\|_{L^2(\mu)} \|\cL_\sigma f \|_{L^2(\mu)} = 0, \label{eq:inter52} \end{align}
since $\|\cL_\sigma f \|_{L^2(\mu)} = \|\bar\cL_\sigma f\|_{L^2(\mu)} < \infty$ by part (I) and $f \in \mathrm{Dom}(\bar \cL_\sigma) \cap C_b^2(\R^d)$. 

For the first term on the right-hand side of~\eqref{eq:inter5101}, the integrand satisfies the following inequality, using~\eqref{eq:inter4202}.
$$
    |g \sigma^2\langle  \nabla_x \phi_n ,  \nabla_x f \rangle | \leq \sigma^2|\nabla_x \phi_n ||g \nabla_x f| \leq M_{\phi, 1}|g| |\sigma \nabla_x f|.
$$
Also, note that $$\int_{\R^d}|g||\sigma\nabla_x f |\ud \mu \leq \|g\|_{L^2(\mu)} (\int \sigma^2 |\nabla_x f |^2 \ud \mu)^{1/2} < \infty$$ using Lemma~\ref{lem:dirichletfinite}. 
Hence, by the dominated convergence theorem, 
\begin{align*}
    \lim_{n \rightarrow \infty } \langle g \nabla_x \phi_n , \sigma^2 \nabla_x f \rangle_{L^2(\mu)} &=   \int_{\R^d}\lim_{n \rightarrow \infty } g \sigma^2 \langle \nabla_x \phi_n , \nabla_x f \rangle \,\ud \mu\\ &= \int_{\R^d}\lim_{n \rightarrow \infty } \2{|x|\geq n}g \sigma^2 \langle \nabla_x \phi_n , \nabla_x f \rangle \,\ud \mu = 0.
    \numberthis \label{eq:inter53}
\end{align*}
For the second term in the right-hand side of~\eqref{eq:inter5101}, we get that
$$|\phi_n \sigma^2 \langle \nabla_x g ,  \nabla_x f \rangle | \leq |\sigma \nabla_x g| |\sigma\nabla_x f |,$$
and $\int_{\R^d}|\sigma \nabla_x g| |\sigma\nabla_x f | \,\ud \mu < \infty $ as $|\sigma \nabla_x g| , |\sigma\nabla_x f | \in L^2(\mu)$ using Lemma~\ref{lem:dirichletfinite}. Hence, using the dominated convergence theorem,
\begin{align}
    \lim_{n \rightarrow \infty } \langle \phi_n \nabla_x g , \sigma^2 \nabla_x f \rangle_{L^2(\mu)} = \int_{\R^d} \lim_{n \rightarrow \infty} \phi_n\sigma^2\langle  \nabla_x g ,  \nabla_x f \rangle\, \ud\mu = \langle  \nabla_x g , \sigma^2 \nabla_x f \rangle_{L^2(\mu)}. \label{eq:inter54}
\end{align}
By~\eqref{eq:inter5101}, \eqref{eq:inter52}, \eqref{eq:inter53} and~\eqref{eq:inter54}, we get that 
\begin{equation}\langle \nabla_x g, \sigma^2 \nabla_x f \rangle_{L^2(\mu)} = \langle g,  (- \cL_\sigma) f \rangle_{L^2(\mu)}.\label{eq:inter5501}\end{equation}
Using part (I), $\cL_\sigma f = \bar \cL_\sigma f$, which ends the proof of the first statement in part (II) of the lemma.

 Let Assumption~\ref{ass:wpi} hold. Then for $g \in C^2_b(\R^d) \cap \mathrm{Dom}(\bar \cL_\sigma)$, we have using Lemma~\ref{lem:dirichletfinite}, $\int_{\R^d} \sigma^2 |\nabla_x g |^2 \, \ud \mu < \infty$. Hence,
\begin{align*}
    \int_{\R^d} \sigma^2 |\nabla_x g |^2 \, \ud \mu & = \langle \sigma^2\nabla_xg ,  \nabla_x (g - \mu(g)) \rangle_{L^2(\mu)} \\
    &= \langle (- \cL_\sigma)g,   g - \mu(g) \rangle_{L^2(\mu)} \leq \|\cL_\sigma g\|_{L^2(\mu)} \|g - \mu(g)\|\\
    &\leq \frac{1}{\sqrt m} \|\cL_\sigma g\|_{L^2(\mu)} \left(\int_{\R^d} \sigma^2 |\nabla_x g |^2 \, \ud \mu\right)^{1/2},
\end{align*}
where the second equality follows from the earlier part of this proof~\eqref{eq:inter5501} and the last inequality follows from Assumption~\ref{ass:wpi}. Hence, we have the inequality
    $\langle g, (- \cL_\sigma) g \rangle_{L^2(\mu)} = \int_{\R^d} \sigma^2 |\nabla_x g |^2 \, \ud \mu \leq m^{-1}\|\cL_\sigma g\|_{L^2(\mu)}^2$ for all $g \in C^2_b(\R^d) \cap \mathrm{Dom}(\bar \cL_\sigma)$, concluding the proof of part~(II).

\noindent \underline{Proof of part (III):}
Let $g \in C^1_b(\R^d)$. Then according to~\cite[Thm 1.6.6, Thm 1.7.4]{cerrai_second_2001}, there exists $\varphi \in C_b^2(\R^d)$, such that $m \varphi - \cL_\sigma \varphi = g$. 

Also, $\cL_\sigma \varphi = m \varphi - g \in L^2(\mu)$ as $\varphi \in C^2_b(\R^d)$ and $g \in C^1_b(\R^d)$. Thus, we can use Lemma~\ref{lem:compactapprox} to get a sequence $(\varphi_n)_{n \in \N} \subseteq C_c^2(\R^d)$ such that the following limit holds (note that the sequence $(\varphi_n)_{n \in \N}$ and the cutoff functions $(\phi_n)_{n \in \N}$ are not the same):
\begin{equation}\|\varphi_n - \varphi\|_{L^2(\mu)} + \|\cL_\sigma(\varphi_n - \varphi)\|_{L^2(\mu)} \rightarrow 0. \label{eq:inter5001}
\end{equation} 
This implies that $(\varphi_n)_{n \in \N} \subseteq C_c^2(\R^d)$ is Cauchy in the norm $\|\cdot \|_{L^2(\mu)} + \|\cL_\sigma (\cdot ) \|_{L^2(\mu)}$. Note that using part (I), as $\varphi_{n_1} - \varphi_{n_2} \in C_c^2(\R^d) \subseteq \mathrm{Dom}(\bar \cL_\sigma) \cap C_b^2(\R^d)$ we have that for $n_1, n_2 \in \N$, 
$$\|\varphi_{n_1} - \varphi_{n_2}\|_{L^2(\mu)} + \|\cL_\sigma(\varphi_{n_1} - \varphi_{n_2})\|_{L^2(\mu)} = \|\varphi_{n_1} - \varphi_{n_2}\|_{L^2(\mu)} + \|\bar \cL_\sigma(\varphi_{n_1} - \varphi_{n_2})\|_{L^2(\mu)}.$$
I.e., the sequence $(\varphi_n)_{n \in \N} \subseteq C_c^2(\R^d)$ is also Cauchy in the norm $\|\cdot \|_{L^2(\mu)} + \|\bar \cL_\sigma (\cdot ) \|_{L^2(\mu)}$. Using that the domain $\mathrm{Dom}(\bar \cL_\sigma)$ is complete under this norm, we have that there exists $f \in \mathrm{Dom}(\bar \cL_\sigma)$ such that $\|\varphi_{n} - f\|_{L^2(\mu)} + \|\bar \cL_\sigma(\varphi_{n} - f)\|_{L^2(\mu)} \rightarrow 0$. In particular, we have that $\|\varphi_{n} - f\|_{L^2(\mu)} \rightarrow 0$. From~\eqref{eq:inter5001}, we also have that $\|\varphi_n - \varphi\|_{L^2(\mu)} \rightarrow 0$. Hence $\varphi = f$ $\mu$-almost everywhere, and $\varphi \in \mathrm{Dom}(\bar \cL_\sigma)$. Thus,  $\varphi \in \mathrm{Dom}(\bar \cL_\sigma) \cap C_b^2(\R^d)$ and using part (I), $m \varphi - \bar \cL_\sigma \varphi = g$.
From~\cite[Ch.1 Prop.2.1]{ethier_markov_1986}, $m - \bar \cL_\sigma$ is injective. Hence, $\varphi = (m - \bar \cL_\sigma)^{-1} g \in C^2_b(\R^d) \cap \mathrm{Dom}(\bar \cL_\sigma)$.

For the proof of the second statement in part (III), let~\ref{ass:wpi} hold. Also, let $g \in C_b^1(\R^d) \cap L_0^2(\mu)$. Let $\varphi = (m - \bar \cL_\sigma)^{-1} g$. Therefore, $m \varphi = g + \bar \cL_\sigma \varphi$. Integrating both sides w.r.t. $\mu$, we get $\mu (\varphi) = 0$ as $\mu(\bar \cL_\sigma \varphi ) = \mu(g) = 0$. Using the first statement in part (III), we have that $\varphi \in C^2_b(\R^d) \cap \mathrm{Dom}(\bar \cL_\sigma) \cap L_0^2(\mu)$. Thus, using~\ref{ass:wpi} for $\varphi$ and part (II) of this lemma, we have
\begin{align*}m\|\varphi\|_{L^2(\mu)}^2 &\leq \int_{\R^d} \sigma^2 |\nabla_x \varphi|^2 \, \ud \mu = \langle (- \bar \cL_\sigma ) \varphi, \varphi \rangle_{L^2(\mu )} \\&=  \langle (m- \bar \cL_\sigma ) \varphi, \varphi \rangle_{L^2(\mu )}- m \|\varphi\|^2_{L^2(\mu)}= \langle g, \varphi \rangle_{L^2(\mu)} - m \|\varphi\|^2_{L^2(\mu)}.\end{align*}
Therefore, $2 m \|\varphi\|^2_{L^2(\mu)} \leq \langle g ,\varphi\rangle_{L^2(\mu)}$. Note that the bound $\|\varphi\|_{L^2(\mu)} \leq (2m)^{-1} \|g\|_{L^2(\mu)}$ holds if $\varphi = 0$. Hence, assume $\varphi \neq 0$. Then, using the Cauchy--Schwarz inequality, we have $2m \|\varphi \|^2_{L^2(\mu)} \leq \| g \|_{L^2(\mu)} \|\varphi\|_{L^2(\mu)}$, implying the inequality in the second statement of part (III) of the lemma.
\end{proof}

\subsection{Proof of Lemma~\ref{lem:lifts}}\label{sec:prooflifts}
We need the following definitions throughout the next proof.
Let $\gamma$ denote the standard Gaussian distribution on $\R^d$.
Let $\Sph = \{x \in \R^d : |x| = 1\}$. Let $\nu_\Sph$ denote the uniform distribution on $\Sph$. Define the $\chi_d$ distribution on $\R_+ =[0, \infty)$ by $\chi_d (A) \coloneqq \gamma(\{z \in \R^d: |z| \in A\})$ for $A$ in the Borel sigma algebra $\cB(\R_+)$. 
\begin{proof}[Proof of Lemma~\ref{lem:lifts}]
\underline{Proof of~\eqref{eq:L1}:}
Recall that $\Pi b $ is independent of the $v$ variable from the definition~\eqref{eq:Lv}.
Let $\tilde b \in C^1(\R^d)$ be such that $\tilde b (x) = \Pi b(x, v) $ for all $(x, v) \in \R^d \times \R^d$. Then, $\Ltr\Pi b (x, v) = -\Ltr^\star \Pi b (x, v)  = \langle v, \nabla_x  \tilde b (x)\rangle$. Thus,
$$\Pi \Ltr\Pi b (x, v) = \Pi\Ltr^\star \Pi b (x, v)  = \int_{\R^d}\langle v, \nabla_x  \tilde b (x)\rangle\, \ud\kappa_x(v) = 0,$$
using the fact that $\kappa_x$ is a centred Gaussian distribution with standard deviation $\sigma(x)$ for  $x \in \R^d$. 

\noindent \underline{Proof of~\eqref{eq:L2}:} Let $\tilde f \in C^2(\R^d)$ be such that $\tilde f (x) = \Pi f(x, v)$ for all $(x, v) \in \R^d \times \R^d$. 
Thus, $\Ltr \Pi f (x, v) = \langle v, \nabla_x \tilde f (x) \rangle$. Recall the definition of $\Ltr^\star$ from~\eqref{eq:Ltradjoint}. Using this, we get
\begin{align*}
        \Ltr^\star \Ltr\Pi f  (x, v) &= -\langle v, \nabla^2_x \tilde f(x) v \rangle +2 \left\langle -\frac{\nabla_x H (x, v)}{|\nabla_x H(x, v)|}, v\right\rangle_+^2 \langle   \nabla_x H(x, v) , \nabla_x \tilde f(x) \rangle,
    \end{align*}
    with convention that the second term  is zero if $\nabla_x H(x, v) = 0$. By definition~\eqref{eq:nablaH} of $\nabla_x H$, we get
\begin{equation}
\begin{aligned}
        \Ltr^\star \Ltr\Pi f  (x, v)&= -\langle v, \nabla^2_x \tilde f(x) v \rangle +2 \left\langle -\frac{\nabla_x H (x, v)}{|\nabla_x H(x, v)|}, v\right\rangle_+^2 \langle   \nabla_x U(x) , \nabla_x \tilde f(x) \rangle
        \\&+ \frac{2}{\sigma(x)} \left(d - \frac{|v|^2}{\sigma^2(x)}\right) \left\langle -\frac{\nabla_x H (x, v)}{|\nabla_x H(x, v)|}, v\right\rangle_+^2 \langle   \nabla_x \sigma(x) , \nabla_x \tilde f(x) \rangle.
        \end{aligned}
         \label{eq:inter43}
    \end{equation}

     Also, recall that $\kappa_x$ is a centred Gaussian distribution with standard deviation $\sigma(x)$.
    Then for the first term in~\eqref{eq:inter43}, we have that
    \begin{align}
        \int_{\R^d} \langle v, \nabla^2_x \tilde f(x) v \rangle \,\ud\kappa_x(v)  &= \sigma^2(x) \int_{\R^d} \langle u, \nabla^2_x \tilde f (x) u \rangle \, \ud \gamma(u)= \sigma^2(x) \Delta_x \tilde f (x), \label{eq:inter44}
    \end{align}
    using $\E[Z^T M Z] = \mathrm{Trace}(M)$ for $Z \sim \gamma = \cN(0, I_d)$.\\
    For the coefficient of the second term in~\eqref{eq:inter43}, we have that
    \begin{align}
        \int_{\R^d}  \left\langle -\frac{\nabla_x H (x, v)}{|\nabla_x H(x, v)|}, v\right\rangle_+^2 \, \ud \kappa_x(v) &= \sigma^2(x)\int_{\R^d}  \left\langle -\frac{\nabla_x H (x, \sigma(x) u)}{|\nabla_x H(x, \sigma(x)u)|}, u\right\rangle_+^2 \, \ud \gamma(u). \label{eq:inter45}
    \end{align}
     Thus, continuing from~\eqref{eq:inter45}, using the decomposition of $\gamma$ into $\chi_d$ and $\nu_\Sph$, we have
    \begin{align*}
         \int_{\R^d} & \left\langle -\frac{\nabla_x H (x, v)}{|\nabla_x H(x, v)|}, v\right\rangle_+^2 \, \ud \kappa_x(v)\\ &= \sigma^2(x)\int_{\R_+}\int_{\Sph}  \left\langle -\frac{\nabla_x H (x, \sigma(x) rz)}{|\nabla_x H(x, \sigma(x)rz)|}, rz\right\rangle_+^2 \, \ud \nu_\Sph(z) \,\ud\chi_d(r).
    \end{align*}
    Using the fact that $\nabla_xH(x, v)$ only depends on $v$ through $|v|$, $e_1 \coloneqq -\frac{\nabla_x H (x, \sigma(x) rz)}{|\nabla_x H(x, \sigma(x)rz)|}$ is a unit vector independent of $z$ in the above integral. Hence, for the innermost integral in the above, we have $\E[\langle e_1, Z_\Sph\rangle_+^2] = \E[\langle e_1, Z_\Sph\rangle^2\2{\langle e_1, Z_\Sph\rangle \geq 0}]= \E[\langle e_1, Z_\Sph\rangle^2]/2 = 1/(2d)$ for $Z_\Sph \sim \nu_\Sph$. Thus, 
    \begin{align}
        \int_{\R^d}  \left\langle -\frac{\nabla_x H (x, v)}{|\nabla_x H(x, v)|}, v\right\rangle_+^2 \, \ud \kappa_x(v) &= \frac{\sigma^2(x)}{2d}\int_{\R_+} r^2 \ud \chi_d(r) = \frac{1}{2}\sigma^2(x). \label{eq:inter46}
    \end{align}
    Similarly, for the last term in~\eqref{eq:inter43}, we have that
    \begin{align*}
        \int_{\R^d} \frac{|v|^2}{\sigma^2(x)}  \left\langle -\frac{\nabla_x H (x, v)}{|\nabla_x H(x, v)|}, v\right\rangle_+^2 \ud \kappa_x(v) 
        &= \sigma^2(x) \int_{\R_+} r^4 \int_\Sph \langle e_1, z \rangle_+^2 \,\ud\nu_\Sph(z) \, \ud \chi_d(r)\\
        &= \frac{\sigma^2(x)}{2d} \int_{\R_+} r^4 \ud \chi_d(r) = \frac{d+2}{2}\sigma^2(x). \numberthis \label{eq:inter47}
    \end{align*}
    Thus, using the expression~\eqref{eq:inter43} and plugging in~\eqref{eq:inter44},~\eqref{eq:inter46} and~\eqref{eq:inter47}, we get~\eqref{eq:L2} as
    \begin{equation*}
\begin{aligned}
        \Pi\Ltr^\star \Ltr\Pi f  (x, v)&= -\sigma^2(x) \Delta_x \tilde f(x) +\sigma^2(x) \langle   \nabla_x U(x) , \nabla_x \tilde f(x) \rangle
        \\&+ \frac{2}{\sigma(x)} \left(\frac{d}{2} - \frac{d+2}{2}\right) \sigma^2(x)  \langle   \nabla_x \sigma(x) , \nabla_x \tilde f(x) \rangle\\
        &= - \cL_\sigma \tilde f (x)= - \cL_\sigma \Pi f (x, v).
        \end{aligned}
    \end{equation*}
\noindent \underline{Proof of~\eqref{eq:L3}:}
Recall the definitions of $W_\mathrm{tr}$ and $W_{\mathrm{tr}^\star}$ from~\eqref{eq:defauxspaces}. Assume first that $g \in C_c^1(\R^d \times \R^d) \subseteq  W_\mathrm{tr}$
 and $h \in W_{\mathrm{tr}^\star}$, we have
\begin{equation}
\begin{aligned}
    \langle h, \Ltr g \rangle_{L^2(\mu_\sigma )} &= \int_{\R^d \times \R^d} h(x, v) \langle v, \nabla_x g (x, v)\rangle \ud\mu(x)\,\ud \kappa_x(v)\\
    &+ \int_{\R^d \times \R^d} h(x, v) \langle v, \nabla_x H(x,v)\rangle_+ g (x, R_{\nabla_xH(x, v)}v) \ud\mu(x)\,\ud \kappa_x(v)\\
    &- \int_{\R^d \times \R^d} h(x, v) \langle v, \nabla_x H(x,v)\rangle_+ g (x, v) \ud\mu(x)\,\ud \kappa_x(v). 
\end{aligned}\label{eq:inter48}
\end{equation}
For the first term on the right-hand side of~\eqref{eq:inter48}, we have
\begin{align*}
    \int_{\R^d \times \R^d} h(x, v) \langle v, \nabla_x g (x, v)&\rangle \ud\mu(x)\,\ud \kappa_x(v)\\ &= \frac{1}{Z_U (2 \pi)^{d/2}}\int_{\R^d \times \R^d} h(x, v) \langle v, \nabla_x g (x, v)\rangle e^{-H(x, v)} \,\ud x \,\ud v\\
    &= \frac{-1}{Z_U (2 \pi)^{d/2}}\int_{\R^d \times \R^d} \langle v,\nabla_x(h(x, v)e^{-H(x, v)})\rangle  g (x, v)  \,\ud x \,\ud v\\
    &= - \int_{\R^d \times \R^d} \langle v, \nabla_x h(x, v) \rangle g(x, v) \,  \ud\mu(x)\, \ud \kappa_x(v)\\
    &+\int_{\R^d \times \R^d} \langle v, \nabla_x H(x, v) \rangle h(x, v) g(x, v) \,  \ud\mu(x)\,\ud \kappa_x(v), \numberthis \label{eq:inter49}
\end{align*}
using integration by parts (boundary terms vanish as  $g$ has compact support) in the second equality.

For the second term in the right-hand side of~\eqref{eq:inter48}, we have that
\begin{align*}
    &\int_{\R^d \times \R^d} h(x, v) \langle v, \nabla_x H(x,v)\rangle_+ g (x, R_{\nabla_xH(x, v)}v) \ud \kappa_x(v) \, \ud\mu(x)\\ &= \int_{\R^d \times \R^d} \sigma(x)h(x, \sigma(x)u) \langle u, \nabla_x H(x,\sigma(x)u)\rangle_+ g (x, \sigma(x)R_{\nabla_xH(x, \sigma(x)u)}u) \ud \gamma(u) \, \ud\mu(x).
\end{align*}
Define
\begin{align*}
    \cW_+ (x, r, z ) &\coloneqq  h(x, \sigma(x)rz) \langle rz, \nabla_x H(x,\sigma(x)rz )\rangle_+ g (x, r\sigma(x)R_{\nabla_xH(x, \sigma(x)rz)}z),\\
    \cW_- (x, r, z ) &\coloneqq h(x, r\sigma(x)R_{\nabla_xH(x, \sigma(x)rz)}z) \langle- rz, \nabla_x H(x,\sigma(x)rz )\rangle_+ g (x, r\sigma(x)z),
\end{align*}
 for $(x, r, z) \in \R^d \times \R_+ \times \Sph$.
Decomposing $\gamma $ into $\nu_\Sph$ and $\chi_d$, we have
\begin{align*}
    \int_{\R^d \times \R^d} h(x, v) &\langle v, \nabla_x H(x,v)\rangle_+ g (x, R_{\nabla_xH(x, v)}v) \ud \kappa_x(v) \, \ud\mu(x)\\
    &= \int_{\R^d } \sigma(x)\int_{\R_+}\int_{\Sph} \cW_+(x, r, z)\, \ud \nu_\Sph(z)\, \ud \chi_d(r) \, \ud\mu(x).
\end{align*}
Recall from~\eqref{eq:nablaH} that $\nabla_x H (x, v)$ only depends on $v$ through $|v|$. 
Hence, the change of variable $z_\rb \coloneqq R_{\nabla_xH(x, \sigma(x)rz)}z$ is a measure-preserving involution on $(\Sph, \nu_\Sph)$. Applying the same on the innermost $\nu_\Sph$ integral in the above and noticing that $\langle rz, \nabla_x H(x,\sigma(x)rz )\rangle = - \langle rz_\rb, \nabla_x H(x,\sigma(x)rz_\rb )\rangle$, we get
\begin{align*}
    \int_{\R^d \times \R^d} h(x, v) &\langle v, \nabla_x H(x,v)\rangle_+ g (x, R_{\nabla_xH(x, v)}v) \ud \kappa_x(v) \, \ud\mu(x)\\
 &= \int_{\R^d } \sigma(x)\int_{\R_+}\int_{\Sph} \cW_-(x, r, z_\rb)\, \ud \nu_\Sph(z_\rb) \, \ud \chi_d(r) \, \ud\mu(x)\\
    &= \int_{\R^d \times \R^d} g(x, v) \langle - v, \nabla_x H(x,v)\rangle_+ h (x, R_{\nabla_xH(x, v)}v) \ud \kappa_x(v) \, \ud\mu(x) \numberthis \label{eq:inter50}.
\end{align*}
Plugging~\eqref{eq:inter49} and~\eqref{eq:inter50} into~\eqref{eq:inter48}, and using the fact that $\langle v, \nabla_x H(x, v) \rangle - \langle v, \nabla_x H(x, v) \rangle_+ = -\langle -v, \nabla_x H(x, v) \rangle_+$, we get~\eqref{eq:L3} for $g \in C_c^1(\R^d \times \R^d)$ and $h \in W_\mathrm{tr^\star}$. Using Lemma~\ref{lem:trdensity}, we extend~\eqref{eq:L3} to all $g \in W_\mathrm{tr}$ and $h \in W_{\mathrm{tr}^\star}$.
\end{proof}

\subsection{Proof of Lemma~\ref{lem:3.2}}\label{sec:proof3.2}
The following lemmas are needed to establish a generalised version of Bochner's formula for the weighted generator $\cL_\sigma$. Recall the operator $\nabla_x^\star F = - \mathrm{div}_x F + \langle \nabla U, F \rangle $ for a continuously differentiable $F : \R^d \rightarrow \R^d$ from Section~\ref{sec:defgen}. Define $\cL_0 g \coloneqq -\nabla_x^\star \nabla_x g$ for $g \in C^2(\R^d)$. (Note that $\cL_0$ coincides with the generator $\cL_\sigma$ when $\sigma \equiv 1$.)

\begin{lemma}[Weighted Bochner's identity] \label{lem:weightedbochner}
    For $g \in C_c^3(\R^d)$,
    \begin{equation}
        \begin{aligned}
            \|\cL_\sigma  g \|_{L^2(\mu)}^2 &=  \int_{\R^d} \sigma^4 \langle \nabla_x g , \nabla_x^2 U \nabla_x g \rangle \ud \mu  +  \int_{\R^d}\sigma^4 |\nabla_x^2 g |^2_F \ud \mu+  \int_{\R^d} \Theta(g, \sigma) \ud \mu,
        \end{aligned}
    \end{equation}
    where
    \begin{equation}
        \begin{aligned}
            \Theta(g, \sigma) &\coloneqq 4 \sigma^3 \langle \nabla_x \sigma , \nabla_x^2 g \nabla_x g \rangle
    +3\sigma^2 |\nabla_x \sigma|^2 |\nabla_x g |^2-\sigma^3 |\nabla_x g |^2 \langle \nabla_x \sigma, \nabla_x U \rangle\\
    &+ \sigma^3 |\nabla_x g |^2  \Delta_x \sigma - 2 \sigma^3 \langle \nabla_x g , \nabla_x \sigma \rangle \cL_0 g- 2 \sigma^2 \langle \nabla_x g , \nabla_x \sigma  \rangle^2\\ &- 2 \sigma^3 \langle \nabla_x g , \nabla_x^2 \sigma \nabla_x g \rangle.
        \end{aligned}
    \end{equation}
\end{lemma}

\begin{proof}
We start by establishing auxiliary identities~\eqref{eq:auxBochnerA} and~\eqref{eq:auxBochnerB}. A standard computation in~\cite[Ch.3]{bakry_analysis_2014} yields the Bochner identity:
 for $g \in C^3(\R^d)$, 
    \begin{equation}
    \label{eq:auxBochnerA}
        -\langle \nabla_x g, \nabla_x (\cL_0 g ) \rangle = - \frac{1}{2}\cL_0 (|\nabla_x g|^2) + \langle \nabla_x g, \nabla^2_x U  \nabla_x  g \rangle + |\nabla^2_x g |_F^2.
    \end{equation}
We now prove that, for any $g \in C^3(\R^d)$, it holds
\begin{equation}
\label{eq:auxBochnerB}
\begin{aligned}
    \cL_\sigma (\sigma^2 |\nabla_x g |^2 ) &= \sigma^4 \cL_0 (|\nabla_x g |^2 ) + 12 \sigma^3 \langle \nabla_x \sigma , \nabla_x^2 g \nabla_x g \rangle\\
    &+ 6 \sigma^2 |\nabla_x \sigma|^2 |\nabla_x g |^2 - 2\sigma^3 |\nabla_x g |^2 \langle \nabla_x \sigma, \nabla_x U \rangle\\
    &+ 2\sigma^3 |\nabla_x g |^2 \Delta_x \sigma.
\end{aligned}
\end{equation}
Indeed, denoting $a \coloneqq \sigma^2$ and $S \coloneqq |\nabla_x g |^2$, the quantity of interest is
    \begin{align*}
        -\nabla_x^\star a \nabla_x (aS) &= -\nabla_x^\star \left(a^2 \nabla_x S + Sa \nabla_x a \right)\\
        &= -a^2 \nabla_x^\star \nabla_x S + 3a \langle \nabla_x a, \nabla_x S \rangle
        + S a \Delta_x a + |\nabla_x a |^2S - Sa\langle \nabla_x U , \nabla_x a \rangle.
    \end{align*} 
    Substituting back $a = \sigma^2$ and $S = |\nabla_x g|^2$ implies~\eqref{eq:auxBochnerB}.
    
    Using~\eqref{eq:auxBochnerA} and~\eqref{eq:auxBochnerB}, we now prove the lemma. Note that
    \begin{align*}
         \cL_\sigma g = -\nabla_x^\star \sigma^2 \nabla_x g =  \sigma^2 \cL_0 g + 2 \sigma \langle\nabla_x \sigma, \nabla_x g  \rangle.
    \end{align*}
    Hence, 
       $$\nabla_x (\cL_\sigma g) = \sigma^2  \nabla_x (\cL_0 g) + 2 \sigma (\nabla_x \sigma) \cL_0 g + 2 \nabla_x \sigma  \langle \nabla_x g, \nabla_x \sigma\rangle + 2 \sigma   \nabla_x^2 \sigma \nabla_x g + 2 \sigma \nabla^2 _x g \nabla_x \sigma.$$
    Thus,
    \begin{align*}
        \langle \sigma^2 \nabla_x g , \nabla_x (\cL_\sigma g ) \rangle &= \sigma^4 \langle \nabla_x g, \nabla_x(\cL_0 g) \rangle + 2 \sigma^3 \langle \nabla_x g , \nabla_x \sigma \rangle \cL_0 g\\
        &+ 2 \sigma^2 \langle \nabla_x g , \nabla_x \sigma  \rangle^2 + 2 \sigma^3 \langle \nabla_x g , \nabla_x^2 \sigma \nabla_x g \rangle \\
        &+ 2 \sigma^3 \langle \nabla_x \sigma , \nabla_x^2 g \nabla_x g \rangle.
    \end{align*} For the first term in the above equality, we use Bochner's identity in~\eqref{eq:auxBochnerA} to get that
    \begin{align*}
        \langle \sigma^2 \nabla_x g , \nabla_x (\cL_\sigma g ) \rangle &= \sigma^4 \left[\frac{1}{2} \cL_0 (|\nabla_x g|^2 ) - \langle \nabla_x g , \nabla_x^2 U \nabla_x g \rangle - |\nabla_x^2 g |^2_F\right] \\ &+ 2 \sigma^3 \langle \nabla_x g , \nabla_x \sigma \rangle \cL_0 g+ 2 \sigma^2 \langle \nabla_x g , \nabla_x \sigma  \rangle^2\\ &+ 2 \sigma^3 \langle \nabla_x g , \nabla_x^2 \sigma \nabla_x g \rangle 
        + 2 \sigma^3 \langle \nabla_x \sigma , \nabla_x^2 g \nabla_x g \rangle.
    \end{align*}
For the first term in the above equality, we use identity~\eqref{eq:auxBochnerB} to get that
\begin{align*}
    \langle \sigma^2 \nabla_x g , \nabla_x (\cL_\sigma g ) \rangle &= \frac{1}{2}\cL_\sigma (\sigma^2 |\nabla_x g|^2)- 4 \sigma^3 \langle \nabla_x \sigma , \nabla_x^2 g \nabla_x g \rangle\\
    &-3\sigma^2 |\nabla_x \sigma|^2 |\nabla_x g |^2 +\sigma^3 |\nabla_x g |^2 \langle \nabla_x \sigma, \nabla_x U \rangle- \sigma^3 |\nabla_x g |^2  \Delta_x \sigma\\ &- \sigma^4 \left[\langle \nabla_x g , \nabla_x^2 U \nabla_x g \rangle + |\nabla_x^2 g |^2_F\right] \\ &+ 2 \sigma^3 \langle \nabla_x g , \nabla_x \sigma \rangle \cL_0 g+ 2 \sigma^2 \langle \nabla_x g , \nabla_x \sigma  \rangle^2+ 2 \sigma^3 \langle \nabla_x g , \nabla_x^2 \sigma \nabla_x g \rangle .
\end{align*}
Integrating both sides w.r.t. $\mu$ and using the fact that
$$\|\cL_\sigma g \|_{L^2(\mu)}^2 = -\int_{\R^d} \langle \sigma^2 \nabla_x g , \nabla_x (\cL_\sigma g ) \rangle \ \ud \mu,$$
which follows from Remark~\ref{rem:ibp}, we have the required result.
\end{proof}

\begin{lemma}\label{lem:bochnerineq}
    Let Assumptions \ref{ass:wpi}, \ref{ass:sigmabounds}, and \ref{ass:curvature}  hold. Recall the constant $K_1$ in~\eqref{eq:defconstants}.  Then 
    $$2\left(1 + \frac{K_1}{m}\right)\|\cL_\sigma  g \|_{L^2(\mu)}^2   \, \geq    \int_{\R^d}\sigma^4 |\nabla_x^2 g |^2_F\, \ud \mu\qquad\text{for all $g \in C^3_c(\R^d)$.}$$
\end{lemma}
\begin{proof}
Recall $\Theta$ from the statement of Lemma~\ref{lem:weightedbochner}.
    \begin{align*}
        -\Theta(g, \sigma) &= - 4 \sigma^3 \langle \nabla_x \sigma , \nabla_x^2 g \nabla_x g \rangle
    -3\sigma^2 |\nabla_x \sigma|^2 |\nabla_x g |^2+\sigma^3 |\nabla_x g |^2 \langle \nabla_x \sigma, \nabla_x U \rangle\\
    &- \sigma^3 |\nabla_x g |^2  \Delta_x \sigma + 2 \sigma^3 \langle \nabla_x g , \nabla_x \sigma \rangle \cL_0 g\\&+ 2 \sigma^2 \langle \nabla_x g , \nabla_x \sigma  \rangle^2+ 2 \sigma^3 \langle \nabla_x g , \nabla_x^2 \sigma \nabla_x g \rangle\\
    &\leq  4 M_D\sigma^3   |\nabla_x^2 g|_F |\nabla_x g| 
    +\sigma^3 |\nabla_x g |^2  |\nabla_x \sigma| | \nabla_x U |+ \sqrt{d} \sigma^3 |\nabla_x g |^2  |\nabla^2_x \sigma|_F
    \\&+ 2 \sigma^3 \langle \nabla_x g , \nabla_x \sigma \rangle \cL_0 g+ 2M_D^2 \sigma^2 | \nabla_x g |^2  + 2 \sigma^3  |\nabla_x g|^2 | \nabla_x^2 \sigma|_F \\
    &\leq 4 M_D\sigma^3   |\nabla_x^2 g|_F |\nabla_x g| + 2 \sigma^3 \langle \nabla_x g , \nabla_x \sigma \rangle \cL_0 g   \\
        &+ \left[(2 + \sqrt{d}) M_H + 2M_D^2 +  M_{U, \sigma } M_D\right] \sigma^2  |\nabla_x g|^2. \numberthis \label{eq:inter13}
    \end{align*}
    The last two inequalities follow from Assumption~\ref{ass:sigmabounds}, the fact that $|\Delta_x \sigma| \leq \sqrt{d}|\nabla^2_x \sigma|_F$ and $\langle v, A w\rangle \leq |v||w||A|_F$ for $v, w \in \R^d$ and $A \in \R^{d \times d}$. 
    Also, we have
    \begin{align*}
        2 \sigma^3 \langle \nabla_x g , \nabla_x \sigma \rangle \cL_0 g &= 2 \sigma^3 \langle \nabla_x g , \nabla_x \sigma \rangle (\Delta_x g -\langle \nabla_x U , \nabla_x g \rangle)\\
        &\leq 2 \sqrt{d} M_D \sigma^3 |\nabla^2_x g|_F|\nabla_x g|
        + 2 \sigma^3 | \nabla_x g |^2 |\nabla_x \sigma | |\nabla_x U |\\
        &\leq 2 \sqrt{d} M_D \sigma^3 |\nabla^2_x g|_F|\nabla_x g|
        + 2 M_{U, \sigma }M_D \sigma^2 | \nabla_x g |^2, \numberthis \label{eq:inter15}
    \end{align*}
    using Assumption~\ref{ass:sigmabounds} for the second term in the last inequality above.
    Thus, from \eqref{eq:inter13} and \eqref{eq:inter15}, we have
    \begin{align*}
        -\Theta(g, &\sigma)\leq\\  &2 M_D(\sqrt d+2 )\sigma^3   |\nabla_x^2 g|_F |\nabla_x g| + \left[(2 + \sqrt{d}) M_H + 2M_D^2+3M_{U, \sigma }M_D\right] \sigma^2  |\nabla_x g|^2,
    \end{align*}
    and using
    \begin{equation*}
        2 M_D(\sqrt{d} + 2)\sigma^3   |\nabla_x^2 g|_F |\nabla_x g| \leq 2M_D^2(\sqrt d + 2)^2 \sigma^2 |\nabla_x g|^2 + \frac{1}{2}\sigma^4|\nabla_x^2 g|_F^2,
    \end{equation*}
    we get
    \begin{equation*}
        - \Theta(g, \sigma) \leq \frac{1}{2}\sigma^4 |\nabla_x^2 g|^2_F + \left(K_1 - K\right)\sigma^2 |\nabla_x g|^2,
    \end{equation*}
    where $K_1$ is defined in~\eqref{eq:defconstants}.
    Hence, using Lemma~\ref{lem:weightedbochner},
    \begin{equation*}
    \begin{aligned}
               \frac{1}{2}\int_{\R^d}&\sigma^4 |\nabla_x^2 g |^2_F \ud \mu \leq\\ &\|\cL_\sigma  g \|_{L^2(\mu)}^2   + (K_1 - K ) \int_{\R^d} \sigma^2 |\nabla_x g |^2 \, \ud \mu - \int_{\R^d} \sigma^4 \langle \nabla_x g , \nabla_x^2 U \nabla_x g \rangle \ud \mu.
    \end{aligned}
    \end{equation*}
    Using Assumption~\ref{ass:curvature}, the last term is upper-bounded as 
    \begin{equation*}
        \frac{1}{2}\int_{\R^d}\sigma^4 |\nabla_x^2 g |^2_F \ud \mu \leq \|\cL_\sigma  g \|_{L^2(\mu)}^2   + K_1 \int_{\R^d} \sigma^2 |\nabla_x g |^2 \, \ud \mu.
    \end{equation*}
    The final result follows from Lemma~\ref{lem:TLD2}(II).
\end{proof}

\begin{proof}[Proof of Lemma~\ref{lem:3.2}]
Let $g \in C^2_b(\R^d) \cap \mathrm{Dom}(\bar \cL_\sigma)$.
    Recall that
    $$\Ltr g (x, v)= \langle v,  \nabla_x g(x)\rangle $$
    and for $f \in C^2_c(\R^d \times \R^d)$, we have that
    $$\Ltr^\star f(x, v) =-\langle v,  \nabla_x f \rangle + \langle -\nabla_x H (x, v), v\rangle_+ (f(x, R_{\nabla_x H (x, v)} v ) - f(x, v)). $$
     Thus, plugging $f (x, v) = \Ltr g(x, v) = \langle v, \nabla_x g(x) \rangle$ in the above, we get
    \begin{align*}
        \Ltr^\star \Ltr g (x, v) &= -\langle v, \nabla^2_x g(x) v \rangle +2 \left\langle -\frac{\nabla_x H (x, v)}{|\nabla_x H(x, v)|}, v\right\rangle_+^2 \langle   \nabla_x H(x, v) , \nabla_x g(x) \rangle.
    \end{align*}
    Hence, using the Cauchy--Schwarz inequality,
    \begin{align*}
        |\Ltr^\star \Ltr g (x, v)| &\leq  |v|^2 |\nabla^2_x g(x)|_F +2 |v|^2 |   \nabla_x H(x, v) || \nabla_x g(x) |.
    \end{align*}
    Therefore,
    \begin{align*}
        |\Ltr^\star \Ltr g (x, v)|^2 &\leq  2|v|^4 |\nabla^2_x g(x)|_F^2 +8 |v|^4 |   \nabla_x H(x, v) |^2| \nabla_x g(x) |^2.
    \end{align*}
    Next, we have that
    \begin{align*}
        |\nabla_x H(x, v)|^2  \leq  2|\nabla_x U (x)|^2 + 2\left(d - \frac{|v|^2}{\sigma^2(x)}  \right)^2  \left|\frac{\nabla_x \sigma(x) }{\sigma(x)}\right|^2 .
    \end{align*}
    Thus, using the fact that $v/\sigma(x)$ has the distribution of a standard $d$-dimensional Gaussian under $\kappa_x$, we have
    \begin{align*}
       \kappa_x\left(\frac{1}{\sigma(x)^4}|\Ltr^\star \Ltr g (x, v)|^2\right) &\leq 2  d(d + 2) |\nabla^2_x g(x) |_F^2 + 16 d(d+2) |\nabla_x U (x)|^2 |\nabla_x g (x)|^2\\
       &+ 32 d(d + 2) (d + 12) \left|\frac{\nabla_x \sigma(x) }{\sigma(x)}\right|^2 |\nabla_x g(x)|^2.   \end{align*}
    where we use $\E[|Z|^4] = d(d+2)$ and that $\E[|Z|^4(d - |Z|^2)^2] = 2 d (d+2) (d + 12)$, for $Z \sim \cN(0, I_d)$. Multiplying by $\sigma^4$ and integrating the above over $\mu$, and using Assumption~\ref{ass:sigmabounds} we have
    \begin{align*}
        \|\Ltr^\star \Ltr &g \|^2_{L^2(\mu_\sigma)} \leq\\& 2d(d+ 2)\biggl[ \int_{\R^d} \sigma^4 |\nabla_x^2 g|^2_F \ \ud \mu +8\left(M_{U, \sigma}^2+2 (d + 12) M_{D}^2\right) \int_{\R^d} \sigma^2 |\nabla_x g|^2 \ud \mu \biggl]. \numberthis \label{eq:inter3201}
    \end{align*}

    Using Lemma~\ref{lem:TLD2}(I), $C^3_c(\R^d)$ is a core for $(\bar \cL_\sigma, \mathrm{Dom}(\bar\cL_\sigma))$.
    Hence, there exists a sequence $(g_n)_{n \in \N} \subset C_c^3(\R^d)$ with
    \begin{equation}
        \|g_n - g \|_{L^2(\mu)} + \|\bar \cL_\sigma (g_n - g )\|_{L^2(\mu)} \rightarrow 0, \quad \text{as } n \rightarrow \infty. \label{eq:inter3202}
    \end{equation}
    For $n , k \in \N$ we have $g_n - g_k \in C^3_c(\R^d)$, so Lemma~\ref{lem:bochnerineq}, Lemma~\ref{lem:TLD2}(II) and Lemma~\ref{lem:TLD2}(I) give
    \begin{align*}
       \||\nabla_x^2g_n - \nabla_x^2g_k|_F\|^2_{L^2(\mu)}\leq  \int_{\R^d} \sigma^4 |\nabla_x^2 (g_n - g_k)|_F^2 \, \ud \mu &\leq 2\left(1 + \frac{K_1}{m}\right)\|\bar\cL_\sigma (g_n - g_k) \|^2_{L^2(\mu)},
    \end{align*}
    using the fact that $\sigma \geq 1$.
    By~\eqref{eq:inter3202}, the right-hand sides above tend to $0$ as $n , k \rightarrow \infty$. Hence $(\nabla^2_x g_n)_{n \in \N}$ is Cauchy in the complete metric space of $\R^{d \times d}$-valued functions on $\R^d$ with the metric $ \||\cdot|_F\|_{L^2( \mu)}$. Denote the corresponding limit by $\cH$.
    Thus, using~\cite[Cor.~2.32]{folland_real_1999}, there exists a subsequence $(\nabla_x^2 g_{n_j})_{j \in \N}$ such that 
    \begin{equation}
        \nabla_x^2 g_{{n_j}} \rightarrow  \cH, \quad \mu\text{-a.s.} \label{eq:inter5801_matrix}
    \end{equation}
    Note that for any compact $\cK \subset \R^d$, and any measurable $h$,
    \begin{align*}
        \int_{\cK} |h|^2 \ud x 
       \leq Z_U \exp\Bigl(\sup_{\cK} U\Bigr) \int_{\cK} |h|^2  \ud \mu,
    \end{align*}
    where we use that $U$ is continuous (so that $\sup_{\cK} U < \infty$) and that $\sigma \geq 1$.
    Hence the convergence $\nabla_x^2 g_n \rightarrow \cH$ in $L^2(\mu)$, implies the same
    convergence in $L^2(\cK, \ud x)$ for every compact $\cK \subset \R^d$.

Let $\phi \in C_c^\infty(\R^d)$ and let the compact set $\cK$ be the support of $\phi$. Fix an orthonormal basis such that $(\partial_i)_{i = 1}^d$ are the corresponding partial derivatives. Hence, using $L^2(\cK, \ud x)$ convergence of $\nabla_x^2 g_n$ as in the earlier paragraph, for $i, j \in \{1, \dots, d\}$,
  integration by parts for $g_n \in C^3_c(\R^d)$ gives
  \begin{align*}
     \int_{\R^d} \cH_{i, j} \phi \ \ud x &= \lim_{n \rightarrow \infty} \int_{\R^d} \partial_i\partial_j g_n \, \phi \ \ud x
     =  \lim_{n \rightarrow \infty}\int_{\R^d}  g_n \, \partial_j\partial_i \phi \ \ud x\\
      &=  \int_{\R^d}  g \, \partial_j\partial_i \phi \ \ud x = \int_{\R^d} \partial_i\partial_j g \, \phi \ \ud x,
  \end{align*}
  where the penultimate equality uses~\eqref{eq:inter3202} and the last equality uses $g \in C^2_b(\R^d)$. As $\phi \in C_c^\infty(\R^d)$ was arbitrary,
   $\cH = \nabla_x^2 g$ almost everywhere. Using the fact that $\mu$ is absolutely continuous w.r.t. the Lebesgue measure on $\R^d$, we have that $\cH = \nabla_x^2 g $ $\mu$-almost everywhere. Next, using the subsequential limit~\eqref{eq:inter5801_matrix}, we have that $\sigma^4|\nabla_x^2 g_{n_j}|_F^2 \rightarrow \sigma^4|\nabla_x^2 g|_F^2$ $\mu$-a.e.
   Thus, using Fatou's lemma, we have
    \begin{equation}
    \begin{aligned}
        \int_{\R^d} \sigma^4 |\nabla_x^2 g|_F^2 \ud \mu &\leq \liminf_{j \rightarrow \infty}\int_{\R^d} \sigma^4 |\nabla_x^2 g_{n_j}|_F^2 \ud \mu\\
        &\leq 2\left(1 + \frac{K_1}{m}\right)\liminf_{j \rightarrow \infty}\|\bar\cL_\sigma g_{n_j} \|^2_{L^2(\mu)} = 2\left(1 + \frac{K_1}{m}\right)\|\bar\cL_\sigma g \|^2_{L^2(\mu)}.
    \end{aligned} \label{eq:inter3901}
    \end{equation}
    In the above, the second inequality follows from Lemma~\ref{lem:bochnerineq} and Lemma~\ref{lem:TLD2}(II), and the final equality follows from~\eqref{eq:inter3202}.

    By~\eqref{eq:inter3201}, \eqref{eq:inter3901} and Lemma~\ref{lem:TLD2}(II) we get
    $\|\Ltr^\star \Ltr g \|^2_{L^2(\mu_\sigma)} \leq K_0 \|\bar \cL_\sigma g\|^2_{L^2(\mu)}$,
    $g \in C_b^2(\R^d) \cap \mathrm{Dom}(\bar \cL_\sigma)$.
\end{proof}

\subsection{Operator \texorpdfstring{$\cA$}{cA} in~\texorpdfstring{\eqref{eq:defcA}}{eq:defcA} is bounded}\label{sec:proofwtr}
The following lemma on subspaces of function spaces $W_\mathrm{tr}$, $W_{\mathrm{tr}^\star}$ and $C^n_\Pi$, $n \in \N$, defined in~\eqref{eq:defauxspaces} and~\eqref{eq:auxspacesB} is used in the proof below of Lemma~\ref{lem:Afbound}.

\begin{lemma}\label{lem:wtr}
    (I) The inclusions $C^1_c(\R^d \times \R^d) \subseteq W_\mathrm{tr} \cap W_{\mathrm{tr}^\star}$  and  $C^n_c(\R^d \times \R^d) \subseteq C_\Pi^n$, $n \in \N$, hold.\\
   (II) Let~\ref{ass:sigmabounds} hold and pick $g \in \mathrm{Dom}(\bar \cL_\sigma) \cap C^2_b(\R^d) $. Then $g \in W_\mathrm{tr} \cap W_{\mathrm{tr}^\star}$. Moreover, if ~\ref{ass:wpi} and~\ref{ass:curvature} also hold then $\Ltr g \in W_{\mathrm{tr}^\star}$.
\end{lemma}
\begin{proof}
 \underline{Proof of part (I):}   
The first statement in~(I) utilises the fact that the integrals $\|\Ltr f\|_{L^2(\mu_\sigma)}$ and $\|\Ltr^\star f\|_{L^2(\mu_\sigma)}$ are finite as $\Ltr f $ and $\Ltr^\star f$ are compactly supported continuous functions for $f \in C_c(\R^d \times \R^d)$.
The second statement in part (I) follows from a straightforward application of the dominated convergence theorem: for $f \in C_c^1(\R^d \times \R^d)$, the function $v \mapsto f(x, v)$ has compact support for each $x \in \R^d$. Hence $f \in C_\Pi$. 

Consider $f\in C_c^n(\R^d\times\R^d)$ for $n\geq1$.
We have the representation $$\Pi f(x, v) = \int_{\R^d } f(x, \sigma(x) z) \,\ud\gamma( z), \quad \text{for $(x, v) \in \R^d \times \R^d$,}$$ where $\gamma$ is the standard Gaussian distribution on $\R^d$.  Thus the gradient of $\Pi f$ in $x$ takes the form:  $$\nabla_x  (\Pi f)(x, v) = \int_{\R^d } \left[(\nabla_x f)(x, \sigma(x) z) + \nabla_x \sigma (x) \langle z , (\nabla_v f)(x, \sigma(x) z)\rangle\right] \,\ud \gamma(z).$$
Here $\nabla_v f$ is the gradient of $f$ in the $v$-component.
The differentiation under the integral is justified by the dominated convergence theorem as the integrand is compactly supported in $z\in\R^d$. Higher order derivatives follow similarly.

\noindent \underline{Proof of part (II):} Recall that $\kappa_x $ is the centered Gaussian distribution on $\R^d$ with standard deviation $\sigma(x)$ for $x \in \R^d$. For $g \in \mathrm{Dom}(\bar \cL_\sigma) \cap C_b^2(\R^d)$,
\begin{align*}
    \|\Ltr g \|_{L^2(\mu_\sigma)}^2 &= \int_{\R^d \times \R^d } \langle v, \nabla_x g(x) \rangle^2 \, \ud\kappa_x(v) \, \ud \mu(x) = \int_{\R^d} \sigma^2 |\nabla_x g|^2 \,\ud  \mu < \infty,
    \end{align*}
using Lemma~\ref{lem:dirichletfinite} (which can be used as due to Lemma~\ref{lem:TLD2} (I), we have $\|\cL_\sigma g \|_{L^2(\mu)} = \|\bar \cL_\sigma g \|_{L^2(\mu)} < \infty$ as $g \in \mathrm{Dom}(\bar \cL_\sigma)$). Also, $g \in L^2(\mu)$. Hence, we have $g \in W_\mathrm{tr}$. Note that $g$ is independent of the $v$ variable. Hence, we have that $\Ltr^\star  g = - \Ltr g$. Therefore, $g \in W_{\mathrm{tr}^\star}$ as well. 
Additionally, if~\ref{ass:wpi} and~\ref{ass:curvature} hold, then $\Ltr g \in W_{\mathrm{tr}^\star}$ due to Lemma~\ref{lem:3.2}. 
\end{proof}

\begin{proof}[Proof of Lemma~\ref{lem:Afbound}]
Recall that $\Pi$ is the velocity-averaging operator from~\eqref{eq:Lv}. 
As $f \in C_c^2(\R^d \times \R^d)$, it can be checked by explicit calculation that $\Pi\, \Ltr^\star f \in  C_c^1(\R^d)$.
Hence, from Lemma~\ref{lem:TLD2}(III) and the definition of $\cA $ in~\eqref{eq:defcA}, it follows that $\cA f \in C^2_b(\R^d) \cap \mathrm{Dom}(\bar \cL_\sigma)$. Using $\cA f = \Pi \cA f $, this also implies that $\cA f \in C^2_\Pi$.
Using Lemma~\ref{lem:wtr}, we get that $\cA f \in W_\mathrm{tr}$ and $\Ltr \cA f \in W_{\mathrm{tr}^\star}$. In particular, $\|\Ltr \cA f \|_{L^2(\mu_\sigma)} < \infty $.
Using~\eqref{eq:L3}, the fact that $\cA f = \Pi \cA f $ and~\eqref{eq:L2}, we get that $\| \Ltr \cA f\|^2_{L^2(\mu_\sigma)} = \langle \cA f , \Pi \Ltr^\star \Ltr \Pi  \cA f \rangle_{L^2(\mu_\sigma)} = \langle \cA f , (-\cL_\sigma)  \cA f \rangle_{L^2(\mu)}$. Hence
\begin{align*}
     m\|\cA f \|_{L^2(\mu_\sigma )}^2 + \|\Ltr \cA f\|_{L^2( \mu_\sigma)}^2\numberthis \label{eq:inter31} &=  m\| \cA f \|_{L^2(\mu )}^2 + \langle  \cA f, (- \cL_\sigma)  \cA f\rangle_{L^2(\mu)} \\
    &=\langle   \cA f, (m- \cL_\sigma)  \cA f\rangle_{L^2(\mu)}.
\end{align*}
Because of the definition~\eqref{eq:defcA} and Lemma~\ref{lem:TLD2}(I) it follows that $(m - \cL_\sigma) \cA f = \Pi \,\Ltr^\star f$. Hence,
\begin{align*}
    m\|\cA f \|_{L^2(\mu_\sigma )}^2 + \|\Ltr \cA f\|_{L^2( \mu_\sigma)}^2&= \langle   \cA f, \Pi \, \Ltr^\star  f\rangle_{L^2(\mu)}\\
    &= \langle  \cA f,  \Ltr^\star  f\rangle_{L^2(\mu_\sigma)}\\
    &= \langle \Ltr \cA f,   f  \rangle_{L^2(\mu_\sigma)}\numberthis\label{eq:inter16a}.
\end{align*}
where the last step follows from~\eqref{eq:L3}.
(Note that $f \in C^2_c(\R^d \times \R^d) \subseteq W_{\mathrm{tr}^\star}$. Also, as earlier stated, $\cA f \in W_\mathrm{tr}$.)
Therefore, by the Cauchy--Schwarz inequality, we have
\begin{align*}
    m\|\cA f \|_{L^2(\mu_\sigma )}^2 + \|\Ltr \cA f\|_{L^2(\mu_\sigma)}^2 = \langle \Ltr  \cA f,   f  \rangle_{L^2(\mu_\sigma)}
    \leq \| \Ltr \cA f\|_{L^2(\mu_\sigma)} \|   f \|_{L^2(\mu_\sigma)}. \numberthis \label{eq:inter32}
\end{align*}
 Thus, comparing the left-hand side of \eqref{eq:inter31} with \eqref{eq:inter32}, we get \eqref{eq:AfboundB}.
To get~\eqref{eq:AfboundA}, we use
\begin{align*}
    2 \sqrt m \|\cA f \|_{L^2(\mu_\sigma )}\|\Ltr \cA f\|_{L^2(\mu_\sigma)} &\leq m\|\cA f \|_{L^2(\mu_\sigma )}^2 + \|\Ltr \cA f\|_{L^2(\mu_\sigma)}^2 \\
    &\leq \| \Ltr \cA f\|_{L^2(\mu_\sigma)} \|   f \|_{L^2(\mu_\sigma)},
\end{align*}
which follows from~\eqref{eq:inter32}. 
The unique extension of the operators $\cA$ and $\Ltr \cA$ to the Hilbert space $L^2(\mu_\sigma)$ follows from the fact that $C^2_c(\R^d \times \R^d)$ is dense in $L^2(\mu_\sigma)$.
\end{proof}

\subsection{Proof of Theorem~\ref{thm:coercivity}}\label{sec:proofcoercivity}
The following lemmas are analogous to the intermediate steps found in the hypocoercivity approach developed in~\cite{dolbeault_hypocoercivity_2010}.

\begin{lemma}\label{lem:DMSboundA} Let~\ref{ass:wpi}, \ref{ass:sigmabounds} and~\ref{ass:curvature} hold.
    For all $f \in C^1_c(\R^d \times \R^d)$,
    \begin{equation}|\langle \cA f, \cL f \rangle_{L^2(\mu_\sigma )}| \leq \|f- \Pi f\|_{L^2(\mu_\sigma)}^2.\label{eq:DMSboundA}\end{equation}
\end{lemma}
\begin{proof}
Assume first that $f \in C^2_c(\R^d \times \R^d )$. By  Lemma~\ref{lem:wtr}(I) we have  $ f \in C^2_\Pi$. Note that $\cA f = \Pi \, \cA  f$, and from the definition of $\Pi$ in~\eqref{eq:Lv}, it follows that $\Pi$ is a self-adjoint operator on $L^2(\mu_\sigma)$.
    This gives us that
    $$|\langle \cA f, \Ltr  f \rangle_{L^2(\mu_\sigma)}|= |\langle \cA f, \Pi\, \Ltr  f \rangle_{L^2(\mu_\sigma )}| = |\langle \cA (f- \Pi f ), \Pi\, \Ltr ( f - \Pi f)  \rangle_{L^2(\mu_\sigma )}|,$$
    where the second equality uses~\eqref{eq:L1} and the identity $\cA \,\Pi f = 0$ (definition~\eqref{eq:defcA} of $\cA$ and~\eqref{eq:L1} yield $\cA g_\chi=0$, where $g_\chi(x,v)=(\Pi f)(x) \chi(v)$  for any compactly supported smooth function $\chi:\R^d\to\R$, implying that  the unique continuous extension in Lemma~\ref{lem:Afbound} of $\cA$ to $L^2(\mu_\sigma)$ satisfies $\cA\, \Pi f = 0$).

    An explicit calculation implies $\Pi \, \Ltr^\star f\in C^1_c(\R^d)$. Hence, by Lemma~\ref{lem:TLD2}(III), we have that $\cA ( f- \Pi f) = (m - \bar \cL_\sigma)^{-1}\Pi \Ltr^\star f\in C^2_b(\R^d) \cap \mathrm{Dom}(\bar\cL_\sigma)$.
     Since $f$ has compact support, $\Pi f$ has compact support in the $x$ variable and is independent of the $v$ variable. Therefore, $\|\Ltr \Pi f\|_{L^2(\mu_\sigma)}^2 = \int_{\R^d} \sigma^2 |\nabla_x \Pi f |^2  \ud \mu < \infty$, hence  $f - \Pi f \in W_\mathrm{tr}$. Also, by Lemma~\ref{lem:wtr}, $\cA (f - \Pi f ) \in W_{\mathrm{tr}^\star}$. Hence, from \eqref{eq:L3}, we have
    \begin{align*}
        |\langle \cA f, \Ltr  f \rangle_{L^2(\mu_\sigma )}|
        &= |\langle \Ltr^\star \cA (f -  \Pi f ) , f- \Pi f  \rangle_{L^2(\mu_\sigma )}|\\
        &= |\langle \Ltr \cA (f -  \Pi f ) , f- \Pi f  \rangle_{L^2(\mu_\sigma )}|\\
        &\leq \|\Ltr \cA (f- \Pi f)\|_{L^2(\mu_\sigma)} \|f- \Pi f\|_{L^2(\mu_\sigma)},
    \end{align*}
    where the second equality uses $\Ltr \cA f = - \Ltr^\star \cA f$ which follows from the fact that $\cA f $ is independent of the velocity $v$. 
    By Lemma~\ref{lem:Afbound}, the operator $\Ltr \cA$ has a unique extension to $L^2(\mu_\sigma)$ as a bounded operator with norm at most one. Applying this fact to the first factor in the upper bound above yields
    $   |\langle \cA f, \Ltr  f \rangle_{L^2(\mu_\sigma )}|\leq \|f- \Pi f\|_{L^2( \mu_\sigma)}^2 .$
 Using $\cA f = \Pi\, \cA f$ as earlier, we have
    $$\langle \cA f,  \Lv  f \rangle_{L^2(\mu_\sigma )} = \langle \cA f,  \Pi \, \Lv  f \rangle_{L^2(\mu_\sigma )} = 0,$$
    using the fact that $\Pi\, \Lv f  =  \Pi(\Pi - I) f = 0$. Hence~\eqref{eq:DMSboundA} holds for $f \in C^2_c(\R^d \times \R^d)$.
    
      Note that $\cA$ is bounded on $L^2(\mu_\sigma)$ due to Lemma~\ref{lem:Afbound}, hence the map $h,g \rightarrow \langle h, \cA g \rangle_{L^2(\mu_\sigma)}$ on $L^2(\mu_\sigma ) \times L^2(\mu_\sigma )$ is continuous. Also, by~\cite[Thm~1.3.2]{hormander_analysis_2003}, the space $C^2_c(\R^d \times \R^d)$ is dense in $C^1_c(\R^d \times \R^d)$ in the norm $ \|\cdot \|_\infty + \|\nabla_x (\cdot) \|_\infty$ . Note that the left and right-hand sides of~\eqref{eq:DMSboundA} involve only the first derivative of $f$. Hence, we get the result for $f \in C^1_c(\R^d \times \R^d)$.
    \end{proof}

\begin{lemma}\label{lem:DMSboundB} Let~\ref{ass:wpi},  \ref{ass:sigmabounds} and~\ref{ass:curvature} hold. Let the constant $K_0$ be as in~\eqref{eq:defconstants}.
    Then, for all $f \in C^1_c(\R^d \times \R^d)$ with $\mu_\sigma(f) = 0$,
    \begin{equation}\langle \cA\,\cL f, f \rangle_{L^2( \mu_\sigma )} \geq \frac{1}{2} \|\Pi f \|_{L^2(\mu_\sigma)}^2 - \left(\frac{3}{2}\sqrt{K_0} + \frac{\lambda_\rr}{2\sqrt m }\right) \|\Pi f \|_{L^2( \mu_\sigma)} \| f - \Pi f \|_{L^2( \mu_\sigma )} .\label{eq:DMSboundB}\end{equation}
\end{lemma}
\begin{proof}
Assume first that $f \in C_c^3(\R^d \times \R^d)\subseteq C_\Pi^3$ (see definition of $C_\Pi^3$ in~\eqref{eq:auxspacesB}) with $\mu_\sigma(f) = 0$.
Note that $\cA f$ does not depend on $v$ by definition~\eqref{eq:defcA}. Hence $ \cA f = \Pi \, \cA  f$, implying that $\Pi \cA=\cA$ on $L^2(\mu_\sigma)$, since $\Pi,\cA$ are bounded operators and $C_c^3(\R^d \times \R^d)$ is dense in $L^2(\mu_\sigma)$. From the definition in~\eqref{eq:Lv} it follows that $\Pi$ is a self-adjoint operator on $L^2(\mu_\sigma)$. Hence, the left-hand side of the inequality in the statement of the lemma can be decomposed as
    \begin{align*}
         \langle \cA\,\cL f, f \rangle_{L^2( \mu_\sigma )}& =\lambda_\rr \langle \cA\, \Lv f, \Pi f \rangle_{L^2( \mu_\sigma )} + \langle \cA\, \Ltr f, \Pi f \rangle_{L^2( \mu_\sigma )}
        \numberthis \label{eq:inter39}
    \end{align*}
    For the first term in the above, by the Cauchy--Schwarz inequality and Lemma~\ref{lem:Afbound}, we have
    \begin{align*}
        |\langle \cA\, \Lv f, \Pi f \rangle_{L^2( \mu_\sigma )} |
        &\leq \|\cA\, \Lv f \|_{L^2(\mu_\sigma)} \|\Pi f \|_{L^2(\mu_\sigma)}
        \leq \frac{1}{2\sqrt m} \|\Lv f \|_{L^2(\mu_\sigma)} \|\Pi f \|_{L^2(\mu_\sigma)}\\
        &= \frac{1}{2\sqrt m} \| f - \Pi f  \|_{L^2(\mu_\sigma)} \|\Pi f \|_{L^2(\mu_\sigma)}. \numberthis \label{eq:inter40}
    \end{align*}
   
    For the second term on the right-hand side in~\eqref{eq:inter39}, we need the following fact proved below.

    \noindent\textbf{Claim}.    For $f \in C^3_c(\R^d \times \R^d)$, it holds $(m - \bar \cL_\sigma)^{-1} \Pi f \in C^{2}_b(\R^d) \cap \mathrm{Dom}(\bar \cL_\sigma)$,  $\Ltr (m - \bar \cL_\sigma)^{-1} \Pi f \in W_{\mathrm{tr}^\star}$ and $$\langle \cA\, \Ltr f, \Pi f \rangle_{L^2(\mu_\sigma )} = \langle f,  \Ltr^\star\Ltr (m - \bar \cL_\sigma)^{-1} \Pi f \rangle_{L^2(\mu_\sigma )}.$$
   
    Using the Claim, we have the following decomposition of the second term in~\eqref{eq:inter39}:
    \begin{equation}
        \langle \cA\, \Ltr f, \Pi f \rangle_{L^2(\mu_\sigma )} =  A + \langle \Pi f , \Ltr^\star\Ltr (m - \bar \cL_\sigma)^{-1} \Pi f \rangle_{L^2(\mu_\sigma )},
        \label{eq:inter5801}
    \end{equation}
    where $A\coloneqq \langle f - \Pi f,  \Ltr^\star\Ltr (m - \bar \cL_\sigma)^{-1} \Pi f \rangle_{L^2(\mu_\sigma )}$.
    By the Cauchy--Schwarz inequality and Lemma~\ref{lem:3.2}, the term $A$ in~\eqref{eq:inter5801} satisfies
    \begin{align*}
        |A| & \leq \|\Ltr^\star \Ltr (m - \bar\cL_\sigma)^{-1} \Pi f \|_{L^2(\mu_\sigma)} \|f - \Pi f \|_{L^2(\mu_\sigma)}\\
        &\leq \sqrt{K_0} \|(-\bar \cL_\sigma) (m - \bar \cL_\sigma)^{-1}  \Pi f \|_{L^2( \mu )} \|f - \Pi f \|_{L^2(\mu_\sigma)}.
    \end{align*}
    Using the triangle inequality, the right-hand side is upper bounded as follows:
    \begin{align*}
        |A| &\leq \sqrt{K_0} \|(m-\bar \cL_\sigma -m) (m - \bar \cL_\sigma)^{-1}  \Pi f \|_{L^2( \mu )} \|f - \Pi f \|_{L^2(\mu_\sigma)}\\
        &= \sqrt{K_0} \|\Pi f  -m (m - \bar \cL_\sigma)^{-1}  \Pi f \|_{L^2( \mu )} \|f - \Pi f \|_{L^2(\mu_\sigma)}\\
        &\leq \sqrt{K_0} \left(\|\Pi f\|_{L^2(\mu)}  +m \|(m - \bar \cL_\sigma)^{-1}  \Pi f \|_{L^2( \mu )}\right) \|f - \Pi f \|_{L^2(\mu_\sigma)}.
    \end{align*}
    Since  $\mu(\Pi f ) = \mu_\sigma(f ) = 0$,  we obtain $\|(m - \bar \cL_\sigma)^{-1}  \Pi f \|_{L^2( \mu )} \leq (2m)^{-1} \|\Pi f \|_{L^2(\mu)}$ by Lemma~\ref{lem:TLD2}(III). Thus 
    \begin{align}
        |\langle f - \Pi f,  \Ltr^\star\Ltr (m - \bar \cL_\sigma)^{-1} \Pi f \rangle_{L^2(\mu_\sigma )}| \leq \frac{3}{2}\sqrt{K_0} \|\Pi f \|_{L^2(\mu)} \|f - \Pi f \|_{L^2(\mu_\sigma)}. \label{eq:inter41}
    \end{align}

   
     For the second summand in~\eqref{eq:inter5801}, by Lemma~\ref{lem:TLD2}(I) we have  $\Pi f \in C^3_c(\R^d) \subseteq \mathrm{Dom}(\bar \cL_\sigma) \cap C_b^2(\R^d )$. Moreover, due to Lemma~\ref{lem:TLD2}(III), it holds $(m - \bar \cL_\sigma)^{-1} \Pi f \in C^2_b(\R^d) \cap \mathrm{Dom}(\bar \cL_\sigma)$. Due to the fact that $\Pi (m - \bar \cL_\sigma)^{-1} \Pi f = (m - \bar \cL_\sigma)^{-1} \Pi f$, we have  $(m - \bar \cL_\sigma)^{-1} \Pi f \in C^2_\Pi$.  By~\eqref{eq:L2} we get
     \begin{align*}
         \langle \Pi f , \Ltr^\star\Ltr (m - \bar \cL_\sigma)^{-1} \Pi f \rangle_{L^2(\mu_\sigma )} &= 
         \langle \Pi f ,\Pi \Ltr^\star\Ltr \Pi (m - \bar \cL_\sigma)^{-1} \Pi f \rangle_{L^2(\mu)} \\
        &=  \langle \Pi f , (- \cL_\sigma) (m - \bar \cL_\sigma)^{-1} \Pi f \rangle_{L^2(\mu_\sigma )}\\
         &= \langle \Pi f , (- \bar \cL_\sigma) (m - \bar \cL_\sigma)^{-1} \Pi f \rangle_{L^2(\mu_\sigma )},
     \end{align*}
     where the last equality holds by Lemma~\ref{lem:TLD2}(I). This can be further decomposed as
     \begin{align*}
         \langle \Pi f , \Ltr^\star\Ltr (m - \bar \cL_\sigma)^{-1} \Pi f \rangle_{L^2(\mu_\sigma )}
         &=\langle \Pi f , (m- \bar \cL_\sigma-m) (m - \bar \cL_\sigma)^{-1} \Pi f \rangle_{L^2(\mu_\sigma )}\\
         &= \langle \Pi f ,  \Pi f \rangle_{L^2( \mu )} - m \langle (m - \bar \cL_\sigma)^{-1}   \Pi f ,  \Pi f \rangle_{L^2( \mu )}.
     \end{align*}
    Recall that $\mu(\Pi f ) = \mu_\sigma(f ) = 0$.
    Hence, using the Cauchy--Schwarz inequality and then Lemma~\ref{lem:TLD2}(III) (which implies that $\|(m - \bar \cL_\sigma)^{-1}  \Pi f \|_{L^2( \mu )} \leq (2m)^{-1} \|\Pi f \|_{L^2(\mu)}$), we have the lower bound
    \begin{align*}
        \langle \Pi f , \Ltr^\star\Ltr (m - \bar \cL_\sigma)^{-1} \Pi f \rangle_{L^2(\mu_\sigma )}&\geq \|\Pi f \|_{L^2(\mu)}^2 - m \|\Pi f\|_{L^2(\mu)} \|(m - \bar \cL_\sigma)^{-1} \Pi f \|_{L^2(\mu)}\\
        &\geq \|\Pi f \|_{L^2(\mu)}^2 - \frac{1}{2} \|\Pi f\|_{L^2(\mu)}^2 = \frac{1}{2} \|\Pi f \|_{L^2( \mu_\sigma )}^2. \numberthis \label{eq:inter42}
    \end{align*}
    Using \eqref{eq:inter39}, \eqref{eq:inter40}, \eqref{eq:inter5801}, \eqref{eq:inter41} and \eqref{eq:inter42}, we get the result~\eqref{eq:DMSboundB} for $f \in C_c^3(\R^d \times \R^d)$ with $\mu_\sigma(f) = 0$.
    
      Note that the left and right-hand sides of~\eqref{eq:DMSboundB} involve only the bounded operator $\cA$ and the first derivative of $f$. Hence, since $C^3_c(\R^d \times \R^d) \cap L_0^2(\mu_\sigma)$ is  dense in $C^1_c(\R^d \times \R^d) \cap L_0^2(\mu_\sigma)$,  the same argument as in the end of the proof of Lemma~\ref{lem:DMSboundA} extends the inequality in~\eqref{eq:DMSboundB} to all $f \in C^1_c(\R^d \times \R^d) \cap L_0^2(\mu_\sigma)$. It now only remains to prove the Claim.
\smallskip

\noindent \underline{Proof of \textbf{Claim}}.  For $f \in C^3_c(\R^d \times \R^d)$, set  $\varphi\coloneqq (m - \bar \cL_\sigma)^{-1} \Pi f$. Note that $\Pi f\in C_c^3(\R^d)$ and hence, by Lemma~\ref{lem:TLD2}(III), we have $\varphi \in C^{2}_b(\R^d) \cap \mathrm{Dom}(\bar \cL_\sigma)$.
Lemma~\ref{lem:wtr}(II) implies
$\varphi\in W_{\mathrm{tr}}$ and $\Ltr\varphi\in W_{\mathrm{tr}^\star}$.
The Claim follows if we prove
\begin{equation}
\label{eq:intermediate_step}
    \bigl\langle
        \cA\Ltr f,\Pi f
    \bigr\rangle_{L^2(\mu_\sigma)}
    =
    \bigl\langle
        \Ltr f,\Ltr\varphi
    \bigr\rangle_{L^2(\mu_\sigma)}.
\end{equation}
Indeed, since $f\in W_{\mathrm{tr}}$ and
$\Ltr\varphi\in W_{\mathrm{tr}^\star}$, an application
of~\eqref{eq:L3} to the right-hand side of~\eqref{eq:intermediate_step} gives
\begin{align*}
    \bigl\langle
        \cA\Ltr f,\Pi f
    \bigr\rangle_{L^2(\mu_\sigma)}
    =
    \bigl\langle
        f,\Ltr^\star\Ltr\varphi
    \bigr\rangle_{L^2(\mu_\sigma)}
    =
    \bigl\langle
        f,
        \Ltr^\star\Ltr
        (m-\bar\cL_\sigma)^{-1}\Pi f
    \bigr\rangle_{L^2(\mu_\sigma)}.
\end{align*}

We now establish~\eqref{eq:intermediate_step} via an approximation argument. 
Although $\Ltr f$ need not be differentiable,
it belongs to $L^2(\mu_\sigma)$ and has compact support, since
$f$ has compact support, the bounce rate is locally bounded and
the reflection $R_{\nabla_xH(x,v)}$ preserves the magnitude of the velocity $|R_{\nabla_x H(x, v)} v|=|v|$. Pick a sequence
$(h_n)_{n\in\N}\subset C_c^\infty(\R^d\times\R^d)$ such that
$h_n\to \Ltr f$
in $L^2(\mu_\sigma)$.
For any $n\in\N$, define a smooth vector field
\[
    M_n(x)
    \coloneqq
    \int_{\R^d}v\,h_n(x,v)\,\ud\kappa_x( v)=\sigma(x)\int_{\R^d}
    z\,h_n\bigl(x,\sigma(x)z\bigr)\,\ud\gamma(z),
    \quad x\in\R^d,
\]
with compact support
(here $\gamma$ denotes the Gaussian measure on $\R^d$).

For any $g\in C_c^1(\R^d)$, identified with a function on
$\R^d\times\R^d$ that is independent of $v$, by~\eqref{eq:Ltradjoint}  we have 
$ \Ltr g = \langle  v,\nabla_x g\rangle$. Hence
identity~\eqref{eq:L3} and Remark~\ref{rem:ibp}
yield
\begin{align*}
    \bigl\langle \Pi\Ltr^\star h_n,g\bigr\rangle_{L^2(\mu)}
    &=
    \bigl\langle \Ltr^\star h_n,g\bigr\rangle_{L^2(\mu_\sigma)}
    =
    \bigl\langle h_n,\Ltr g\bigr\rangle_{L^2(\mu_\sigma)}
    =
    \int_{\R^d}
    \bigl\langle M_n(x),\nabla_xg(x)\bigr\rangle\,\mu(\ud x)
    \\
    &=
    \bigl\langle \nabla_x^\star M_n,g\bigr\rangle_{L^2(\mu)}.
\end{align*}
Thus
$\Pi\Ltr^\star h_n=\nabla_x^\star M_n$
in $L^2(\mu)$,
so that $\Pi\Ltr^\star h_n$ admits a representative in $C_c^1(\R^d)$.
Since $h_n\in C_c^\infty(\R^d\times\R^d)$, the definition
of $\cA$ in~\eqref{eq:defcA} and Lemma~\ref{lem:TLD2}(III) imply
$\cA h_n \in C_b^2(\R^d)\cap\mathrm{Dom}(\bar\cL_\sigma)$.
Moreover, by Lemma~\ref{lem:TLD2}(I), we have 
$\Pi f = (m-\bar\cL_\sigma)\varphi=(m-\cL_\sigma)\varphi$,
where the second equality follows from Lemma~\ref{lem:TLD2}(I) since $\varphi \in C^{2}_b(\R^d) \cap \mathrm{Dom}(\bar \cL_\sigma)$.
As both $\cA h_n$ and $\varphi$ are in
$C_b^2(\R^d)\cap\mathrm{Dom}(\bar\cL_\sigma)$, the symmetry identity
in Lemma~\ref{lem:TLD2}(II) yields
\begin{align*}
    \bigl\langle \cA h_n,\Pi f\bigr\rangle_{L^2(\mu_\sigma)}
    &=\bigl\langle\cA h_n,(m-\cL_\sigma)\varphi\bigr\rangle_{L^2(\mu)}=
    \bigl\langle(m-\cL_\sigma)\cA h_n,\varphi\bigr\rangle_{L^2(\mu)}
    \\&=\bigl\langle\Pi\Ltr^\star h_n,\varphi\bigr\rangle_{L^2(\mu)}
    =
    \bigl\langle
        \Ltr^\star h_n,\varphi
    \bigr\rangle_{L^2(\mu_\sigma)}
    =
    \bigl\langle
        h_n,\Ltr\varphi
    \bigr\rangle_{L^2(\mu_\sigma)},
\end{align*}
where the final equality follows from~\eqref{eq:L3}.

By Lemma~\ref{lem:Afbound}, $\cA$ extends to a bounded operator on
$L^2(\mu_\sigma)$. Thus, letting $n\to\infty$ and using
$h_n\to\Ltr f$ in $L^2(\mu_\sigma)$, we obtain~\eqref{eq:intermediate_step}.
\end{proof}

\begin{proof}[Proof of Theorem~\ref{thm:coercivity}]
      We first note that 
    \begin{equation*}
        \langle f, (- \cL) f \rangle_{L^2(\mu_\sigma )}  = \langle f, (- \Ltr ) f \rangle_{L^2(\mu_\sigma)} + \lambda_\rr \|f - \Pi f \|_{L^2(\mu_\sigma)}^2, \quad \text{for  $f \in C_c^1(\R^d \times \R^d)$,}
    \end{equation*}
    by using the definition~\eqref{eq:Lv} of $\Lv$. 
    Note that in the above, the term $$\langle f, (- \Ltr ) f \rangle_{L^2(\mu_\sigma)} =\frac{1}{2} \langle f, -(\Ltr + \Ltr^\star )f\rangle_{L^2(\mu_\sigma)},$$ using~\eqref{eq:L3} (which is applicable here due to Lemma~\ref{lem:wtr}(I)). Also, using the expressions for $\Ltr$ and $\Ltr^\star$ as given in~\eqref{eq:Ltradjoint}, we have that $-(\Ltr + \Ltr^\star ) f = |\langle \nabla_x H (x, v ), v \rangle| (f(x, v) - f(x, R_{\nabla_x H(x, v)} v ))$. Denoting $A_0\coloneqq \int_{\R^d \times \R^d} |\langle\nabla_x H (x, v), v\rangle| f^2(x, v )\,\ud \mu_\sigma(x, v)$, by an application of the Cauchy--Schwarz inequality, we obtain
    \begin{align*}
       2 \langle f, (- \Ltr ) &f \rangle_{L^2(\mu_\sigma)} = A_0 - \int_{\R^d \times \R^d} |\langle\nabla_x H (x, v), v\rangle| f(x, v ) f(x, R_{\nabla_x H (x, v)} v )\,\ud \mu_\sigma(x, v)\\
        &\geq A_0- A_0^{1/2}\left(\int_{\R^d \times \R^d} |\langle\nabla_x H (x, v), v\rangle| f^2(x, R_{\nabla_x H (x, v)} v ) \,\ud \mu_\sigma(x, v)\right)^{1/2}.
    \end{align*}
    By a change of variables of the form found in the proof of Lemma~\ref{lem:lifts}\eqref{eq:L3}, we have that 
    \begin{align*}
    \int_{\R^d \times \R^d} |\langle\nabla_x H (x, v), v\rangle| f^2(x, &R_{\nabla_x H (x, v)} v )\,\ud \mu_\sigma(x, v)=  A_0.
    \end{align*}
    Hence, we have that $\langle f, (- \Ltr ) f \rangle_{L^2(\mu_\sigma)} \geq 0$.
    Therefore,
     \begin{equation}
        \langle f, (- \cL) f \rangle_{L^2(\mu_\sigma )}  \geq \lambda_\rr \|f - \Pi f \|_{L^2(\mu_\sigma)}^2, \quad \text{for  $f \in C_c^1(\R^d \times \R^d)$.}\label{eq:poincarev}
    \end{equation}
   
    Following from Lemmas~\ref{lem:DMSboundA},~\ref{lem:DMSboundB} and equations \eqref{eq:defcE} and \eqref{eq:poincarev}, we have
    \begin{align*}
        \cE_{\delta_\ast}(f) \geq &(\lambda_\rr - \delta_\ast)\|f - \Pi f \|_{L^2(\mu_\sigma)}^2 \\
        & - \delta_\ast \left(2\sqrt{K_0} + \frac{\lambda_\rr }{2\sqrt m}\right) \|\Pi f \|_{L^2(\mu_\sigma)} \|f - \Pi f \|_{L^2(\mu_\sigma)} + \frac{\delta_\ast}{2} \|\Pi f \|_{L^2(\mu_\sigma )}^2.
    \end{align*}
    We have that
    \begin{align*}
    \left(2\sqrt{ K_0} + \frac{\lambda_\rr }{2\sqrt m}\right) \|\Pi f \|_{L^2(\mu_\sigma)} &\|f - \Pi f \|_{L^2(\mu_\sigma)} \leq \\&\frac{1}{4} \|\Pi f \|_{L^2( \mu_\sigma)}^2 + \left(2\sqrt{K_0} + \frac{\lambda_\rr }{2\sqrt m}\right)^2 \|f - \Pi f \|_{L^2(\mu_\sigma)}^2.\end{align*}
    Therefore, 
    \begin{align*}
        \cE_{\delta_\ast}(f) \geq \left(\lambda_\rr - \delta_\ast - {\delta_\ast}\left(2\sqrt{K_0} + \frac{\lambda_\rr }{2\sqrt m}\right)^2\right)\|f - \Pi f \|_{L^2(\mu_\sigma)}^2  + \frac{\delta_\ast}{4} \|\Pi f \|_{L^2(\mu_\sigma )}^2.
    \end{align*}
    The coefficient of the first term in the above is greater than $\delta_\ast/4$ as we have that
    $$ \delta_\ast \leq \frac{\lambda_\rr}{{5}/{4} + \left(2\sqrt{K_0} + {\lambda_\rr }/{2\sqrt m}\right)^2}.$$
    Hence, we get \eqref{eq:coercivity}.
    The second part of the theorem follows from using the Cauchy--Schwarz inequality and Lemma~\ref{lem:Afbound}:
    \begin{align*}
        |\cV_{\delta_\ast} (f) - \frac{1}{2} \|f\|^2_{L^2(\mu_\sigma)}| &= \delta_\ast |\langle \cA f, f\rangle_{L^2(\mu_\sigma)}| \leq \frac{\delta_\ast}{2\sqrt m }  \|f\|^2_{L^2( \mu_\sigma)}.
    \end{align*}
    The constant $\delta_\ast/(2\sqrt{m})$ is bounded by 1/4.
\end{proof}

\subsection{Proof of Theorem~\ref{thm:maincnvg}}\label{sec:proofdecay}
Recall the Markov semigroup $\cP = (\cP_t)_{t \in \R_+}$ and the generator $(\bar \cL, \mathrm{Dom}(\bar \cL))$ as defined in Section~\ref{sec:invariance}, corresponding to the piecewise deterministic Markov process given by~\eqref{eq:TBP}.
For the proof of Theorem~\ref{thm:maincnvg}, we need the following definition. For $\delta \in (0, \infty)$, define the functional
\begin{equation}\begin{aligned}
    \bar \cE_\delta(f) &\coloneqq \langle f, (- \bar \cL ) f \rangle_{L^2(\mu_\sigma )} + \delta \langle \cA \, \bar \cL  f, f \rangle_{L^2(\mu_\sigma )} + \delta \langle\cA f, \bar\cL f \rangle_{L^2( \mu_\sigma)}, \qquad  \text{for } f \in \mathrm{Dom}(\bar \cL).
\end{aligned}\label{eq:defbarcE}\end{equation}
Throughout the next proof, we use $L_0^2(\mu_\sigma) \coloneqq \{f \in L^2(\mu_\sigma) : \mu_\sigma(f) = 0\}$.

\begin{proof}[Proof of Theorem~\ref{thm:maincnvg}]
  Assume first that $f \in \mathrm{Dom}(\bar \cL) $ with $ \mu_\sigma (f) = 0 $. Hence, $\cP_tf \in \mathrm{Dom}(\bar \cL)$ (using \cite[Ch.1 Prop.1.5]{ethier_markov_1986}) with $\mu_\sigma(\cP_t f) = 0$ for  $t \geq 0$.
  Let $\delta_\ast$ be as in Theorem~\ref{thm:coercivity}.
  We have that $\cA : L^2(\mu_\sigma) \rightarrow L^2(\mu_\sigma)$ is bounded (due to Lemma~\ref{lem:Afbound}) and hence the map $(g, h) \mapsto \langle g, \cA h \rangle_{L^2(\mu_\sigma)}$ is continuous for $(g, h) \in L^2(\mu_\sigma) \times L^2(\mu_\sigma)$. 
  Thus, using $\frac{\ud}{\ud t} \cP_t f = \bar \cL \cP_t f $, we have that 
\begin{equation}
\frac{\ud}{\ud t}\cV_{\delta_\ast}(\cP_t f) = -\bar \cE_{\delta_\ast}(\cP_tf).\label{eq:inter28}
\end{equation}
From Theorem~\ref{thm:stationaritycore}(II), $C^1_c(\R^d \times \R^d)$ is a core of $\bar \cL$. Fix $t \in \R_+$. Hence, there exists a sequence $(f_i)_{i \in \N} \subset C^1_c(\R^d \times \R^d)$ such that $\|f_i - \cP_t f\|_{L^2(\mu_\sigma)} + \| \bar \cL( f_i - \cP_t f )\|_{L^2(\mu_\sigma)} \rightarrow 0 $. In the next few steps, we construct a compactly supported approximating sequence that is also $\mu_\sigma$-mean zero.

Fix a function $\psi \in C_c^1(\R^d \times \R^d)$ such that $\mu_\sigma(\psi) = 1$. Then $g_i \coloneqq f_i - \mu_\sigma(f_i)\psi \in C_c^1(\R^d \times \R^d) \cap L_0^2(\mu_\sigma)$.
It remains to show that the sequence $(g_i)_{i \in \N}$ also approximates $\cP_t f$. We have that
\begin{align*}
    \| g_i - f_i\|_{L^2(\mu_\sigma)} + \| \bar \cL (g_i - f_i )\|_{L^2(\mu_\sigma)} = |\mu_\sigma(f_i)| \left(\|\psi\|_{L^2(\mu_\sigma)} + \|\bar \cL \psi\|_{{L^2(\mu_\sigma)}}\right).
\end{align*}
In the above, the right-hand side is finite as $\psi \in \mathrm{Dom}(\bar \cL)$. Hence, $\|f_i - \cP_t f\|_{L^2(\mu_\sigma)} + \| \bar \cL( f_i - \cP_t f )\|_{L^2(\mu_\sigma)} \rightarrow 0 $ implies that $\mu_\sigma(f_i) \rightarrow 0$ as $\mu_\sigma(\cP_t f) = 0$. Thus,
\begin{align*}
    \| g_i - \cP_t f\|_{L^2(\mu_\sigma)} + \| \bar \cL (g_i - \cP_t f )\|_{L^2(\mu_\sigma)}& \leq |\mu_\sigma(f_i)| \left(\|\psi\|_{L^2(\mu_\sigma)} + \|\bar \cL \psi\|_{{L^2(\mu_\sigma)}}\right)\\
    & + \|f_i - \cP_t f\|_{L^2(\mu_\sigma)} + \| \bar \cL( f_i - \cP_t f )\|_{L^2(\mu_\sigma)} \rightarrow 0.
\end{align*}

As stated earlier, the map $(g, h) \mapsto \langle g, \cA h \rangle_{L^2(\mu_\sigma)} $ for $(g, h) \in L^2(\mu_\sigma) \times L^2(\mu_\sigma)$ is continuous and hence we have that
\begin{equation}
    |\bar\cE_{\delta_\ast} (g_i) - \bar \cE_{\delta_\ast}(\cP_t f)| + |\cV_{\delta_\ast} (g_i) - \cV_{\delta_\ast} (\cP_t f)| \rightarrow 0 \quad \text{as } i \rightarrow \infty. \label{eq:inter29}
\end{equation}
      Thus, for every $\epsilon > 0$, there exists $n \in \N$ such that for all $i \in \N$, with $i > n$,
      \begin{align*}
          \frac{\ud}{\ud t}\cV_{\delta_\ast}(\cP_t f) &= -\bar \cE_{\delta_\ast}(\cP_tf) \leq- \bar \cE_{\delta_\ast}(g_i) + \epsilon\\&= -  \cE_{\delta_\ast}(g_i) + \epsilon\leq -\frac{\delta_\ast}{3}\cV_{\delta_\ast}(g_i) + \epsilon\leq -\frac{\delta_\ast}{3}\cV_{\delta_\ast}(\cP_t f) + 2 \epsilon.
      \end{align*}
      In the above, the first step is merely~\eqref{eq:inter28}. The second step follows from~\eqref{eq:inter29}. The third step follows from Theorem~\ref{thm:stationaritycore}(I). The fourth step follows from Theorem~\ref{thm:coercivity}, and the fifth step follows from~\eqref{eq:inter29}. Since $\epsilon > 0$ is arbitrary, we have the inequality $\frac{\ud}{\ud t}\cV_{\delta_\ast}(\cP_t f) \leq  -\frac{\delta_\ast}{3}\cV_{\delta_\ast}(\cP_t f)$.

      By an application of Gronwall's inequality, we have that
      $$\cV_{\delta_\ast}(\cP_tf) \leq \exp\left(-\frac{\delta_\ast t}{3}  \right) \cV_{\delta_\ast}(f),$$
      and~\eqref{eq:Vbounds} gives us the $L^2$-decay result~\eqref{eq:L2decayresult} for $f \in \mathrm{Dom}(\bar \cL) \cap L_0^2(\mu_\sigma) $.

      We have that $C_c^1(\R^d \times \R^d)$ is dense in $L^2(\mu_\sigma)$.
      Using the same construction as earlier, we can see that $ C_c^1(\R^d \times \R^d ) \cap L_0^2(\mu_\sigma)$ is dense in $L_0^2(\mu_\sigma)$. Hence, $\mathrm{Dom}(\bar \cL) \cap L_0^2(\mu_\sigma)$ is dense in $L_0^2( \mu_\sigma)$. Due to Lemma~\ref{lem:Pcontraction}, the semigroup $(\cP_t)_{t \in \R_+}$ extends uniquely to a contraction on $L^2(\mu_\sigma)$ . Thus, we get~\eqref{eq:L2decayresult} for $f \in L_0^2(\mu_\sigma)$.
\end{proof}

\section{Open problems and future directions}\label{sec:conclusions}
The open problems discussed in this section fall into two broad categories: applying~\eqref{eq:TBP} without prior knowledge of the tail behaviour of $\mu$ (Section~\ref{subsec:(A)}), and  extending the weighted hypocoercivity methods developed in this paper to  related models (Section~\ref{subsec:(B)}).

\subsection{Adaptive tempering}
\label{subsec:(A)}
How should the tempering function $\sigma$ be chosen when the tail behaviour of the target is anisotropic or unknown? This question is important because, although the process~\eqref{eq:TBP} is well defined and non-explosive for every choice of $\sigma$ (see Lemma~\ref{lem:explosiveness}), our convergence theory requires the weighted Poincar\'e, regularity, and curvature conditions in Assumptions~\ref{ass:wpi}--\ref{ass:curvature}. The weighted Poincar\'e condition is monotone in the weight: if it holds for some $\sigma_1:\R^d\to[1,\infty)$, then it also holds, with the same constant, for every $\sigma_2\geq\sigma_1$. In contrast, the regularity and curvature conditions need not be preserved when the weight is increased; in particular, the curvature condition may hold for $\sigma_1$ but fail for some $\sigma_2\geq\sigma_1$. We refer to this phenomenon as \textit{over-tempering}. It therefore remains open to construct a tempering function that requires less prior knowledge of the tails (e.g., if tails are polynomial or sub-exponential), adapts to different tail behaviours in different directions, while satisfying Assumptions~\ref{ass:wpi}--\ref{ass:curvature}. Possible approaches include constructing $\sigma$ from local information about $U$ or learning it using neural networks or other modern adaptive methods. The latter is similar in spirit to the tail-index estimation strategies of~\cite{HillTailestimation,pmlrTailAdaptiveLaszkiewicz22a}. More broadly, these questions are in the spirit of the open problems discussed in~\cite[Sec.~4.3]{powervaskedis2025robustness}. 

Motivated by Remark~\ref{rem:temperingheuristics} and the heuristic in~\eqref{eq:cons_heuristic}, a natural adaptive tempering function should satisfy
\[
    \sigma(x)\asymp 1+\frac{1}{|\nabla U(x)|}
    \qquad \text{on a complement of a compact set.}
\]
We conjecture that for a  set of targets with a  finite positive moment (e.g., the class of targets described in Remark~\ref{rem:tails}(I)-(II)), such a choice of $\sigma$ satisfies~\ref{ass:wpi}--\ref{ass:curvature}. 
The polynomial and sub-exponential targets in Corollary~\ref{cor:mixingtimeB} provide evidence for this conjecture, as do the convex potentials considered in~\cite{deligiannidis_exponential_2019}, for which the untempered choice $\sigma\equiv1$ yields exponential ergodicity. Moreover, the above growth condition on $\sigma$ is typically compatible with the regularity and curvature conditions in Assumptions~\ref{ass:sigmabounds} and~\ref{ass:curvature}. 
The main remaining step in proving the conjecture is therefore to establish the weighted Poincar\'e inequality in Assumption~\ref{ass:wpi} for such a choice of $\sigma$.

The conjecture is not true for distributions that do not have any polynomial moments. Indeed, 
Lemma~\ref{lem:inf_oscillation}(ii) provides a one-dimensional example (with only logarithmic moments), such that  
 the Weighted Poincar\'e inequality in~\ref{ass:wpi} 
 does not hold for any $\sigma$ with a bounded first derivative. Hence, Assumption~\ref{ass:sigmabounds} would need to be weakened to accommodate such extremely heavy-tailed target distributions.

For targets with tails that are thinner than Gaussian, the modelling approach of this paper naturally suggests a tempering function $\sigma$ that tends to zero at infinity: $\sigma(x)\to0$ as $|x|\to\infty$. This choice of tempering is motivated by the heuristic argument in Remark~\ref{rem:temperingheuristics} (see~\eqref{eq:cons_heuristic}), as well as in~\cite[Sec.~3.2]{deligiannidis_exponential_2019}.
The process~\eqref{eq:TBP} with such tempering would be a lift of a tempered Langevin diffusion with an asymptotically vanishing volatility. Moreover, for an appropriately chosen such $\sigma$, the drift of~\eqref{eq:templangevin} would be globally Lipschitz. 


\subsection{Tempered underdamped Langevin dynamics and other non-reversible lifts of the process in\texorpdfstring{~\eqref{eq:templangevin}}{ (TLD)}}
\label{subsec:(B)}
Can the weighted-Poincar\'e approach developed in this paper be extended to other non-reversible lifts (in the sense of~\cite[Def.~1]{eberle_non-reversible_2026}) of the tempered Langevin diffusion~\eqref{eq:templangevin}?  As shown in~\cite{eberle_convergence_2025}, underdamped Langevin dynamics (ULD), the Bouncy Particle Sampler (BPS), and the Zig-Zag sampler (ZZS) are among the lifts of the untempered overdamped Langevin diffusion. Their tempered counterparts can be constructed analogously: just as~\eqref{eq:TBP} is a tempered version of BPS, one can construct tempered versions of ULD and ZZS that lift~\eqref{eq:templangevin}.

For example, consider the $\R^d\times\R^d$-valued process $(Q_t,P_t)_{t\in\R_+}$ satisfying
\begin{equation}
    \begin{aligned}
        \ud Q_t &= P_t\,\ud t,\\
        \ud P_t
        &=
        \left[
            -\sigma^2(Q_t)\nabla_x U(Q_t)
            -
            \left(
                d-2-\frac{|P_t|^2}{\sigma^2(Q_t)}
            \right)
            \sigma(Q_t)\nabla_x\sigma(Q_t)
        \right]\ud t\\
        &\quad
        -\frac{\lambda}{\sigma^2(Q_t)}P_t\,\ud t
        +\sqrt{2\lambda}\,\ud W_t,
    \end{aligned}
    \label{eq:udlangevintemp}\tag{T-ULD}
\end{equation}
where $\lambda>0$ is the friction parameter and $(W_t)_{t \in \R_+}$ is a standard $d$-dimensional Brownian motion. The process~\eqref{eq:udlangevintemp} admits $\mu_\sigma$, defined in~\eqref{eq:mu_sigma_def}, as its stationary distribution and is a second-order lift of~\eqref{eq:templangevin}; it may therefore be regarded as a tempered analogue of ULD. Under Assumptions~\ref{ass:wpi}--\ref{ass:curvature}, suitable adaptations of the arguments developed in this paper should yield non-asymptotic exponential convergence rates for~\eqref{eq:udlangevintemp}. 
The process in~\eqref{eq:udlangevintemp} is particularly interesting
because it is driven by Brownian motion, involves only additive noise,
and, under the natural stationary scaling
$|P|=O(\sigma(Q))$, has drift of at most linear order in the tails for
polynomially heavy-tailed targets, while nevertheless being expected to
be exponentially ergodic.

More generally, we expect that minor modifications of our methods will also provide non-asymptotic exponential convergence rates for the Speed up Zig-Zag (SUZZ) process~\cite{g_vasdekis_speed_2023} when targeting heavy-tailed distributions.


\begin{appendix}

\section{Auxiliary Lemmas}\label{sec:appendixaux}
Throughout this section, we use the smooth cutoff functions $\phi$ and $\phi_n$ (for $n \in \N$) as defined in Section~\ref{sec:proofTLD2}.
For the proof of Lemma~\ref{lem:TLD2}(III), we need the following finiteness lemma, which follows from the monotone convergence theorem.

\begin{lemma}\label{lem:dirichletfinite} Let~\ref{ass:sigmabounds} hold. For $g \in C^2_b(\R^d )$ satisfying $\|\cL_\sigma g \|_{L^2(\mu)} < \infty$, we have  $$\int_{\R^d} \sigma^2 |\nabla_x g|^2 \, \ud \mu < \infty  .$$
\end{lemma}
\begin{proof}
    Set $E_n \coloneqq \int_{\R^d} \phi_n^2 \sigma^2 |\nabla_x g|^2 \, \ud \mu  = \langle\phi_n^2 \sigma^2 \nabla_x g, \nabla_x g \rangle_{L^2(\mu)}$. Recall $\nabla_x^\star $ from~\ref{sec:defgen}. By an integration-by-parts argument, we have
    \begin{align*}
        E_n = \langle \nabla^\star (\phi_n^2 \sigma^2 \nabla_x g),  g \rangle_{L^2(\mu)} =\int_{\R^d} g\nabla_x^\star(\phi_n^2 \sigma^2 \nabla_x g)  \, \ud \mu .
    \end{align*}
    Boundary terms are zero in the above as $\phi_n$ has compact support. Thus, using the definition of $\nabla_x^\star F = -\mathrm{div}_x F + \langle \nabla_x U, F\rangle$ for a differentiable vector field $F : \R^d \rightarrow \R^d$, we have 
    \begin{align*}
        E_n &=\int_{\R^d}\phi_n^2 g(\nabla_x^\star\sigma^2 \nabla_x g ) \, \ud \mu - 2\int_{\R^d} g\, \phi_n\sigma^2 \langle \nabla_x \phi_n, \nabla_x g \rangle\, \ud \mu\\
        &\leq \|g\|_{L^2(\mu)} \|\cL_\sigma g \|_{L^2(\mu)} + 2\int_{\R^d} g^2 \sigma^2 |\nabla_x \phi_n|^2\, \ud \mu + \frac{1}{2}\int_{\R^d} \phi_n^2 \sigma^2 |\nabla_x g|^2\, \ud \mu , \numberthis \label{eq:inter4201}
    \end{align*}
    using the Cauchy--Schwarz inequality for the first term and Young's inequality ($2ab \leq a^2/2 + 2b^2$ for $a, b \in \R$) for the second term. The second term in~\eqref{eq:inter4201} can be bounded as $|g \sigma \nabla_x \phi_n| \leq \|g\|_\infty M_{\phi, 1}$, using~\eqref{eq:inter4202}.
     Thus, $\int_{\R^d} g^2 \sigma^2 |\nabla_x \phi_n|^2\, \ud \mu < \infty$. By assumption, $\cL_\sigma g \in L^2(\mu)$. Continuing from~\eqref{eq:inter4201}, we have
    $$E_n \leq 2\|g\|_{L^2(\mu)} \| \cL_\sigma g \|_{L^2(\mu)} + 4\| g\|_\infty^2 M_{\phi, 1}^2 =:M_g < \infty \quad \text{for all } n \in \N\setminus \{0\}.$$
    Note that $\phi_n^2 \sigma^2 |\nabla_x g |^2 (x) \uparrow \sigma^2 |\nabla_x g |^2 (x) $ monotonically, for all $x \in \R^d$. Thus, using the monotone convergence theorem, $$\int_{\R^d } \sigma^2 |\nabla_x g|^2 \, \ud \mu = \int_{\R^d}\lim_{ n \rightarrow \infty} \phi_n^2 \sigma^2 |\nabla_x g|^2 \, \ud \mu = \lim_{n \rightarrow \infty } E_n \leq M_g < \infty.$$
\end{proof}
We need the following approximation lemma for the proofs of Lemma~\ref{lem:TLD2} parts (III) and (II). 
\begin{lemma}\label{lem:compactapprox} Let~\ref{ass:sigmabounds} hold. Let $l \in \{2, 3 \ldots\} \cup \{\infty\}$.
    For $g \in C^l_b(\R^d)$ satisfying $\|\cL_\sigma g\|_{L^2(\mu)} < \infty$, set $g_n \coloneqq  \phi_n g \in C_c^l(\R^d)$ for $n \in \N \setminus \{0\}$. Then $\|g_n -g\|_{L^2(\mu)} + \|\cL_\sigma (g_n - g )\|_{L^2(\mu)} \rightarrow 0$ as $n \rightarrow \infty$.  
\end{lemma}
\begin{proof}
    Using dominated convergence and the fact that $|g_n - g| \leq |g| \in L^2(\mu)$ we have
\begin{equation}
    \|g_n - g\|_{L^2(\mu)} \rightarrow 0 \quad \text{as } n \rightarrow \infty. \label{eq:inter4401}
\end{equation}
Also, using the product rule for differentiation, we get
\begin{equation}
    \cL_\sigma (g_n - g) = (\phi_n - 1) \cL_\sigma g + g \cL_\sigma \phi_n + 2 \sigma^2 \langle \nabla_x \phi_n , \nabla_x g \rangle. \label{eq:inter4501}
\end{equation}
By assumption, $\cL_\sigma g \in L^2(\mu)$ . 
Therefore, using $|(\phi_n - 1) \cL_\sigma g| \leq |\cL_\sigma g|$ and the dominated convergence theorem, we get that
\begin{equation}
    \|(\phi_n - 1) \cL_\sigma g\|_{L^2(\mu)} \rightarrow 0 \quad \text{as } n \rightarrow \infty. \label{eq:inter4601}
\end{equation}

For the second term in~\eqref{eq:inter4501}, we have
\begin{align*}g(x) \cL_\sigma &\phi_n (x) =\\&g(x)\left[\sigma^2(x) \Delta_x \phi_n(x) - \sigma^2(x)\langle\nabla_x U(x), \nabla_x \phi_n(x) \rangle + 2\sigma(x) \langle \nabla_x\sigma(x), \nabla_x \phi_n (x) \rangle\right],\end{align*}
for $x \in \R^d$.
Thus, using~\eqref{eq:inter4202}, \eqref{eq:inter4302} and~\ref{ass:sigmabounds}, we get
$$|g \cL_\sigma \phi_n | \leq \2{|x| \geq n} \|g\|_\infty (M_{\phi, 2} + M_{U, \sigma} M_{\phi, 1} + 2 M_D M_{\phi, 1}) $$
which is uniformly bounded by $\|g\|_\infty (M_{\phi, 2} + M_{U, \sigma} M_{\phi, 1} + 2 M_D M_{\phi, 1}) < \infty$. Hence, by the dominated convergence theorem, 
\begin{equation}
    \lim_{n \rightarrow \infty}\|g \cL_\sigma \phi_n \|_{L^2(\mu)} = 0. \label{eq:inter4801}
\end{equation}

For the third term in~\eqref{eq:inter4501}, we have the following, using~\eqref{eq:inter4202}. $$|\sigma^2 \langle \nabla_x \phi_n, \nabla_x g \rangle| \leq  \2{|x| \geq n} \sigma^2 |\nabla_x \phi_n|\,|\nabla_x g | \leq \2{|x| \geq n} M_{\phi, 1}\sigma |\nabla_x g | .$$  Also, we have the uniform bound $\2{|x| \geq n} M_{\phi, 1}\sigma |\nabla_x g | \leq M_{\phi, 1} \sigma |\nabla_x g |$. Also, we have that $\sigma |\nabla_x g | \in L^2(\mu)$ using Lemma~\ref{lem:dirichletfinite}. Thus, using the dominated convergence theorem, 
\begin{align*}
    \lim_{n \rightarrow \infty} \| \sigma^2 \langle \nabla_x \phi_n , \nabla_x g \rangle\|_{L^2(\mu)}^2 &\leq M_{\phi, 1}^2\lim_{n \rightarrow \infty} \int_{\R^d} \2{|x| \geq n}  \sigma^2 |\nabla_x g |^2 \, \ud \mu\\ &= M_{\phi, 1}^2 \int_{\R^d} \lim_{n \rightarrow \infty} \2{|x| \geq n}  \sigma^2 |\nabla_x g |^2 \, \ud \mu =  0. \numberthis \label{eq:inter4701}
\end{align*}
Using~\eqref{eq:inter4401}, \eqref{eq:inter4501}, \eqref{eq:inter4601}, \eqref{eq:inter4801} and~\eqref{eq:inter4701}, we have that 
$\|g_n -g\|_{L^2(\mu)} + \|\cL_\sigma (g_n - g )\|_{L^2(\mu)} \rightarrow 0$ as $n \rightarrow \infty$. 
\end{proof}

The following lemma is required for the proof of Lemma~\ref{lem:lifts} part~\eqref{eq:L3}.
\begin{lemma}\label{lem:trdensity}
    The set $C_c^1(\R^d \times \R^d)$ is dense in $W_\mathrm{tr}$ with respect to the norm
    $\|\cdot\|_{L^2(\mu_\sigma)} + \|\Ltr (\cdot)\|_{L^2(\mu_\sigma)}$, and dense in
    $W_{\mathrm{tr}^\star}$ with respect to the norm
    $\|\cdot\|_{L^2(\mu_\sigma)} + \|\Ltr^\star (\cdot)\|_{L^2(\mu_\sigma)}$.
\end{lemma}
\begin{proof}
    For $R > 0$, set
    $$\phi_R (x, v) \coloneqq \phi(|x|/R)\, \phi(|v|/R), \qquad (x,v) \in \R^d \times \R^d .$$
    As  $\phi$ is constant near the origin, $\phi_R \in C^\infty_c(\R^d \times \R^d)$, and
    $|\nabla_x \phi_R| \leq \|\phi'\|_\infty / R$.

    Let $g \in W_\mathrm{tr}$ and put $g_R \coloneqq \phi_R\, g \in C_c^1(\R^d \times \R^d)$.
    Since $|R_u v| = |v|$ for all $u , v\in \R^d$, we have
    $\phi_R(x , R_{\nabla_x H(x,v)}v) = \phi_R(x, v)$, so we have
    \begin{equation}
        \Ltr g_R = \phi_R \Ltr g + g \langle v , \nabla_x \phi_R \rangle . \label{eq:inter51}
    \end{equation}
    For the second term in~\eqref{eq:inter51}, note that $\nabla_x \phi_R$ vanishes unless
    $|x| \geq R$, and that $\phi(|v|/R) = 0$ unless $|v| \leq 2R$; hence
    $$|\langle v , \nabla_x \phi_R (x,v)\rangle | \leq \frac{|v|}{R}\|\phi'\|_\infty \2{|v| \leq 2R}\,\2{|x| \geq R}
        \leq 2 \|\phi'\|_\infty \2{|x| \geq R}. $$
    Therefore, using dominated convergence ($g^2 \2{ |x |\geq R} \leq g^2$ and $g \in L^2(\mu_\sigma)$), we have
    $$\| g \langle v , \nabla_x \phi_R\rangle \|_{L^2(\mu_\sigma)}^2
        \leq 4 \|\phi'\|_\infty^2 \int_{\{|x| \geq R\} \times \R^d} g^2 \,\ud \mu_\sigma \rightarrow 0 \quad \text{as } R \rightarrow \infty $$
    For the first term in~\eqref{eq:inter51}, $0 \leq \phi_R \leq 1$ and
    $\phi_R \uparrow 1$ pointwise, so $\|\phi_R \Ltr g - \Ltr g\|_{L^2(\mu_\sigma)}  \rightarrow 0$ by
    dominated convergence, using $\Ltr g \in L^2(\mu_\sigma)$. The same argument gives
    $\|g_R - g\|_{L^2(\mu_\sigma)} \rightarrow 0$. 
    Putting the above limits together, we get
    $$\|g_R - g \|_{L^2(\mu_\sigma)} + \|\Ltr (g_R - g)\|_{L^2(\mu_\sigma)} \rightarrow 0 \quad \text{as } R \rightarrow \infty.$$
    For $h\in W_{\mathrm{tr}^\star}$, set
$h_R:=\phi_Rh$.
    The argument for $W_{\mathrm{tr}^\star}$ is identical, using
    $\Ltr^\star h_R = \phi_R \Ltr^\star h - h \langle v , \nabla_x \phi_R\rangle$ in place
    of~\eqref{eq:inter51}.
\end{proof}

\section{Mixing times and the \texorpdfstring{$\chi^2$}{Chi2}-divergence}\label{sec:appendixchi}
The $\chi^2$-\textit{divergence} between probability measures $\P$ and $\Q$ on a Polish space $H$ is defined as
\begin{equation}
    \chi^2 (\P \| \Q ) \coloneqq  \E_\Q\left[\left(\frac{\ud\P}{\ud\Q} - 1\right)^2\right] \qquad\text{if $\P \ll \Q$ and $\chi^2 (\P \| \Q ) \coloneqq \infty$ otherwise,} \label{eq:chi2def}
\end{equation}
    where $\frac{\ud\P}{\ud\Q}$ is the Radon-Nikodym derivative of $\P$ with respect to $\Q$ when absolute continuity $\P \ll \Q$ holds.
We also use the following notation. Let $f, g: \R \rightarrow \R$. 
 We say that $f =  O (g)$ if $\limsup_{x \rightarrow + \infty} |f(x)|/|g(x)| < \infty$.

\begin{proof}[Proof of Corollary~\ref{cor:mixingtime}] As shown in Theorem~\ref{thm:maincnvg}, we have that the Markov semigroup $\left(\cP_{t}\right)_{t \in \R_+}$ satisfies the decay $\|\cP_t f \|_{L^2(\mu_\sigma)}^2 \leq 3 e^{-\nu t } \|f\|_{L^2(\mu_\sigma)}^2$ for all $f \in L^2_0(\mu_\sigma)$, with $\nu > 0$ given by Theorem~\ref{thm:maincnvg}. Let $\mu_0$ be some initial distribution on $\R^d \times \R^d$. Assume that $ \mu_0 \ll  \mu_\sigma$. Then, we have 
\begin{align*}
     \chi^2( \mu_0  \cP_t \|   \mu_\sigma) &= \left\|\frac{\ud  (\mu_0 \cP_t)}{\ud \mu_\sigma} - 1\right\|^2_{L^2(\mu_\sigma)} = \sup_{ f \in L^2( \mu_\sigma) : \|f\|_{L^2(\mu_\sigma)} = 1} \left\langle \frac{\ud (\mu_0 \cP_t)}{\ud\mu_\sigma} - 1, f \right\rangle_{L^2(\mu_\sigma)}^2,
 \end{align*}
 where the first equality follows from the definition \eqref{eq:chi2def} and the second equality is a standard variational representation of the $L^2$-norm.
 Hence, using the Cauchy--Schwarz inequaliy, we obtain
     \begin{align*}
     \chi^2( \mu_0  \cP_t \|   \mu_\sigma) 
    &\leq \left\|\frac{\ud \mu_0}{\ud \mu_\sigma} - 1\right\|_{L^2( \mu_\sigma)}^2 \left( \sup_{f \in L^2_0 ( \mu_\sigma) : \|f\|_{L^2( \mu_\sigma)} = 1} \|  \cP_t f - \mu_\sigma(\cP_t f )\|_{L^2(\mu_\sigma)}^2\right)\\
  & \leq 3\chi^2 ( \mu_0 \| \mu_\sigma) e^{-\nu t }, 
 \end{align*}  
Thus, Corollary~\ref{cor:mixingtime} is implied by Theorem~\ref{thm:maincnvg}.
\end{proof}

\section{Failure of~\ref{ass:wpi} under an infinite-oscillation perturbation}
\label{app:failure_wpi}
Consider on $\R$ the probability measure
$\mu_1(\ud x)=Z_1^{-1}\langle x\rangle^{-4}\,\ud x$, $\langle x\rangle=(1+x^2)^{1/2}$,
which satisfies the weighted Poincar\'e inequality
\[
    m_1\,\mathrm{Var}_{\mu_1}(g)
    \leq
    \int_{\R}\langle x\rangle^2|g'(x)|^2\,\mu_1(\ud x)\qquad\text{for some $m_1\in(0,\infty)$ and all $g\in C_b^1(\R)$.}
\]
Define the perturbed probability measure
$\mu_i(\ud x)
    \coloneqq
    e^{W_i(x)}\mu_1(\ud x)$,
where
$$W_i(x)
    \coloneqq
    \log\left(Z_1/Z_i\right)
    +i\log\langle x\rangle
    -2\log\log(e+\langle x\rangle),\quad i=2,3,$$ for an appropriate constant $Z_i>0$. 
    
    Denote by $\rho_i$ the Lebesgue density of  $\mu_i (\ud x) = \rho_i (x) \, \ud x$. For $a, b : \R \rightarrow \R_+$, we define the following notation.
    \begin{align*}a (x) &\lesssim b(x) \text{ for $x \in J \subseteq \R$ if there exists $c > 0$ such that $a(x) \leq c b(x)$ for all $x \in J$};\\
    a(x) &\asymp b(x) \text{ for $x \in J$ if $a(x) \lesssim b(x)$ and $b(x) \lesssim a(x)$ for $x \in J$}.
    \end{align*}
     We have the relations 
     \begin{align*}
     \rho_2(x) \asymp \frac{1}{(|x| \log |x|)^2} \quad \text{and} \quad \rho_3(x) \asymp \frac{1}{|x|(\log |x| )^2}, \quad  \text{ for $|x| > 1$,}
     \end{align*}
    implying the following tail behaviours.
    \begin{align*}
        \mu_2([R, \infty)) &= \mu_2((-\infty, -R]) \asymp 1/(R (\log R)^2), \quad \text{for $R > 1$};\\
        \mu_3([R, \infty)) &= \mu_3((-\infty, -R]) \asymp 1/\log R, \quad \text{for $R > 1$.}
    \end{align*}

\begin{lemma}
\label{lem:inf_oscillation}
(i) For $i=2,3$, the perturbation $W_i$ has infinite
oscillation
$\operatorname{osc}_{\mu_1}(W_i)=\infty$ 
and the potential $U_i \coloneqq - \log \rho_i$ of $\mu_i$ satisfies $|U_i'(x)|\asymp 1/|x|$ on $\R\setminus(-1,1)$.\\
(ii)   The probability measure $\mu_3$ on $\R$ does \textbf{not} satisfy the weighted Poincar\'e inequality in~\ref{ass:wpi} for any tempering function $\sigma:\R\to[1,\infty)$ with bounded first derivative $\sup_{x\in\R}\sigma'(x)<\infty$.\\
(iii)   The probability measure $\mu_2$ satisfies~\ref{ass:wpi} with a tempering function $\sigma(x)\asymp |x|\asymp 1/|U'_2(x)|$
on $\R\setminus(-1,1)$.
\end{lemma}

\begin{proof}
(i) Since $W_i(x)\to\infty$ as $|x|\to\infty$ for $i=2,3$, we conclude $\operatorname{osc}_{\mu_1}(W_i)=\infty$. 
Since the potential $U_i$ of $\mu_i$ satisfies $U_i(x)=-W_i(x)+4\log\langle x\rangle+\textrm{const}$, a direct differentiation implies $|U_i'(x)|\asymp 1/|x|$ as $|x|\to\infty$.

(ii) Recall the tail behaviour of $\mu_3$ and density $\rho_3$ mentioned before this lemma. We now prove that $\mu_3$ does not satisfy~\ref{ass:wpi}. 
Assuming $\|\sigma'\|_\infty<M<\infty$, we have $\sigma(x)\leq \sigma(0)+M|x|$ and hence $\sigma^2(x)\lesssim x^2$.
For any large $R$, pick $g_R\in C_b^1(\R)$ such that 
$0\leq g_R\leq1$, $g_R\equiv0$ on $(-\infty,R]$, $g_R\equiv1$ on $[2R,\infty)$ and $\|g_R'\|_\infty\leq 2/R$. Since $0\leq \mu_3(g_R)\leq\mu_3([R,\infty))\to0$ as $R\to\infty$, we have that $\mu_3(g_R) \leq 1/2$ for $R$ large enough. Hence, 
\[
    \mathrm{Var}_{\mu_3}(g_R)
    \geq
    \int_{[2R,\infty)}(1-\mu_3(g_R))^2\,\mu_3(\ud x)\geq \frac{1}{4}\mu_3([2R,\infty))
    \gtrsim
    \frac{1}{\log R} .
\]
Moreover, since $g_R'=0$ on $\R\setminus[R,2R]$, we have 
\[
    \int_{\R}\sigma^2(x)|g_R'(x)|^2\,\mu_2(\ud x)
    = \frac{4}{R^2}\int_R^{2R} \sigma^2(x) \rho_3(x)\ud x \lesssim
    \frac{1}{R^2}
    \int_R^{2R}
    \frac{x^2}
    {x(\log x)^2}\,\ud x 
    \lesssim
    \frac{1}{(\log R)^2}.
\]
Consequently,
\[
    \frac{\mathrm{Var}_{\mu_3}(g_R)}
    {\displaystyle
      \int_{\R}\sigma^2(x)|g_R'(x)|^2\,\mu_3(\ud x)}
    \gtrsim \log R
    \to\infty\quad\text{as $R\to\infty$.}
\]
Hence, there does not exist a constant $m\in(0,\infty)$ satisfying the weighted Poincar\'e inequality in~\ref{ass:wpi}. 

(iii) 
Recall the tail behaviour of $\mu_2$ and density $\rho_2$ mentioned before this lemma.
Since $\sigma^2(x)\asymp |x|^2$, we have
\begin{equation}
\label{eq:v_symp}
\int_0^x\frac{\ud t}{\sigma^2(t)\rho_2(t)}\asymp \int_0^x(\log t)^2\ud t\asymp x(\log x)^2\quad\text{as $x\to\infty$.} 
\end{equation}
The corresponding expression for the left tail holds by symmetry. 
For any $g\in C_b^1(\R)$ we have 
\begin{align}
\nonumber
\mathrm{Var}_{\mu_2}(g)&=\inf_{c\in\R}\int_\R (g(x)-c)^2\mu_2(\ud x) \leq 
\int_\R (g(x)-g(0))^2
\rho_2(x)\ud x\\
&=\int_\R \left(\int_0^x g'(t)\ud t\right)^2
\rho_2(x)\ud x \leq \int_\R \left(\int_0^x |g'(t)|\ud t\right)^2
\rho_2(x)\ud x.
\label{eq:Var_uppr_bound}
\end{align}

By Muckenhoupt's weighted Hardy inequality~\cite[Thm~1]{MR311856} (with $p=p'=2$), for measurable  $u,v:(0,\infty)\to(0,\infty)$  there exists a finite constant 
$C_H$  such that
\begin{equation}
    \int_0^\infty
    \left(\int_0^x f(t)\,\ud t\right)^2
    u(x)\,\ud x
    \leq
    C_H
    \int_0^\infty f(x)^2v(x)\,\ud x,
    \label{eq:muckenhoupt_hardy}
\end{equation}
for all measurable $f:(0,\infty)\to[0,\infty)$
if and only if 
\[
    B
    \coloneqq
    \sup_{r>0}
    \left(\int_r^\infty u(x)\,\ud x\right)
    \left(\int_0^r\frac{\ud x}{v(x)}\right)<\infty.
\]
Moreover, if $B<\infty$, the smallest  constant $C_H$ in~\eqref{eq:muckenhoupt_hardy}
satisfies
$B\leq C_H\leq 4B$.

For any $g\in C_b^1(\R)$, set $f\coloneqq |g'|$ and define $u(x)\coloneqq \rho_2(x)$ and $v(x)\coloneqq \sigma^2(x)\rho_2(x)$. Since $\int_r^\infty u(x)\ud x=\mu_2([r,\infty))\asymp 1/(r(\log r)^2)$ and $\int_0^r\ud x/v(x)\asymp x(\log x)^2$ (by~\eqref{eq:v_symp}), we conclude that in this case $B<\infty$. Moreover, by~\eqref{eq:muckenhoupt_hardy} we get
$$
\int_0^\infty \left(\int_0^x |g'(t)|\ud t\right)^2
\mu_2(\ud x)\leq 4B \int_0^\infty \sigma^2(x) f(x)^2 \mu_2(\ud x).
$$
Since $\rho_2$ and $\sigma^2$ are symmetric, the same inequality holds on the interval $(-\infty,0)$.
By the inequality in~\eqref{eq:Var_uppr_bound}, the weighted Poincar\'e inequality in~\ref{ass:wpi} holds for $\mu_2$ and any  $\sigma$ with linear growth at infinity.
\end{proof}

\end{appendix}

\bibliographystyle{amsplain}
\bibliography{references}

\providecommand{\bysame}{\leavevmode\hbox to3em{\hrulefill}\thinspace}
\providecommand{\MR}{\relax\ifhmode\unskip\space\fi MR }
\providecommand{\MRhref}[2]{%
  \href{http://www.ams.org/mathscinet-getitem?mr=#1}{#2}
}
\providecommand{\href}[2]{#2}
\begin{thebibliography}{10}

\bibitem{christophe_andrieu_subgeometric_2021}
Christophe Andrieu, Paul Dobson, and Andi~Q. Wang, \emph{Subgeometric hypocoercivity for piecewise-deterministic {Markov} process {Monte} {Carlo} methods}, Electronic Journal of Probability \textbf{26} (2021), 1--26.

\bibitem{andrieu_hypocoercivity_2021}
Christophe Andrieu, Alain Durmus, Nikolas N\"{u}sken, and Julien Roussel, \emph{Hypocoercivity of piecewise deterministic {M}arkov process-{M}onte {C}arlo}, Ann. Appl. Probab. \textbf{31} (2021), no.~5, 2478--2517. \MR{4332703}

\bibitem{andrieu_weak_2026}
Christophe Andrieu, Anthony Lee, Sam Power, and Andi~Q. Wang, \emph{Weak {Poincaré} inequalities for {Markov} chains: {Theory} and applications}, The Annals of Applied Probability \textbf{36} (2026), no.~1, 46--107 (en), Publisher: Institute of Mathematical Statistics.

\bibitem{bakry_analysis_2014}
D.~Bakry, Ivan Gentil, and Michel Ledoux, \emph{Analysis and {Geometry} of {Markov} {Diffusion} {Operators}}, Grundlehren der mathematischen {W}issenschaften, vol. 348., Springer, Cham, 2014 (English).

\bibitem{bobkov_weighted_2009}
Sergey~G. Bobkov and Michel Ledoux, \emph{Weighted {Poincaré}-type inequalities for {Cauchy} and other convex measures}, The Annals of Probability \textbf{37} (2009), no.~2, 403--427 (en), Publisher: Institute of Mathematical Statistics.

\bibitem{bouchard-cote_bouncy_2018}
Alexandre Bouchard-Côté, Sebastian~J. Vollmer, and Arnaud Doucet, \emph{The {Bouncy} {Particle} {Sampler}: {A} {Nonreversible} {Rejection}-{Free} {Markov} {Chain} {Monte} {Carlo} {Method}}, Journal of the American Statistical Association \textbf{113} (2018), no.~522, 855--867, Publisher: Taylor \& Francis.

\bibitem{bresar2026diffeomorphicmarkovchainmonte}
Miha Brešar and Aleksandar Mijatović, \emph{Diffeomorphic markov chain monte carlo: fast mixing for heavy-tailed distributions}, 2026, arXiv:2608.04284 [stat.CO].

\bibitem{cao_explicit_2023}
Yu~Cao, Jianfeng Lu, and Lihan Wang, \emph{On {Explicit} ${L}^2$-{Convergence} {Rate} {Estimate} for {Underdamped} {Langevin} {Dynamics}}, Archive for Rational Mechanics and Analysis \textbf{247} (2023), no.~5, 90.

\bibitem{cattiaux_functional_2010}
Patrick Cattiaux, Nathael Gozlan, Arnaud Guillin, and Cyril Roberto, \emph{Functional {Inequalities} for {Heavy} {Tailed} {Distributions} and {Application} to {Isoperimetry}}, Electronic Journal of Probability \textbf{15} (2010), 346--385.

\bibitem{cerrai_second_2001}
Sandra Cerrai, \emph{Second {Order} {PDE}’s in {Finite} and {Infinite} {Dimension}: {A} {Probabilistic} {Approach}}, Springer Berlin Heidelberg, Berlin, Heidelberg, 2001.

\bibitem{Chewi26Book}
Sinho Chewi, \emph{Log-concave sampling}, Forthcoming, 2026, Available online at \url{https://chewisinho.github.io/}.

\bibitem{davis_markov_2018}
M.~H.~A. Davis, \emph{Markov {Models} \& {Optimization}}, Routledge, New York, February 2018.

\bibitem{deligiannidis_exponential_2019}
George Deligiannidis, Alexandre Bouchard-Côté, and Arnaud Doucet, \emph{Exponential {Ergodicity} of {The} {Bouncy} {Particle} {Sampler}}, The Annals of Statistics \textbf{47} (2019), no.~3, 1268--1287, Publisher: Institute of Mathematical Statistics.

\bibitem{dolbeault_hypocoercivity_2010}
Jean Dolbeault, Clément Mouhot, and Christian Schmeiser, \emph{Hypocoercivity for linear kinetic equations conserving mass}, Transactions of the American Mathematical Society \textbf{367} (2010), 3807--3828.

\bibitem{durmus_piecewise_2021}
Alain Durmus, Arnaud Guillin, and Pierre Monmarché, \emph{Piecewise deterministic {Markov} processes and their invariant measures}, Annales de l'Institut Henri Poincaré, Probabilités et Statistiques \textbf{57} (2021), no.~3, 1442--1475 (en), Publisher: Institut Henri Poincaré.

\bibitem{dwivediLogConcavesampling}
Raaz Dwivedi, Yuansi Chen, Martin~J Wainwright, and Bin Yu, \emph{{L}og-concave sampling: {M}etropolis-{H}astings algorithms are fast!}, Proceedings of the 31st Conference On Learning Theory (Sébastien Bubeck, Vianney Perchet, and Philippe Rigollet, eds.), Proceedings of Machine Learning Research, vol.~75, PMLR, 06--09 Jul 2018, pp.~793--797.

\bibitem{eberle_convergence_2025}
Andreas Eberle, Arnaud Guillin, Leo Hahn, Francis Lörler, and Manon Michel, \emph{Convergence of non-reversible {Markov} processes via lifting and flow {Poincaré} inequality}, July 2025, arXiv:2503.04238 [math.AP].

\bibitem{eberle_non-reversible_2026}
Andreas Eberle and Francis Lörler, \emph{Non-reversible lifts of reversible diffusion processes and relaxation times}, Probability Theory and Related Fields \textbf{194} (2026), no.~1, 173--203.

\bibitem{engel_one-parameter_2000}
Klaus-Jochen Engel and Rainer Nagel (eds.), \emph{One-{Parameter} {Semigroups} for {Linear} {Evolution} {Equations}}, Springer New York, New York, NY, 2000.

\bibitem{ethier_markov_1986}
S.N. Ethier and T.G. Kurtz, \emph{Markov {Processes}}, Wiley {Series} in {Probability} and {Statistics}, Wiley, March 1986.

\bibitem{fan_sharp_2026}
Zexi Fan, Bowen Li, and Jianfeng Lu, \emph{Sharp hypocoercive convergence estimates for underdamped {Langevin} dynamics via the modified ${L}^2$ method}, April 2026, arXiv:2604.10068 [math.AP].

\bibitem{folland_real_1999}
Gerald~B. Folland, \emph{Real {Analysis}: {Modern} {Techniques} and {Their} {Applications}}, John Wiley \& Sons, Incorporated, New York, United States, 1999.

\bibitem{g_vasdekis_speed_2023}
{G. Vasdekis} and {G. O. Roberts}, \emph{Speed up {Zig}-{Zag}}, The Annals of Applied Probability \textbf{33} (2023), no.~6A, 4693--4746.

\bibitem{he2022transformedULA}
Ye~He, Krishnakumar Balasubramanian, and Murat~A. Erdogdu, \emph{An analysis of transformed unadjusted langevin algorithm for heavy-tailed sampling}, IEEE Transactions on Information Theory \textbf{70} (2024), no.~1, 571--593.

\bibitem{HillTailestimation}
Bruce~M. Hill, \emph{A simple general approach to inference about the tail of a distribution}, The Annals of Statistics \textbf{3} (1975), no.~5, 1163--1174.

\bibitem{hormander_analysis_2003}
Lars Hörmander, \emph{The {Analysis} of {Linear} {Partial} {Differential} {Operators} {I}: {Distribution} {Theory} and {Fourier} {Analysis}}, Springer Berlin Heidelberg, Berlin, Heidelberg, 2003.

\bibitem{pmlrTailAdaptiveLaszkiewicz22a}
Mike Laszkiewicz, Johannes Lederer, and Asja Fischer, \emph{Marginal tail-adaptive normalizing flows}, Proceedings of the 39th International Conference on Machine Learning (Kamalika Chaudhuri, Stefanie Jegelka, Le~Song, Csaba Szepesvari, Gang Niu, and Sivan Sabato, eds.), Proceedings of Machine Learning Research, vol. 162, PMLR, 17--23 Jul 2022, pp.~12020--12048.

\bibitem{leif_t_johnson_variable_2012}
{Leif T. Johnson} and {Charles J. Geyer}, \emph{Variable transformation to obtain geometric ergodicity in the random-walk {Metropolis} algorithm}, The Annals of Statistics \textbf{40} (2012), no.~6, 3050--3076.

\bibitem{livingstone_geometric_2021}
Samuel Livingstone, \emph{Geometric {Ergodicity} of the {Random} {Walk} {Metropolis} with {Position}-{Dependent} {Proposal} {Covariance}}, Mathematics \textbf{9} (2021), no.~4, 341.

\bibitem{lu2026windowedthinningquerycomplexity}
Jianfeng Lu and Yinchen Luo, \emph{Windowed thinning and query complexity for the bouncy particle and zigzag samplers}, 2026, arXiv:2607.28413 [math.NA].

\bibitem{piecewise_hypocoercivity_wang_2022}
Jianfeng Lu and Lihan Wang, \emph{{On} {Explicit} ${L}^2$-{Convergence} rate estimate for {Piecewise} {Deterministic} {Markov} {Processes} in {MCMC} {Algorithms}}, The Annals of Applied Probability \textbf{32} (2022), no.~2, 1333--1361.

\bibitem{YouTube_talk}
Aleksandar Mijatovi\'c and Vinayak Niraj, \emph{Short {Y}ou{T}ube videos on ``hypocoercivity of tempered bouncy particle sampler for heavy-tailed targets''}, \href{https://youtu.be/HhlTlsyhhzY?si=S9ZTErX9NkR8Ctwj}{\underline{Part~I} (Main results)} and \href{https://youtu.be/tX93nqGz1WU?si=as1Pfu75ruAXv4A5}{\underline{Part~II} (Proofs)} on the YouTube channel \href{https://www.youtube.com/@prob-am7844}{\underline{Prob-AM}}, 2026.

\bibitem{monmarché2026entropicconvergencepiecewisedeterministic}
Pierre Monmarché and Lihan Wang, \emph{On the entropic convergence for piecewise deterministic samplers: speedup and obstruction}, 2026, arXiv:2606.26086 [math.PR].

\bibitem{MR311856}
Benjamin Muckenhoupt, \emph{Hardy's inequality with weights}, Studia Math. \textbf{44} (1972), 31--38. \MR{311856}

\bibitem{pavliotis_stochastic_2014}
Grigorios~A. Pavliotis, \emph{Stochastic {Processes} and {Applications}: {Diffusion} {Processes}, the {Fokker}-{Planck} and {Langevin} {Equations}}, Springer New York, New York, NY, 2014.

\bibitem{powervaskedis2025robustness}
Sam Power and Giorgos Vasdekis, \emph{Some aspects of robustness in modern markov chain monte carlo}, 2025, arXiv:2511.21563 [stat.CO].

\bibitem{protter_stochastic_2004}
Philip~E. Protter, \emph{Stochastic integration and differential equations}, second ed., Stochastic {M}odelling and {A}pplied {P}robability, vol. 21., Springer, New York;Berlin;, 2004 (English).

\bibitem{roberts_polynomial_2023}
Gareth~O. Roberts and Jeffrey~S. Rosenthal, \emph{Polynomial convergence rates of piecewise deterministic {M}arkov processes}, Methodology and Computing in Applied Probability \textbf{25} (2023), no.~1, 1--18.

\end{thebibliography}

\end{document}